\documentclass[a4paper,12pt]{amsart}

\usepackage{geometry}
\usepackage[comma, numbers]{natbib}
\usepackage[inline]{enumitem}
\usepackage{graphicx}
\usepackage[textfont=it,tableposition=above]{caption}		
\usepackage{subfig}					

\usepackage{framed}					
\usepackage[foot]{amsaddr}	
\usepackage{url}						

\usepackage{amssymb}				
\usepackage{mathtools}			

\usepackage[icon=Note, draft]{pdfcomment}
\hypersetup{hidelinks,breaklinks}

\usepackage[norefs,nocites]{refcheck}				

\newtheorem{theorem}{Theorem}

\newtheorem{proposition}[theorem]{Proposition}

\theoremstyle{definition}
\newtheorem*{definition}{Definition}
\newtheorem{example}{Example}
\newtheorem{problem}{Problem}

\newcommand{\R}{\mathbb R}																			
\newcommand{\N}{\mathbb N}																			
\newcommand{\M}{\mathcal M}																			
\newcommand{\Prob}[1][\R^d]{\mathcal P\left({#1}\right)}				
\newcommand{\Fsets}{\mathcal F}																	
\newcommand{\PP}{\mathsf P}																			
\newcommand{\as}{\mathrm {a.s.}}																
\newcommand{\HD}{hD}																						
\newcommand{\SVD}{svD}																					
\newcommand{\MahD}{MD}																					
\newcommand{\dd}{\,\mathrm{d}\,}																
\newcommand{\E}{\operatorname{E}}																
\newcommand{\Var}{\operatorname{Var}}														
\newcommand{\CR}[1][\delta]{P_{#1}}															
\newcommand{\CRFB}[1][\delta]{P_{#1}^{FB}}											
\newcommand{\vol}[2][d]{\operatorname{vol}_{#1}\left(#2\right)}	
\newcommand{\asa}[2][]{\operatorname{as}_{#1}\left(#2\right)}		
\newcommand{\Asa}[2][]{\operatorname{As}_{#1}\left(#2\right)}		
\newcommand{\Sph}[1][d-1]{\mathbb S^{#1}}												
\newcommand{\B}[1][d]{B^{#1}}																		
\newcommand{\dH}{\operatorname{d}_H}														
\newcommand{\dS}{\operatorname{d}_S}														
\newcommand{\tr}{^\mathsf{T}}																		
\newcommand{\MS}{\mathfrak S}																		
\newcommand{\CB}[1][d]{\mathcal K^{#1}}													
\newcommand{\MD}{\Pi}																						
\newcommand{\AHD}{AhD}																					
\DeclareMathOperator{\rad}{rad}																	
\DeclareMathOperator{\Supp}{Supp}																
\newcommand{\dist}[1][\Sigma]{\operatorname{d}_{#1}}						
\newcommand{\iu}{\mathrm{i}\mkern1mu}														
\newcommand{\eqd}{\stackrel{\mathclap{\mbox{\tiny{\emph{d}}}}}{=}}	
\newcommand{\Euler}{e}																					
\newcommand{\AIP}{\mathcal P^{d}}																
\newcommand{\FD}{\mathcal F^d}																	

\author{Stanislav Nagy$^{1}$}
\author{Carsten Sch\"utt$^2$}
\author{Elisabeth Werner$^3$}

\email{nagy@karlin.mff.cuni.cz}
\email{schuett@math.uni-kiel.de}
\email{elisabeth.werner@case.edu}

\address{$^1$
	Charles University, Prague,
	Department of Probability and Math. Statistics,
	Czech Republic
}
\address{$^2$
	Universit\"at Kiel,
	Mathematisches Seminar,
	Germany
}
\address{$^3$
	Case Western Reserve University, Cleveland,
	Department of Mathematics, Applied Mathematics and Statistics,
	United States
}

\title{Data depth and floating body}
\date{\today}


\begin{document}

\begin{abstract}
Little known relations of the renown concept of the halfspace depth for multivariate data with notions from convex and affine geometry are discussed. Halfspace depth may be regarded as a measure of symmetry for random vectors. As such, the depth stands as a generalization of a measure of symmetry for convex sets, well studied in geometry. Under a mild assumption, the upper level sets of the halfspace depth coincide with the convex floating bodies used in the definition of the affine surface area for convex bodies in Euclidean spaces. These connections enable us to partially resolve some persistent open problems regarding theoretical properties of the depth. 
\end{abstract}

\maketitle

\section{Introduction} 

Halfspace depth and floating body are the same concept. The first is extensively studied in nonparametric statistics, the second is of great importance in convex geometry. Until recently, work on data depth has not been recognized by the convex geometry community, and that in convex geometry not by researchers in statistics. Of course, the motivation and the goals in both fields are different, and even their philosophies are not the same. Nonetheless, there is an abundance of results common to both fields. We want to explore and summarize here what is common to both fields, what is known and what is not known.

In nonparametric statistics, data depth is a generalization of order statistics and ranks to multivariate random variables. Its aim is, for a multivariate probability distribution, to devise a distribution-specific ranking of points in the sample space. In other words, depth is a function intended to distinguish points that fit the overall pattern of the distribution, from measurement errors and other outliers.

In convex geometry, the concept of floating body was used, among other things, to introduce the affine surface area to all convex bodies. The associated affine isoperimetric inequality is much stronger than the classical isoperimetric inequality. It provides solutions to many problems where ellipsoids are extrema.

\section{Motivation and background}	

\begin{figure}[htpb]
\includegraphics[width=.75\textwidth]{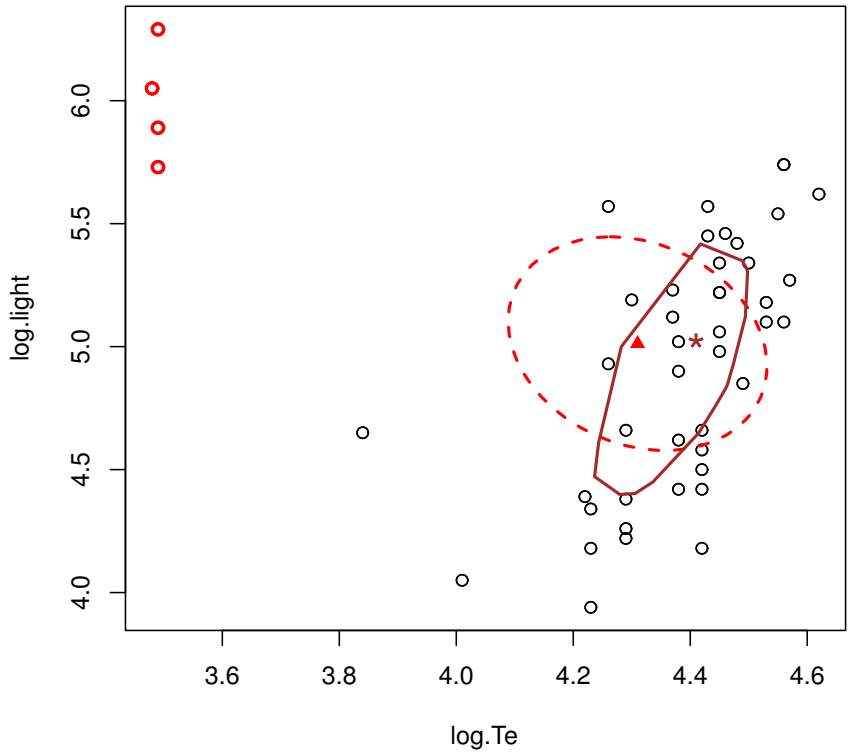} 
\caption{The Hertzsprung-Russell diagram of the Star Cluster CYG~OB1. The four giant stars (red points) attract both the sample mean (red triangle), and the sample covariance (represented by the red ellipse). The halfspace depth-based median (brown star), and the depth-based central region containing $25~\%$ of the observations, provide a more appropriate representation of the location and the variability of the main data cloud.}
\label{figure:outliers}
\end{figure}

In classical statistics of univariate data it is well known that the ordering of data points and the corresponding rank statistics constitute powerful statistical tools, valid under very broad sets of assumptions. The median, for instance, is a rather efficient, robust, affine equivariant location estimator. Quantiles are invaluable in both visualization and inference. Rank tests provide versatile analogues to the traditional testing procedures, and unlike many standard parametric statistical tests, work under minimal assumptions imposed on the data.

In multivariate spaces, however, no natural ordering exists. For $d>1$, we are not able to rank points $x$ and $y$ in $\R^d$ according to their magnitude, or tell whether $x$ lies ``to the left" of $y$. Though, for a given dataset, one can still ask how well a point fits into the overall pattern of the observations. If the data concentrate around a focal point, and follow a simple scatter structure, we can say that a point $x$ inside the data cloud is ``deeper'' inside the mass of the data, than a point $y$ that lies on the outskirts, or outside the data cloud. This notion of multivariate center-outwards ranking is formalized by the idea of data depth. In general, a depth $D(x;P)$ is a function that, given a probability distribution $P$ on $\R^d$ (or a random sample from this distribution), quantifies the centrality (the depth) of a point $x$ with respect to (w.r.t.) the geometry of $P$. The more $x$ is inside the main bulk of the mass of $P$, the higher the depth $D(x;P)$. As such, the depth enables us to rank the points of $\R^d$ according to their centrality w.r.t. $P$, and devise the corresponding depth-rank statistics, or depth-based quantile regions.

Let us illustrate our point by giving two simple examples where the depth plays an instrumental role. In the first one, our search for sensible data ordering is motivated by the problem to define a multivariate analogue of the median. In the classification task presented afterwards, we stress the importance of general global ranking procedures in data science.

Consider the dataset of $47$ bivariate observations taken from \citep{Rousseeuw_Leroy1987}, displayed in Figure~\ref{figure:outliers}. The data correspond to the Hertzsprung-Russell diagram of the stars in the Star Cluster CYG~OB1 in the Cygnus constellation. In the scatterplot, the logarithm of the effective temperature at the surface of the star (\texttt{log.Te}) is plotted, against the logarithm of its light intensity (\texttt{log.light}). The majority of the observations follows a common pattern --- their data points concentrate in the south-east part of the plot, and appear to be scattered rather regularly. Four stars clearly do not follow that pattern, and could be considered as outliers (the red points in Figure~\ref{figure:outliers}). Those are known to be stars of different characteristics (so-called giant stars). Let us determine the location of the random sample. The sample mean (red triangle) is attracted towards the outlying observations, and does not represent the location of the majority of the data appropriately. That is, of course, caused by the fact that the expectation is known to be affected severely by erroneous data, and outliers, i.e., it is not robust. For univariate data, one can opt for the median in such situations. But, what is a median of a multivariate dataset? Intuitively, the median should capture the location of the majority of observations, and should be little affected by errors, or other anomalies in the data. The median should be a point ``deep" inside the data cloud. With the notion of the halfspace depth (the precise definition is given below), we consider the depth median being the point whose depth w.r.t. the random sample is the highest (brown star in Figure~\ref{figure:outliers}). The depth median is robust, i.e. it is much less affected by the four giant stars than the mean. It captures the center of the main bulk of data much better then the sample mean. Additionally, let us consider the Mahalanobis ellipse (for precise definition see \eqref{Mahalanobis level set} below) that corresponds to the sample mean, and the sample covariance of our dataset, and contains $25~\%$ of the data points (red ellipse in Figure~\ref{figure:outliers}). This ellipse is intended to represent the scatter pattern of the data. As seen in Figure~\ref{figure:outliers}, it is also heavily biased towards the anomalous observations. On the other hand, the halfspace depth region that contains (roughly) $25~\%$ of the deepest points (brown polygon), still represents the main modes of variation of the data quite reliably.

\begin{figure}[htpb]
\includegraphics[width=.45\textwidth]{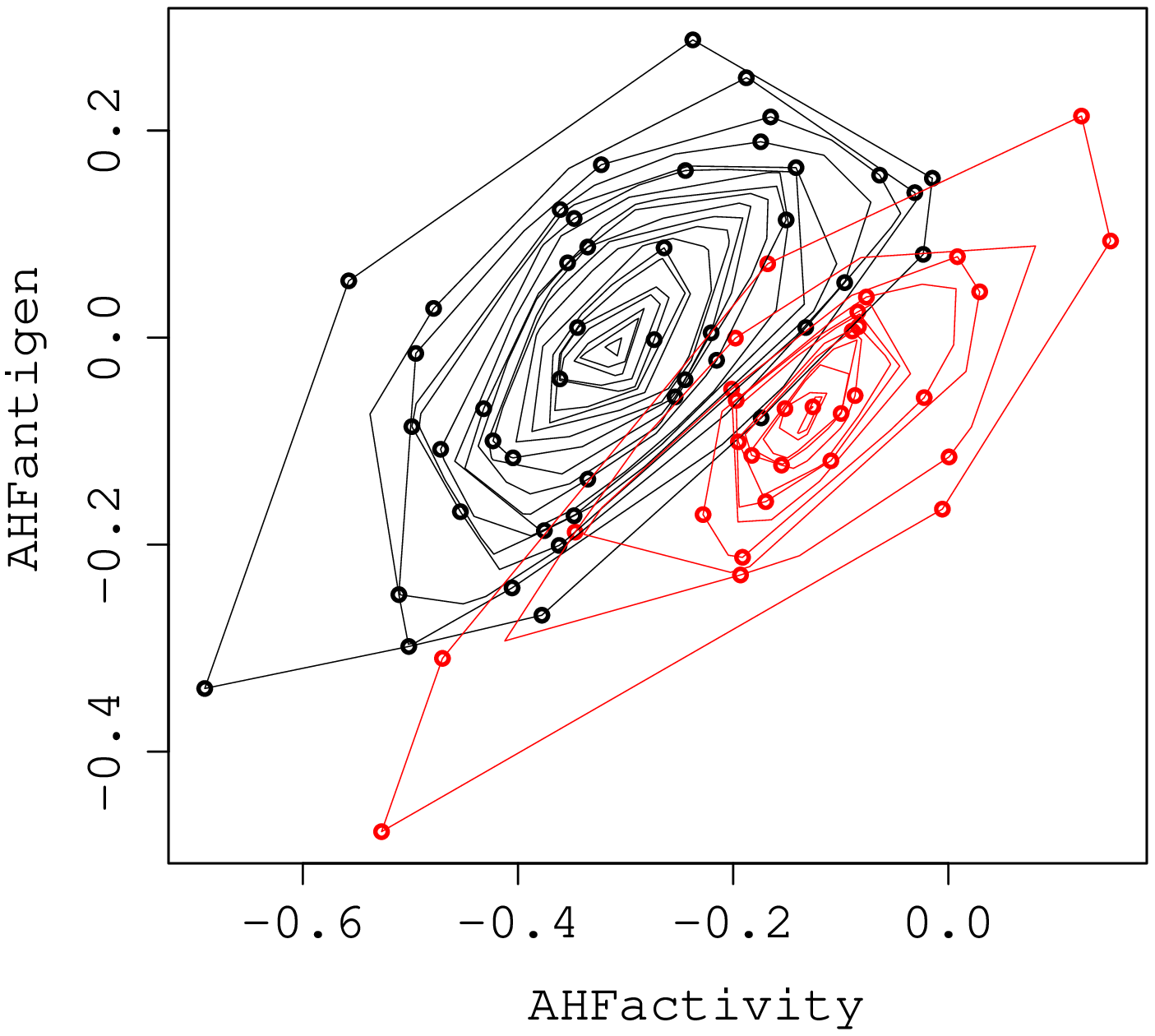} \quad \includegraphics[width=.45\textwidth]{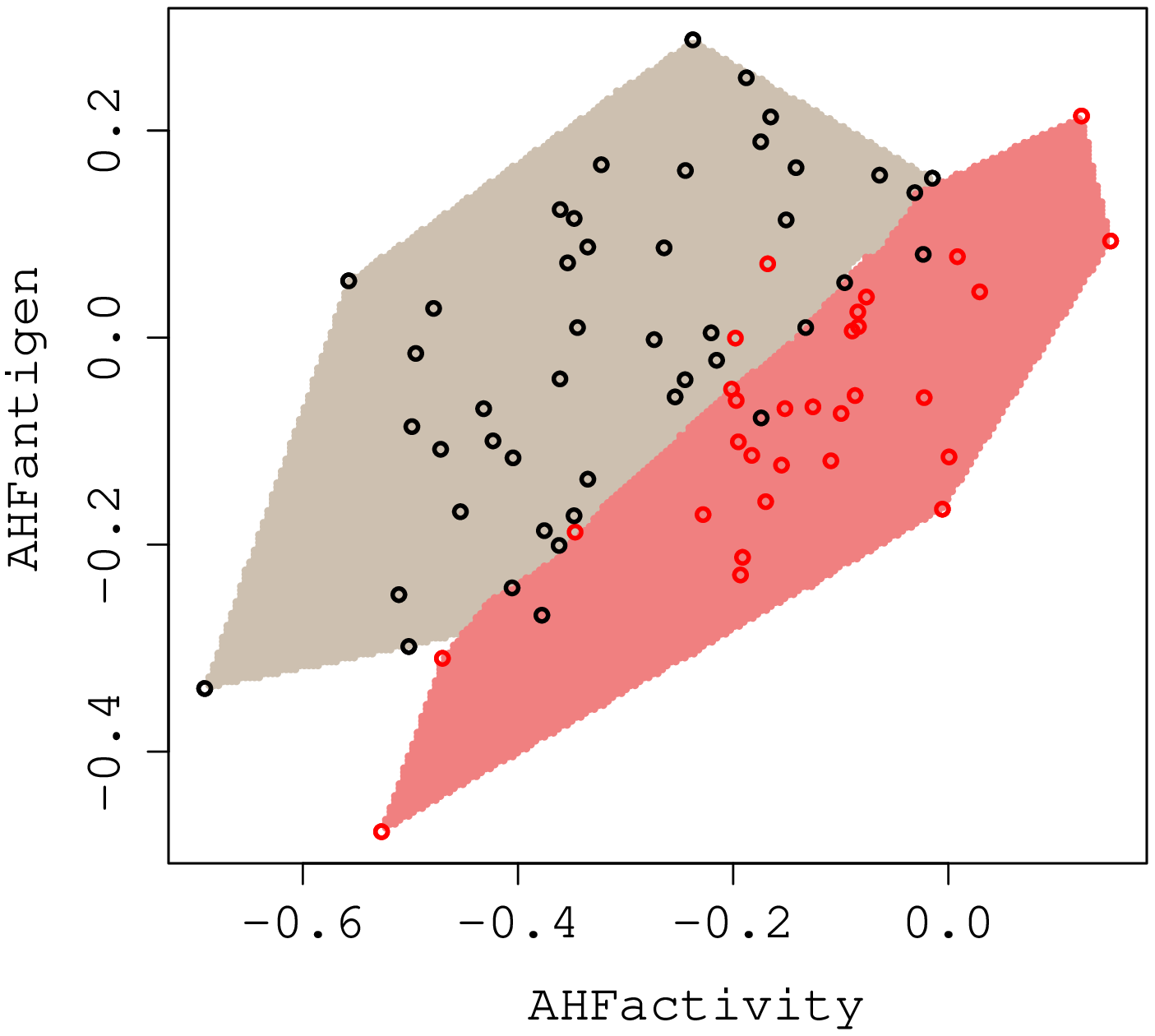}
\caption{Hemophilia data. The group of carriers of Hemophilia A (black points) is to be separated from the group of non-carriers (red points). Left panel: halfspace depth contours of the two groups; right panel: regions where one of the depths dominates the other --- a new observation inside the light-gray polygon is classified into the black group, an observation inside the light-red polygon into the red group. Observations outside the polygons remain unclassified.}
\label{figure:classification}
\end{figure}

For our second motivating example, the \texttt{hemophilia} data (available in \citep{ddalpha}) are visualized in Figure~\ref{figure:classification}. The dataset consists of bivariate measurements (AHF activity and AHF antigen) taken from blood samples of $75$ women, out of whom $45$ are known to be hemophilia A carriers (black dots on Figure~\ref{figure:classification}). Our task here is classification --- given a new datum with the two measured characteristics, decide whether the new patient is a potential hemophilia A carrier. The literature where problems of this type are studied in statistics is immense. One approach to this problem is to make use of the depth, and the ranking of the observations. Firstly, compute the depth of the point w.r.t. both random samples, i.e. rank the new data point inside the group of carriers, and the group of non-carriers, respectively. Then, assign the new datum to that group for which it is more typical, reflected by its higher depth-based rank. Figure~\ref{figure:classification} shows the contours of the halfspace depth functions for both random samples (left panel), and the regions where one of the depths is larger than the other (the polygons on the right hand side). A new observation within the light-gray polygon would be assigned into the black cloud of points (carriers); a point inside the light-red polygon is assessed to come from a non-carrier patient\footnote{Note that there are points that remain unclassified. For instance, points outside both convex hulls of the sample points are not assigned to any of the clusters using our simple classification rule. More advanced depth-based techniques dealing with these problems are available in the literature.}. This approach, sometimes called the maximum depth classification and its other variants based on the depth, turned out be particularly appealing in the past years. Mainly due to their conceptual simplicity, versatility, and good robustness properties, depth-based classification rules have gained great importance over the past decade.

As we saw, data depth introduces ranking and ordering also for multivariate datasets. Other applications of the depth include multivariate extensions of the rank tests, L-statistics (linear combinations of order statistics), and many other nonparametric and robust procedures.

Let us now provide a rigorous definition of the halfspace depth. For $d\geq 1$, a point $x$ in $d$-dimensional Euclidean space $\R^d$, and a probability distribution $P$ on $\R^d$, the halfspace depth $\HD$ of $x$ with respect to $P$ is given by
	\begin{equation*}	
	\HD(x;P) = \inf\left\{ P(H^{-}) \colon \mbox{ $H$ is a hyperplane with $x\in H$} \right\},	
	\end{equation*}
where $H^{-}$ denotes one of the closed halfspaces associated with its boundary hyperplane $H$ in $\R^d$. In other words, the depth $\HD$ is given as the smallest probability of a closed halfspace that contains $x$. Points outside the convex hull of the support of $P$ have zero depth. More generally, any point $x$ with $\HD(x;P) < \delta$ can be separated from the main mass of $P$ by a hyperplane cutting away both the point $x$, and a mass of probability at most $\delta$. Points with rather high depth values can be seen as those lying at the center of the distribution, as no halfspace of small probability can separate them from the rest of $P$. This way, the depth $\HD$ acts as a mapping that orders the sample points in a center-outwards direction, with the ordering given subject to the distribution $P$. 

One plausible statistical application of the depth is the possibility to introduce quantiles to multivariate data. Consider, for observations on the real line $\R$, the central quantile regions given as the intervals bounded by the $\alpha$-, and $(1-\alpha)$-quantiles of $P$, for $\alpha \in (0,1/2]$. A natural multivariate analogue of these sets are then the loci of points whose depth exceeds given thresholds. Such sets of points are called the \emph{central regions} of $P$ in $\R^d$ (given by the depth $\HD$). As discussed in Section~\ref{section:depth} below, the collection of all central regions of $P$ consists of affine equivariant nested closed convex sets, monotone in the sense of set inclusion, see also the left panel of Figure~\ref{figure:classification}, or Figure~\ref{figure:floating bodies} below. For sufficiently regular distributions, the smallest non-empty set in this collection is a singleton --- the most central point of $\R^d$ for $P$. This point is frequently recognized as a generalization of the median to $\R^d$-valued data.

The earliest contribution in statistics that deals with some form of the halfspace depth is believed to be \citep{Hodges1955} from 1955. There, a sign test for bivariate data is proposed and examined. Its test statistic takes the form of the depth $\HD$ at a single, given point $x$. 

The seminal paper that introduced the depth in the sample case (that is, for datasets) is \citet{Tukey1975} from 1975. In that paper, the depth is proposed as a tool that enables efficient visualization of random samples in $\R^2$. It is in \citep{Tukey1975}, where the word \emph{depth} is used for the first time. The original formal definition of the halfspace depth for multivariate data can be found in \citet{Donoho1982} and \citet{Donoho_Gasko1992} (see also \citep{Small1987}). 

Starting with the study of Donoho, much research has focused on data depth and related concepts. The prominent, loosely related simplicial depth for multivariate data was defined by \citet{Liu1988, Liu1990}, building upon the ideas presented in \citet{Oja1983}. Soon, the idea of depth was extended to data in non-linear spaces \citep{Small1987,Liu_Singh1992}, general metric spaces \citep{Carrizosa1996}, observations on graphs \citep{Small1997}, regression \citep{Rousseeuw_Hubert1999}, or data taking values in functional spaces \citep{Fraiman_Muniz2001}, and Banach spaces \citep{Cuevas_Fraiman2009}.

The general concept of data depth in $\R^d$ was formalized by \citet{Zuo_Serfling2000}, \citet{Dyckerhoff2004}, and \citet{Serfling2006}, see also \citet{Mosler2013}. Nowadays, dozens of depth functions and related methods for all types of data can be found in the literature. It is, however, the halfspace depth $\HD$, that is the single most important depth that continues to reappear as the prime representative of this idea.  

Apart from statistics, halfspace depth gained considerable attention also in discrete and computational geometry (see \citep{Liu_etal2006}). There, the combinatorial nature of the sample version of $\HD$ provides a rich source of interesting problems, especially in connection with its computational aspects. For instance, the halfspace medians, i.e. points at which $\HD$ is maximized over $\R^d$, are closely related to the notion of \emph{centerpoints} studied in discrete geometry (see \citep[Section~1.4]{Matousek2002}). For a recent overview of data depth and its links to computational geometry see \citet{Rousseeuw_Hubert2018}.

A notable article on the properties of the halfspace depth for general (probability) measures is \citet{Rousseeuw_Ruts1999}. There, the population version of the depth $\HD$ (i.e. the depth w.r.t. the true sampling distribution $P$) is investigated. Several interesting links between the concept of the halfspace depth, and some sources outside mathematical statistics are outlined. In \citet[Section~8]{Rousseeuw_Ruts1999} it is noted that the halfspace depth relates with the voting problem studied in \citet{Caplin_Nalebuff1988, Caplin_Nalebuff1991} in the theory of social choice. Further, it is also observed that some results concerning the maximal depth of a point in $\R^d$ can be found already in \citet{Neumann1945, Rado1946}, and \citet{Grunbaum1960}, in the literature concerned with the geometric properties of functions and sets. 

In the present paper, we pursue this line of research, and point to the remarkable similarity of the notion of halfspace depth, and some concepts used in other fields of mathematics, especially in convex and affine geometry. 

The paper is organized as follows. In Section~\ref{section:depth} we introduce the notation, and give a brief overview of some of the most important properties of the halfspace depth. In Section~\ref{section:Winternitz} we follow the lead provided by \citet{Rousseeuw_Ruts1999}, and trace a little known early precursor of the halfspace depth to be the so-called Winternitz measure of symmetry of convex bodies, a functional that dates back at least to \citet{Blaschke1923}. In Sections~\ref{section:convex floating bodies} and~\ref{section:floating bodies for measures} we examine relations of the halfspace depth with the (convex) floating bodies, an important tool used in the study of convex sets in $\R^d$. As demonstrated, the history of the halfspace depth is much longer than assumed: the earliest predecessors of the depth $\HD$ appear to be the floating bodies in $\R^d$, studied already by \citet{Dupin1822} in 1822. Later, floating bodies reappear in mathematics in 1923 in \citet{Blaschke1923}, in connection with an affine invariant, the affine surface area of convex bodies, and other problems. As discussed in Section~\ref{section:convex floating bodies}, the modern notion of the floating body, the convex floating body, defined independently by \citet{Barany_Larman1988} and \citet{Schutt_Werner1990}, plays a major role for the concept of affine surface area studied in geometry. We present extensions of this notion to log-concave measures and show its importance in questions of approximation of convex bodies by polytopes. It is also discussed in Section~\ref{section:floating bodies for measures} that the convex floating body corresponds to the upper level sets of the halfspace depth. Using this identity, we provide in Section~\ref{section:ellipsoids} a surprising bound of the halfspace depth in terms of the Mahalanobis depth. Section~\ref{section:characterization} is devoted to the distribution-by-depth characterization problem, concerned with finding conditions under which no two probability measures can have the same depth over $\R^d$. It is shown that two important partial positive results to this problem \citep{Hassairi_Regaieg2008, Kong_Zuo2010} are both special cases of a more general theorem, conveniently stated in terms of floating bodies of measures. As a corollary, we obtain some new classes of distributions characterized by their depth. Finally, in Section~\ref{section:extensions} we discuss some extensions of the depth to more exotic data. The survey is completed with a series of open problems relevant to the topics of halfspace depth and floating bodies.

\section{Data depth: Notation and essential properties}	\label{section:depth}

Let $\left( \Omega, \Fsets, \PP \right)$ be the probability space on which all random variables are defined. For a measurable space $\M$, denote by $\Prob[\M]$ the space of all probability distributions on $\M$, and write $X\sim P$ for a random variable $X$ with distribution $P\in\Prob[\M]$. The support of $P$ is denoted by $\Supp(P) \subseteq \M$. For $n \in \N = \{1,2,\dots\}$ and a random sample $X_1, X_2, \dots, X_n$ from $P$, let $P_n \in \Prob[\M]$ be the associated empirical measure, i.e. the uniform measure supported in the sample points. For $X \sim P$ and $\phi \colon \M \to \M$ measurable, write $P_{\phi(X)} \in \Prob[\M]$ for the probability distribution of the transformed random variable $\phi(X)$. This way, $P \equiv P_X$. 

The space $\R^d$ is equipped with the Euclidean norm $\left\Vert \cdot \right\Vert$ and the inner product $\left\langle \cdot, \cdot \right\rangle$.  For $x \in \R^d$ and $r>0$, 
	\[	\B(x,r) = \left\{ y \in \R^d \colon \left\Vert y - x \right\Vert \leq r \right\}	\]
is the closed Euclidean ball centered at $x$ with radius $r$. $\B$ stands for the unit ball $\B(0,1)$ and  $\Sph = \partial \B$  denotes the unit sphere. $\partial K$ stands for the topological boundary of $K \subset \R^d$. The Lebesgue measure of a measurable set $K$ will be denoted also by $\vol{K}$. 

A convex body is a convex, compact subset of $\R^d$ with non-empty interior. For $k \in\N$ it is said to have a $C^k$ boundary if its boundary, locally parametrized as a function from $\R^{d-1}$, is $k$-times continuously differentiable. We denote the collection of all convex bodies in $\R^d$ by $\CB$. For an interior point $x_{0}$ of a convex body $K$ define the polar body $K^{x_{0}}$ of $K$ w.r.t. $x_{0}$ by
	\begin{equation}\label{polar}	
		K^{x_{0}}=\left\{y\in\R^d \colon \left\langle y,x-x_{0}\right\rangle\leq 1 \mbox{ for all } x\in K\right\}.	
	\end{equation}		
If  $0 \in  \operatorname{Int} (K)$, the interior of $K$, we write $K^\circ$ for the polar body \eqref{polar} of $K$ w.r.t. $0$. A star body $K$ is a compact subset of $\R^d$ with the property that there exists $x \in K$ such that the open line segment from $x$ to any point $y \in K$ is contained in the interior of $K$. The Minkowski addition of $K, L \in \CB$ is $K + L = \left\{ x + y \colon x \in K, y \in L \right\}$. Likewise, for $\lambda \in \R$,  $\lambda K = \left\{ \lambda x \colon x \in K \right\}$.

We write 
	\[	H_{u,\alpha} = \left\{ z \in \R^d \colon \langle z, u \rangle = \alpha \right\}	\quad \mbox{for }u \in \R^d \setminus \{0\}, \alpha \in \R,	\]
for a hyperplane in $\R^d$, and denote for $H = H_{u,\alpha}$ the two halfspaces bounded by this hyperplane by $H^- = H_{u,\alpha}^-$ and $H_+ = H_{u,\alpha}^+$, respectively.
By $\mathcal H$ we denote the set of all hyperplanes, and by $\mathcal H^-$ the set of all closed halfspaces in $\R^d$. 

We say that the halfspace $H^+$ supports the set $K \subset \R^d$ if $K \cap H^- \ne \emptyset$ and $K \subseteq H^+$. A hyperplane $H\in\mathcal H$ is said to support $K$ if either $H^-$ or $H^+$ supports $K$. The collection of supporting halfspaces of an empty set is empty. Recall that the boundary of a convex body $K \in \CB$ is $C^1$ if and only if for any $x \in \partial K$ there exists a unique hyperplane $H \in \mathcal H$ that supports $K$ with $x \in H$ \citep[Theorem~2.2.4]{Schneider2014}. For a convex set $K \subset \R^d$, the support function of $K$ is defined as
	\begin{equation}	\label{support function}
	h_K \colon \Sph \to \R \colon u \mapsto \sup\left\{ \langle x, u \rangle \colon x \in K \right\}.
	\end{equation}
The centroid (or the barycenter) of a compact set $K \subset \R^d$ is the expectation of a random variable distributed uniformly on $K$.

\subsection{Statistical depth for multivariate data}

The first formal definition of the halfspace depth in $\R^d$ for general probability distributions can be found in \citet{Donoho1982}.

\begin{definition}
Let $P\in\Prob$ and $x\in\R^d$. The \emph{halfspace depth} (or \emph{Tukey depth}) of $x$ w.r.t. $P$ is defined as
	\begin{equation}	\label{halfspace depth}
	\HD\left( x; P \right) = \inf \left\{ P(H^-) \colon H \in \mathcal H, x \in H^- \right\}.	
	\end{equation}
\end{definition}

In \eqref{halfspace depth}, the infimum can be equally well taken only over those $H\in\mathcal H$ with $x \in H$ \citep[Proposition~3]{Rousseeuw_Ruts1999}. 

In the study of the theoretical properties of $\HD$, two regularity conditions imposed on $P\in\Prob$ frequently play an important role. The first, a smoothness condition, appears in \citet{Dumbgen1992} and \citet{Mizera_Volauf2002}, and reads 
	\begin{equation}	\label{Delta}	
	P(H) = 0 \mbox{ for all }H \in \mathcal H. 
	\end{equation} 
It is trivially satisfied if, for instance, $P$ has a density in $\R^d$. 

The second requirement concerns the support of $P$. We say that $P\in\Prob$ has \emph{contiguous support} \citep{Mizera_Volauf2002, Kong_Zuo2010} if there are no two disjoint halfspaces $H_1^-, H_2^- \in \mathcal H^-$ such that $P(H_1^-)>0$ and $P(H_2^-)>0$, but $P(H_1^-) + P(H_2^-) = 1$. In other words, the support of $P$ cannot be separated by a slab between two closed parallel hyperplanes. 

In Section~\ref{section:ellipsoids} we demonstrate a surprising relation between the halfspace depth, and another renown depth function that can be found in the literature: the Mahalanobis depth. To this end, let us briefly recall its definition, and some elementary properties. 

For any symmetric positive definite matrix $\Sigma\in\R^{d\times d}$, the Mahalanobis distance \citep{Mahalanobis1936} of two points $x, y \in \R^d$ is defined as 
	\begin{equation}	\label{Mahalanobis distance}
	\dist(x,y) = \sqrt{\left(x - y\right)\tr \Sigma^{-1} \left(x - y\right)}.	
	\end{equation}
It is a metric on $\R^d$. Based on this distance, \citet{Liu1992} proposed the following depth function.

\begin{definition}
Let $X \sim P \in \Prob$ be such that $\E X = \mu$ and $\Var X = \Sigma$ is positive definite. The \emph{Mahalanobis depth} of $x$ w.r.t. $P$ is defined as
	\[	\MahD(x;P) = \left( 1 + \dist(x,\mu) \right)^{-1}.	\]
\end{definition}

The Mahalanobis depth w.r.t. $P$ takes the maximal value $1$ at the expectation of $P$. Its upper level sets are concentric ellipsoids given, for $\delta\in(0,1]$, by
	\begin{equation}	\label{Mahalanobis level set}
	\left\{ x \in \R^d \colon \MahD(x;P) \geq \delta \right\}	= \left\{ x \in \R^d \colon \sqrt{\left(x - \mu\right)\tr \Sigma^{-1} \left(x - \mu\right)} \leq \frac{1-\delta}{\delta} \right\}.	
	\end{equation}
These ellipsoids are also called the Mahalanobis ellipsoids of the distribution $P$. Note that unlike the halfspace depth $\HD$, the Mahalanobis depth $\MahD$ is not defined for all $P\in\Prob$, but rather it is restricted to distributions with finite second moments, and positive definite variance matrices.

\subsection{Properties of the halfspace depth}

In this section we collect some basic properties of the halfspace depth \eqref{halfspace depth}, and of its upper level sets that will prove to be useful in the sequel. 

For any $P\in\Prob$, consider the upper level sets of $\HD$
	\begin{equation}	\label{central region}
	\CR = \left\{ y \in \R^d \colon \HD(y;P) \geq \delta	\right\} \quad\mbox{for }\delta\in[0,1]. 	
	\end{equation}
Immediately from the definition we see that the collection of sets $\CR$, $\delta\in[0,1]$ is nested, decreasing in the sense of set inclusion, and $\CR[0]=\R^d$. The set $\CR$ is also called the central region of $P$ corresponding to $\delta\in[0,1]$.

\begin{example}	\label{example:depth level sets of convex bodies}
Let $X \sim P\in\Prob$ be the uniform distribution on the unit ball $\B$. The marginal distribution function of the first coordinate of $X$ is given by
	\[	F_1(s) = 	\begin{cases}
								\frac{\Gamma\left((d+2)/2\right)}{\Gamma\left((d+1)/2\right)\sqrt{\pi}} \int_{-1}^s \left( 1 - t^2 \right)^{(d-1)/2} \dd t & \mbox{for }s \in (-1,1), \\
								0 & \mbox{for }s\leq -1, \\
								1 & \mbox{for }s\geq 1.
								\end{cases}	\]
It is not difficult to see that 
	\[	\HD(x;P) = F_1\left( - \left\Vert x \right\Vert \right) \quad \mbox{for all }x \in \R^d,	\]
i.e. the central region $\CR$ of $P$ is a ball with radius $-F_1^{-1}\left( \delta \right)$ for $\delta \in [0,1/2]$ and $F_1^{-1}$ the quantile function corresponding to $F_1$. For $d=2$ we obtain
		\[	\HD(x;P) = 	\begin{cases}
										\frac{1}{2}-\frac{1}{\pi}\left(\arcsin(\left\Vert x \right\Vert) + \left\Vert x \right\Vert \sqrt{1-\left\Vert x \right\Vert^2} \right)	& \mbox{if }\left\Vert x \right\Vert \leq 1, \\
										0 & \mbox{otherwise},	
										\end{cases}	\]
which agrees with \citet[Section~5.6]{Rousseeuw_Ruts1999}. Uniform distributions on balls are a special case of spherically (and elliptically) symmetric distributions. Such distributions will be treated in Example~\ref{example:alpha symmetric distributions} below.
	
For the uniform distribution $P\in\Prob[\R^2]$ on the unit square $[0,1]^2$, the halfspace depth and its central regions were computed by \citet[Section~5.4]{Rousseeuw_Ruts1999}
	\[	\HD(x;P) = 	\begin{cases}
									2 \min\left\{x, 1-x \right\} \min\left\{ y, 1-y \right\} & \mbox{for }x,y \in[0,1],	\\
									0 & \mbox{otherwise}.
									\end{cases}	\]
The expression for the halfspace depth of $P\in\Prob[\R^2]$ distributed uniformly on the equilateral triangle can be found in \citet[Section~5.3]{Rousseeuw_Ruts1999}. Several central regions \eqref{central region} of the halfspace depth for the latter two distributions centered at their halfspace medians are displayed in Figure~\ref{figure:floating bodies}. Exact expressions for the halfspace depth for the uniform distribution on a simplex, and a (hyper)-cube in $\R^d$ are much more involved for $d>2$ than for $d=2$. They can be obtained from \citep[Lemma~1.3 and its proof]{Schutt1991}.
\end{example}

\begin{figure}[htpb]
\includegraphics[width=.45\textwidth]{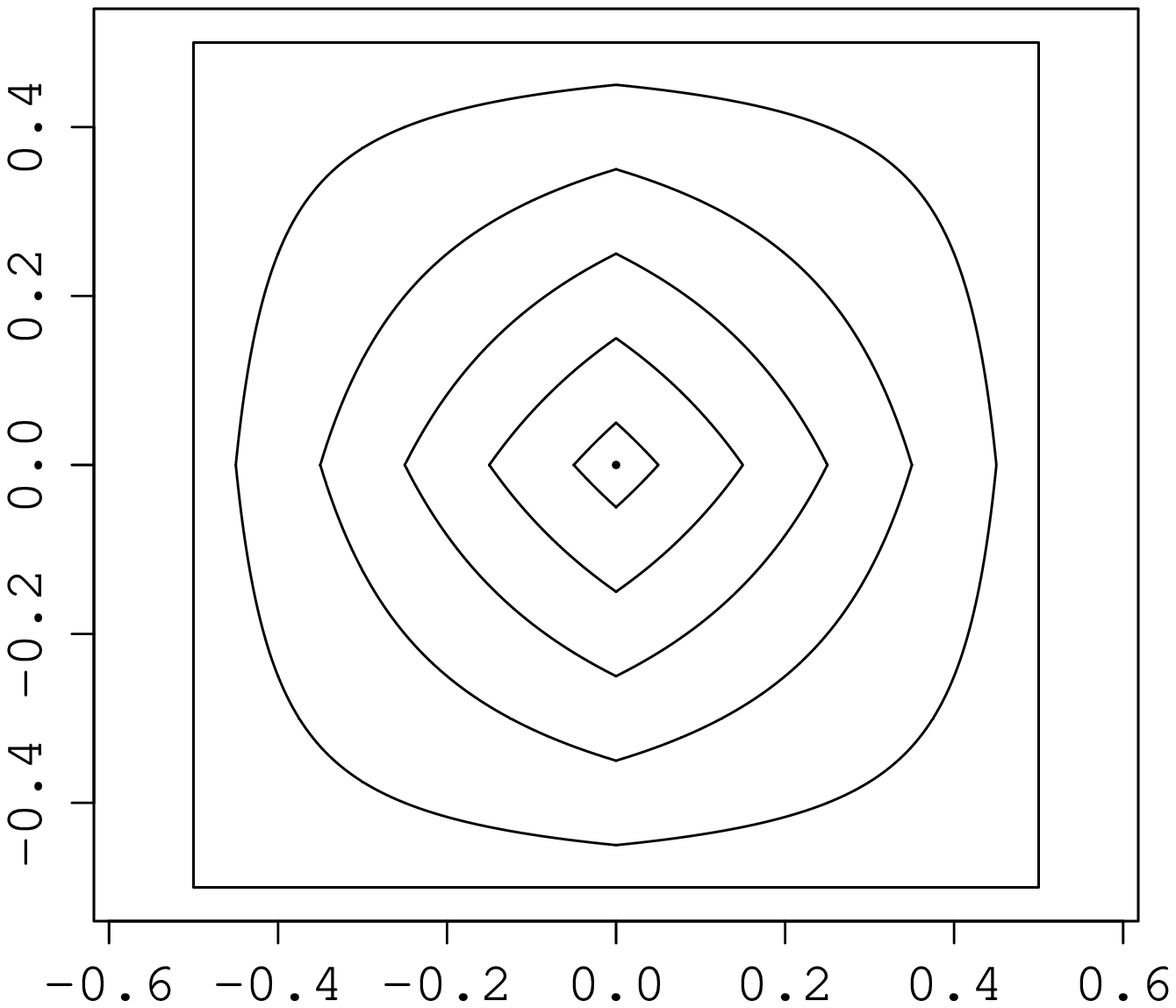} \quad \includegraphics[width=.45\textwidth]{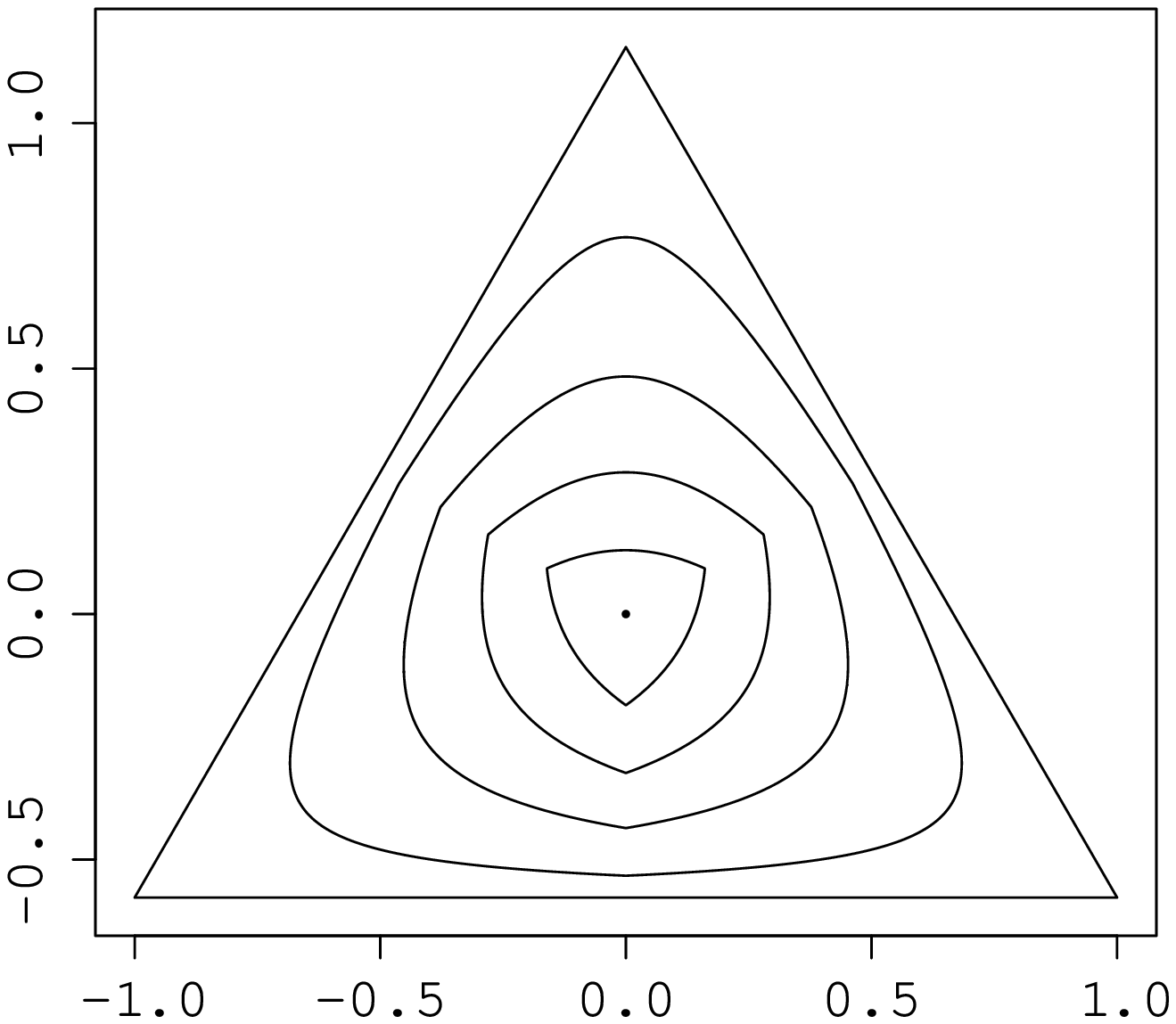}
\caption{Halfspace depth central regions \eqref{central region} for $\delta \in \{ 0.05, 0.15, 0.25, 0.35, 0.45 \}$ for the uniform distribution on the unit square (left panel), and an equilateral triangle (right panel). The black dots stand for the halfspace medians of these two distributions. Note that the central region for $\delta = 0.45$ is empty for $P$ uniform on a triangle. In that case $\MD(P) = 4/9 < 0.45$ for $\MD$ from \eqref{maximal depth}.}
\label{figure:floating bodies}
\end{figure}

\begin{example}	\label{example:alpha symmetric distributions}
In accordance with \citet{Fang_etal1990} we say that the distribution of $X \sim P\in\Prob$ is $\alpha$-symmetric, $0<\alpha\leq 2$, if for some continuous function $\phi \colon \R \to \R$ the characteristic function of the random vector $X$ takes the form 
	\[	\psi_X(t) = \E \Euler^{\iu \left\langle t, X \right\rangle} = \phi\left( \left\Vert t \right\Vert_\alpha \right) \quad \mbox{for all }t \in \R^d,	\]
where $\iu$ is the imaginary unit, and for $t = \left(t_1, \dots, t_d\right) \in \R^d$ we set $\left\Vert t \right\Vert_\alpha = \left(\sum_{i=1}^d \left\vert t_i \right\vert^\alpha\right)^{1/\alpha}$ for $0<\alpha<\infty$, and $\left\Vert t \right\Vert_\infty = \max_{i=1,\dots,d} \left\vert t_i \right\vert$. For $\alpha = 2$ we obtain the collection of all spherically symmetric distributions, i.e. distributions invariant under all orthonormal rotations of the sample space \citep[Chapter~2]{Fang_etal1990}. For instance, the uniform distribution on the unit ball $\B$, the uniform distribution on the unit sphere $\Sph$, or the standard multivariate Gau{\ss}ian distribution are all spherically symmetric. The multivariate probability distribution with independent Cauchy marginals is $1$-symmetric. For $\phi(s) = \Euler^{-s^\alpha}$ with $\alpha \in (0,2]$ we obtain the multivariate symmetric stable laws.

$\alpha$-symmetric distributions have been studied by many authors \citep{Cambanis_etal1983, Fang_etal1990, Koldobsky2005}. They are distinguished by the special property that all univariate projections of an $\alpha$-symmetric measure $X = \left(X_1, \dots, X_d\right) \sim P$ are multiples of the same univariate distribution
	\[	\left\langle u, X \right\rangle \eqd \left\Vert u \right\Vert_\alpha X_1 \quad \mbox{for all }u \in \R^d, 	\]
where $\eqd$ stands for ``is equal in distribution" \citep[Theorem~7.1]{Fang_etal1990}. This makes it possible to compute the depth $\HD\left(\cdot;P\right)$ exactly. For an $\alpha$-symmetric $P\in\Prob$ we have for all $x\in\R^d$
	\begin{equation}	\label{depth for symmetric distributions}
	\begin{aligned}
	\HD\left(x;P\right) & = \inf_{u \in \R^d \setminus \{0\}} \PP\left( \left\langle u, X \right\rangle \leq \left\langle u, x \right\rangle \right) = \inf_{u \in \R^d \setminus \{0\}} \PP\left( X_1 \leq \left\langle u, x \right\rangle/\left\Vert u \right\Vert_\alpha \right) \\
	& = F_1\left( - \left\Vert x \right\Vert_ {\alpha^*} \right),
	\end{aligned}
	\end{equation}
where $F_1$ is the distribution function of $X_1$, and 
	\[	\alpha^* = 	\begin{cases}
									\alpha/(\alpha-1) & \mbox{if }\alpha>1, \\
									\infty & \mbox{if }\alpha \in (0,1], 
									\end{cases}	\]
is the conjugate exponent to $\alpha$. The last equality in \eqref{depth for symmetric distributions} is due to the (generalized) H\"older inequality (see, e.g., \citep[Lemma~A.1]{Chen_Tyler2004}). All central regions $\CR$ of an $\alpha$-symmetric distribution are therefore the lower level sets of the norm $\left\Vert \cdot \right\Vert_{\alpha^*}$. In particular, for all spherically symmetric distributions the central regions are centered balls, and for all $\alpha\leq 1$ the central regions are centered (hyper)-cubes in $\R^d$. Apart from simple uniform distributions on convex bodies such as those in Example~\ref{example:depth level sets of convex bodies} and atomic distributions (see the left panel of Figure~\ref{figure:classification}), $\alpha$-symmetric distributions (and their affine images) are the only class of probability distributions whose depth $\HD$ are we able to evaluate exactly. This was noticed by \citet[Example~(C)]{Masse_Theodorescu1994} and \citet{Chen_Tyler2004}. See also Figure~\ref{figure:depth for distributions with density}.
\end{example}

\subsubsection{Affine invariance}	

For a non-singular matrix $A \in \R^{d\times d}$ and $b \in \R^d$, consider the affine transformation $T \colon \R^d \to \R^d \colon x \mapsto A x + b$. The depth $\HD$ is invariant with respect to $T$
	\begin{equation*}	
	\HD\left( x ; P_X \right) = \HD\left( T(x) ; P_{T(X)} \right) \quad\mbox{for all }x\in\R^d, P_X\in\Prob.	
	\end{equation*}
This implies that central regions $\CR\equiv (P_X)_\delta$ are affine equivariant under affine transformations $T$ of full rank, i.e. $T((P_X)_\delta) = (P_{T(X)})_\delta$ for any $\delta\in[0,1]$. Due to the affine invariance of $\HD$ and Example~\ref{example:alpha symmetric distributions}, the central regions $\CR$ of elliptically symmetric distributions (i.e. invertible affine images of spherically symmetric distributions, see \citep[Chapter~2]{Fang_etal1990}) are concentric ellipsoids with the same center and orientation as the density level sets of $P$ (if the density exists). In particular, this holds true for the central regions of any full-dimensional multivariate Gau{\ss}ian distribution (see Figure~\ref{figure:depth for distributions with density}).

\begin{figure}[htpb]
\includegraphics[width=.45\textwidth]{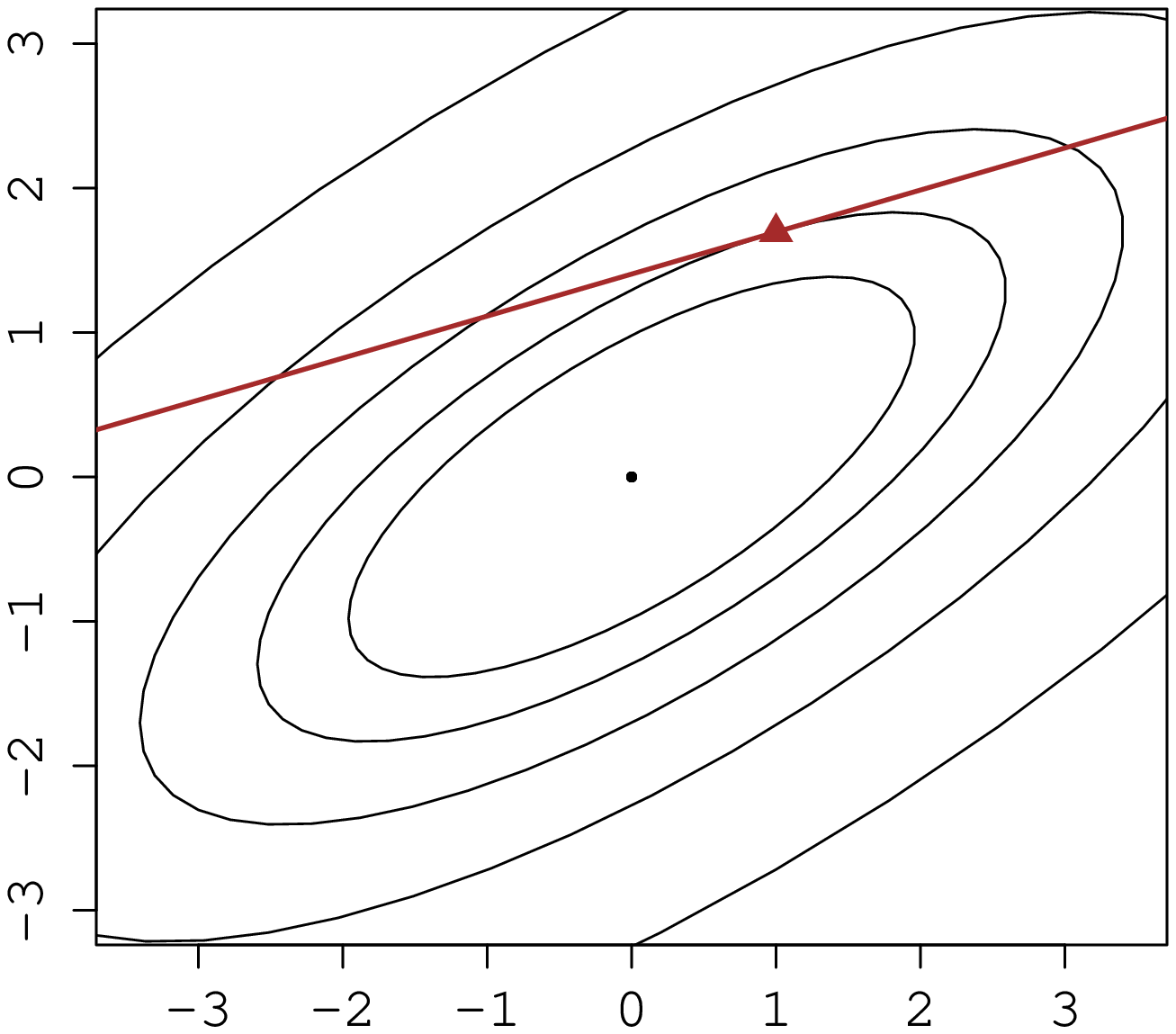} \quad \includegraphics[width=.45\textwidth]{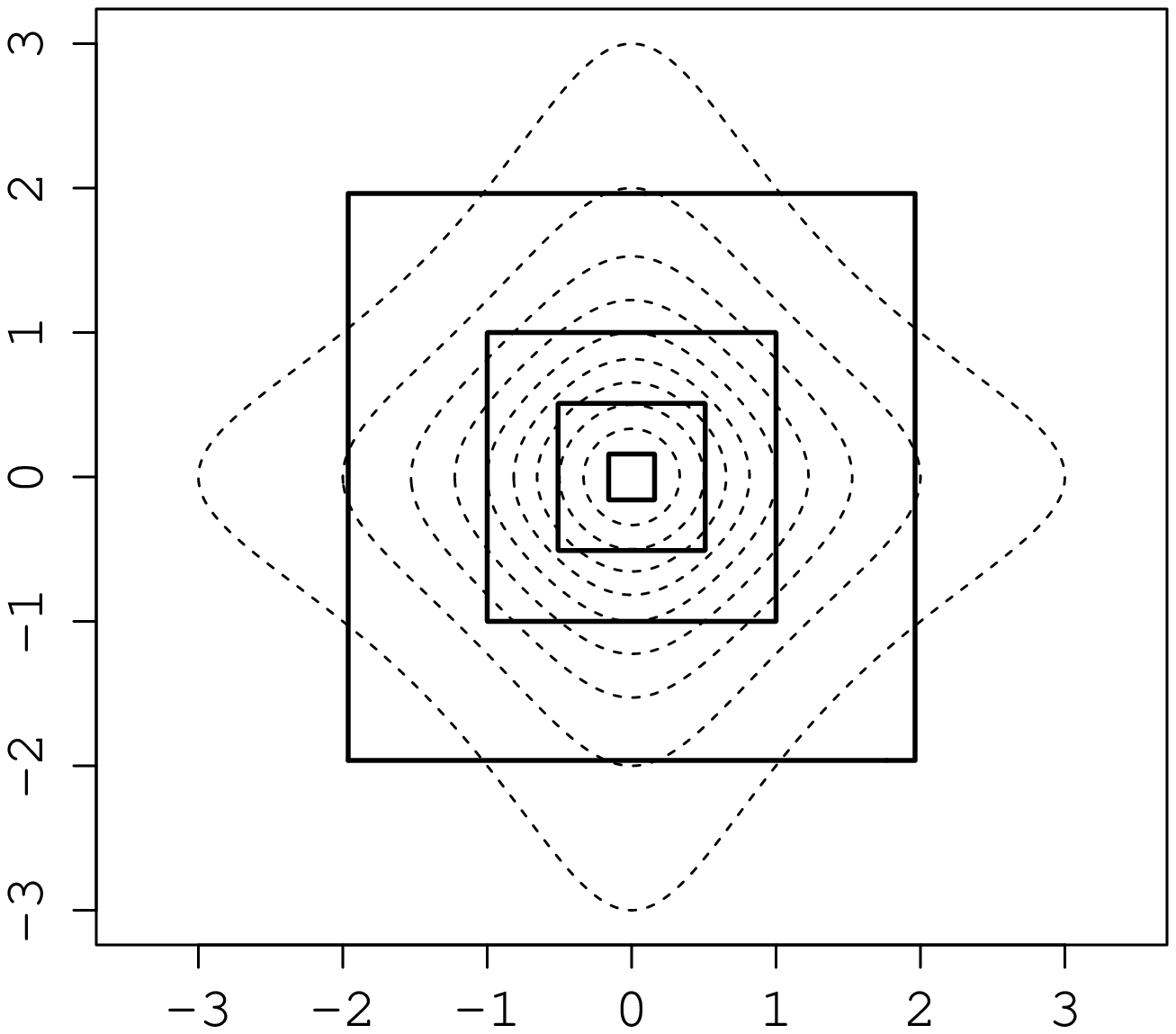}
\caption{Left panel: several contours of the density of a bivariate centered Gau{\ss}ian measure $P$ (thin black lines). For any $x \in \R^d\setminus\{0\}$ (brown triangle) the unique hyperplane $H \ni x$ such that $P(H^-) = \HD(x;P)$ is the hyperplane that supports the contour ellipsoid of the density of $P$ passing through $x$ (solid brown line). Thus, the depth central regions $\CR$ with $\delta \in (0,1/2)$ are all concentric ellipsoids of the same shapes as the density contours. Right panel: several density contours (dashed lines) and the corresponding halfspace depth contours (solid lines) of the bivariate distribution with independent Cauchy marginals $P$. Since $P$ is $1$-symmetric, the central regions $\CR$ are all concentric squares.}
\label{figure:depth for distributions with density}
\end{figure}

\subsubsection{Quasi-concavity} \label{section:quasi-concavity}

The sets $\CR$ are all convex, which means that the mapping $\HD$ is quasi-concave in its first argument. Quasi-concavity of $\HD$ is essential for the construction of estimators based on the depth, such as the depth-trimmed means.

\subsubsection{Maximality at the center}	\label{section:symmetry}

Denote the maximal depth value of a distribution by
	\begin{equation}	\label{maximal depth}
	\MD(P) = \sup_{x\in\R^d} \HD(x;P) \quad\mbox{for }P\in\Prob.
	\end{equation}
By \citet[Lemma~1]{Rousseeuw_Struyf2004}, 
	\[	\MD(P) \leq \left(1+\sup_{x \in \R^d} P\left(\{x\}\right)\right)/2,	\]
and $\MD(P) \leq 1/2$ for $P$ that satisfies \eqref{Delta}. As shown by \citet[Proposition~7]{Rousseeuw_Ruts1999}, for any $P\in\Prob$ the maximal depth is attained in $\R^d$. Therefore, it makes sense to define the halfspace median (or depth median) of $P$ as any point $x_P\in\R^d$ such that
	\[	\HD(x_P;P) = \MD(P).	\]
The halfspace median is not necessarily unique --- consider, for instance, the uniform distribution $P$ on the vertices of a simplex in $\R^d$, where any point in that simplex is a halfspace median of $P$. If the set $M(P) = \left\{ x \in \R^d \colon \HD\left(x;P\right) = \MD(P) \right\}$ is not a singleton, some authors prefer to define the halfspace median as the barycenter of the region $M(P)$. In this paper, we do not follow that convention, and unless stated otherwise, we call all elements of the set $M(P)$ halfspace medians of $P$. 

In general, the set of all halfspace medians of $P$ can be shown to be non-empty, compact and convex. If \eqref{Delta} is true for $P$  with contiguous support, then by \citet[Proposition~7]{Mizera_Volauf2002} the halfspace median of $P$ is unique. In any case, the central regions \eqref{central region} are non-empty if and only if $\delta\in[0,\MD(P)]$.

If the distribution $P$ is (in some sense) symmetric around a point $x_P\in\R^d$, it is natural to require that the center of symmetry $x_P$ is the unique halfspace median of $P$, i.e. the only point such that $\MD(P) = \HD(x_P;P)$.

\begin{definition}
The distribution $P \equiv P_X\in\Prob$ is said to be \emph{centrally symmetric} around $x_P \in \R^d$, if
	\[	P_{X-{x_P}} = P_{{x_P} - X}.	\]
$P$ is centrally symmetric, if it is centrally symmetric around some $x_P \in \R^d$.
\end{definition}
	
If $P$ is centrally symmetric, the maximal depth value $\MD(P)$ must be at least $1/2$, and this depth is attained only at the center of symmetry $x_P$. But centrally symmetric distributions are not the only ones for which the maximal depth is at least $1/2$. This leads to the following definition, due to \citet{Zuo_Serfling2000B}.

\begin{definition}
$P\in\Prob$ is \emph{halfspace symmetric} around $x_P \in \R^d$, if 
	\[	\MD(P) = \HD(x_P;P) \geq 1/2.	\]
$P$ is said to be halfspace symmetric, if it is halfspace symmetric around some $x_P\in\R^d$.
\end{definition}
	
As discussed in \citet{Zuo_Serfling2000B}, the halfspace symmetry of measures in $\R^d$ is more general than the rather restrictive central symmetry, in the sense that any centrally symmetric distribution is also halfspace symmetric. To see that the converse does not hold true, consider the following example.

\begin{example}	\label{example:symmetry}
Let $P\in\Prob[\R^2]$ be the uniform distribution concentrated in the vertices $(\pm 1, \pm 1)$ of a centered square in $\R^2$. $P$ is halfspace symmetric, and centrally symmetric around $x_P = 0 \in \R^2$. For any $\lambda>0$, translate the point mass from $(1,1)$ to $(\lambda, \lambda)$. The resulting distribution $P^{\prime}$ is then still halfspace symmetric around the origin. Yet, for $\lambda \ne 1$, $P^\prime$ is not centrally symmetric. 
\end{example}

Any univariate distribution is halfspace symmetric around its (univariate) median. For a comprehensive discussion on the subject of symmetry of multivariate probability distributions see \citet{Serfling_symmetry}.

\subsubsection{Vanishing at infinity}

Any random vector $X \sim P \in \Prob$ lives with large probability inside a closed ball of finite diameter. Thus, it is reasonable to ask that also the depth associated to $P$ assigns high values of $\HD$ only to points inside (big) closed balls. This property, often called the vanishing at infinity property of $\HD$, can be expressed as
	\[	\lim_{M \to \infty} \sup \left\{ \HD(x;P) \colon \left\Vert x \right\Vert > M \right\} = 0 \quad\mbox{for all }P\in\Prob.	\]
For the halfspace depth this condition is satisfied (see, for instance, \citep[Theorem~2.1]{Zuo_Serfling2000}). The central regions \eqref{central region} are therefore bounded for all $\delta>0$.

\subsubsection{Continuity of the depth}	\label{section:continuity}

As observed by \citet[Lemma~6.1]{Donoho_Gasko1992}, the halfspace depth is upper semi-continuous in its first argument
	\begin{equation}	\label{semicontinuity}
	{\lim\sup}_{x \to x_0} \HD(x;P) \leq \HD(x_0;P)	\quad\mbox{for all }P\in\Prob, x_0\in\R^d.	
	\end{equation}
By \citet[Proposition~1]{Mizera_Volauf2002} if \eqref{Delta} holds true for $P$, then $\HD$ is also continuous in $x$. For the central regions \eqref{central region} condition \eqref{semicontinuity} means that each $\CR$ is a (convex) closed set for $\delta \in [0,\MD(P)]$, and compact for $\delta\in(0,\MD(P)]$ for any $P\in\Prob$.

\subsubsection{Continuity of the central regions}

Consider now the set-valued mapping that for $P \in \Prob$ given, to $\delta \in [0,\MD(P)]$ assigns its central region \eqref{central region}. This mapping is essential for understanding the properties of the depth, as the level sets of $\HD$ are usually of greater interest than individual depth values at fixed points in $\R^d$. The mapping $\delta \mapsto \CR$ takes values in the space $\CB$ of convex subsets of $\R^d$. That space can be equipped with the Hausdorff distance (see, e.g., \citep[Section~1.8]{Schneider2014})
	\begin{equation}	\label{Hausdorff distance}
	\dH(K, C) = \inf\left\{ \varepsilon > 0 \colon K \subseteq C_\varepsilon \mbox{ and }C \subseteq K_\varepsilon \right\} \quad\mbox{for }K, C \in \CB,		
	\end{equation}
where $K_\varepsilon$ is the $\varepsilon$-neighborhood of $K$, 
	\[	K_\varepsilon = K + \B(x,\varepsilon) = \bigcup_{x \in K} \B(x,\varepsilon) \quad\mbox{for }K \in \CB\mbox{ and }\varepsilon >0.	\]
Continuity properties of the map $\delta \mapsto \CR$ were investigated by several authors. The following result was first stated by \citet[Remark~3.6]{Masse_Theodorescu1994}, and later refined by \citet[Theorem~6 and Proposition~7]{Mizera_Volauf2002}, and \citet[Theorem~3.2 and Example~4.2]{Dyckerhoff2018}. In a slightly different context, it was also considered by \citet{Kong_Mizera2012}.

\begin{theorem}
Let \eqref{Delta} be true for $P\in\Prob$ with contiguous support. Then the map $\delta \mapsto \CR$ is continuous in the Hausdorff distance for $\delta \in (0,\MD(P))$. 
\end{theorem}

\subsubsection{Consistency, robustness and other statistical properties}

In statistics, the true distribution $P\in\Prob$ is seldom known. Instead, one usually observes for $n\in\N$ only a random sample $X_1, \dots, X_n$ of independent random variables with distribution $P$, and infers the properties of $P$ from the empirical distribution $P_n$ of that sample. As $n \to \infty$, the halfspace depth is universally consistent, which means that for any $P\in\Prob$ the depth $\HD$ based on the empirical distribution $P_n$ (the sample depth) approaches the true depth evaluated w.r.t. $P$ uniformly over the whole space $\R^d$
	\[	\sup_{x\in\R^d} \left\vert \HD(x;P_n) - \HD(x;P) \right\vert \xrightarrow[n\to\infty]{\as} 0 \quad\mbox{for any }P\in\Prob.	\]
This result was first established in \citet[p.~1817]{Donoho_Gasko1992}. Interestingly, it does not require any properties of the distribution $P$. For $P$ satisfying \eqref{Delta}, it can be strengthened to the form that for any sequence of measures $\left\{ P_\nu \right\}_{\nu=1}^\infty \subset \Prob$ weakly convergent to $P$,
	\[	\sup_{x\in\R^d} \left\vert \HD(x;P_\nu) - \HD(x;P) \right\vert \xrightarrow[\nu\to\infty]{} 0.	\]
This property follows by an argument of \citet[Corollary~2]{Dumbgen1992} applied to $\HD$, see also \citep[Theorem~A.3]{Nagy_etal2016}, and is frequently called the uniform qualitative robustness property of $\HD$. Further robustness properties of $\HD$ were studied by \citet{Romanazzi2001, Romanazzi2002}, and \citet{Chen_Tyler2002, Chen_Tyler2004}, among others.

Uniform consistency results hold true also for the depth level sets \eqref{central region}. In its full generality, the following result, recently established in \citet[Theorem~4.5 and Example~4.2]{Dyckerhoff2018}, unifies and completes the partial results from \citep{Masse_Theodorescu1994, He_Wang1997, Zuo_Serfling2000C}.

\begin{theorem}	\label{theorem:consistency sets}
Let \eqref{Delta} be true for $P\in\Prob$ with contiguous support. Then for every compact interval $A \subset (0,\MD(P)]$
	\[	\sup_{\delta \in A} \dH\left( P_\delta,(P_n)_\delta \right) \xrightarrow[n\to\infty]{\as} 0.	\]
\end{theorem}

In Theorem~\ref{theorem:consistency sets}, $(P_n)_\delta$ stands for the $\delta$-central region \eqref{central region} of the empirical measure $P_n$. Further valuable improvements of the statistical theory of the halfspace depth include the derivation of the rates of convergence of the depth and its central regions \citep{Kim2000, Burr_Fabrizio2017, Brunel2018}, and distributional asymptotics of these and related quantities \citep{Bai_He1999, Zuo_etal2004, Zuo_He2006, Masse2002, Masse2004, Masse2009}.

\section{Description at the center: Winternitz measure of symmetry}	\label{section:Winternitz}

\subsection{Maximal depth of a point}

Several results on the maximal depth mapping $\MD$ from \eqref{maximal depth} can be found in literature much earlier than the definition of the halfspace depth (see \citep[Sections~3 and~4]{Rousseeuw_Ruts1999}). From these references, it appears that the behavior of the maximal depth relates to the degree of concavity of the measure $P$. Following \citet{Borell1974}, see also \citet{Bobkov2010}, let us first provide a rigorous definition of concave probability measures.

\begin{definition}
We say that $P\in\Prob$ is an $s$-concave measure for $-\infty\leq s<\infty$, if 
	\begin{equation*}	
	P\left(\lambda A + (1-\lambda)B\right) \geq \begin{cases}
																							\min\left\{ P(A), P(B) \right\} & \mbox{for }s = -\infty, \\
																							P(A)^\lambda P(B)^{1-\lambda}	& \mbox{for }s = 0, \\
																							\left( \lambda P(A)^{s} + (1-\lambda) P(B)^{s} \right)^{1/s}	& \mbox{otherwise},
																							\end{cases}
	\end{equation*}
for all non-empty Borel sets $A, B \subseteq \R^d$ and $\lambda \in [0,1]$. 
\end{definition}

As noted by \citet{Bobkov2010}, if $P$ is not a Dirac measure, then $s\leq 1$. Further, a measure $P\in\Prob$ is $s$-concave with $s \leq 1/d$ if and only if $P$ has a density $f$ that is supported on an open convex subset $U$ of $\R^d$ and that is $s_d=s/(1-ds)$-concave, i.e., for all $x,y \in U$, for all $\lambda \in [0,1]$, 
	\[	f\left(\lambda x +(1-\lambda) y\right) \geq \begin{cases}
																									\min\left\{ f(x), f(y) \right\} & \mbox{for }s = -\infty, \\
																									f(x)^\lambda f(y)^{1-\lambda}	& \mbox{for }s = 0, \\
																									\left(\lambda f(x)^{s_d} + (1-\lambda) f(y)^{s_d} \right)^\frac{1}{s_d}	& \mbox{otherwise}.
																									\end{cases}	\]

For $s = 0$, $s$-concave measures are also called log-concave measures, and represent a natural generalization of uniform measures on convex bodies. Indeed, any uniform measure on a convex body is log-concave.

We are ready to state a result that summarizes what is known about the maximal depth functional $\MD(P)$ defined in \eqref{maximal depth}.

\begin{theorem}	\label{theorem:maximum}
The following inequalities hold true:
\begin{enumerate}[label=(\roman*)]
\item \label{maximum:1} For any $P\in\Prob$
	\[	\left(1+\sup_{x \in \R^d} P\left(\{x\}\right)\right)/2 \geq \MD(P) \geq \frac{1}{d+1}.	\]
\item \label{maximum:2} For $P\in\Prob$ uniformly distributed on a convex body, 
	\[	1/2 \geq \MD(P) \geq \left(\frac{d}{d+1}\right)^d > \Euler^{-1} > 0.367.	\]
\item \label{maximum:3} For an $s$-concave measure $P\in\Prob$ with $-1<s\leq 1$,
	\begin{equation}	\label{Bobkov bound}
			1/2 \geq \MD(P) \geq 	\begin{cases}
														\Euler^{-1} & \mbox{for }s=0, \\
														\left(\frac{1}{s+1}\right)^{1/s} & \mbox{otherwise.}
														\end{cases}
	\end{equation}
\end{enumerate}
\end{theorem}

As noted by \citet[Section~4]{Grunbaum1960}, the lower bounds in parts \ref{maximum:1} and \ref{maximum:2} are sharp. In part~\ref{maximum:1} it is enough to take the uniform distribution in the vertices of a simplex in $\R^d$. For part~\ref{maximum:2} one takes the uniform distribution on the simplex in $\R^d$. 

\begin{problem}
Are the lower bounds in part~\ref{maximum:3} of Theorem~\ref{theorem:maximum} sharp? That is, does there exist an $s$-concave probability measure $P$ with equality on the right hand side of \eqref{Bobkov bound}?
\end{problem}

The lower bounds in parts \ref{maximum:1} and \ref{maximum:2} of Theorem \ref{theorem:maximum} were proved by \citet{Neumann1945} for $d=2$. In full generality, part \ref{maximum:2} was proved independently by \citet{Grunbaum1960}, and \citet{Hammer1960}. Part \ref{maximum:3} can be found in \citet[Proposition~3]{Caplin_Nalebuff1991}, see also \citet[Theorem~5.2]{Bobkov2010}. As discussed by \citet{Bobkov2010}, the condition $s>-1$ implies the existence of the expectation $\E X$ of $X \sim P$. Actually, in all three parts of Theorem~\ref{theorem:maximum} in the proofs it is shown that $\HD(\E X; P)$ is never smaller than the given lower bounds.

\begin{problem}
Is there a non-trivial lower bound for $\MD(P)$ for all $s$-concave measures with $s\leq -1$?
\end{problem}

\subsection{Central and halfspace symmetry: Funk's theorem}

For part~\ref{maximum:2} of Theorem~\ref{theorem:maximum}, there exists a remarkable converse.

\begin{theorem}	\label{theorem:Funk}
Let $P\in\Prob$ be uniformly distributed on a convex body $K\in\CB$. Then $P$ is halfspace symmetric around $x_P\in\R^d$ if and only if it is centrally symmetric around $x_P$.
\end{theorem}

The proof of Theorem~\ref{theorem:Funk} was first obtained in 1915 for $d=2$ and $d=3$ by \citet{Funk1915}. In its full generality the result was conjectured, among others, by \citet[p.~251]{Grunbaum1963}, but completely solved only in 1970 in \citet[Satz~4.2]{Schneider1970B} and \citet[Theorem~1.5]{Schneider1970}, see also \citet{Falconer1983}. For its modern version, including an extension to star convex bodies $K \subset \R^d$ see \citet[Section~5.6]{Groemer1996}.

By Theorem~\ref{theorem:Funk}, the two notions of central and halfspace symmetry from Section~\ref{section:symmetry} coincide for uniform distributions on (star) convex bodies in $\R^d$, see also Example~\ref{example:symmetry}. This suggests the following problem.

\begin{problem}	\label{problem:symmetry}
Under which conditions can Theorem~\ref{theorem:Funk} be generalized to probability measures?
\end{problem}

A partial answer to Problem~\ref{problem:symmetry} can be found if one considers the notion of angular symmetry for random vectors, proposed by \citet[Section~2]{Liu1988}.

\begin{definition}
The distribution of a random vector $X \sim P \in\Prob$ is said to be \emph{angularly symmetric} around $x_P \in \R^d$, if the random variables $(X - x_P)/\left\Vert X - x_P \right\Vert$ and $-(X - x_P)/\left\Vert X - x_P \right\Vert$ are identically distributed. $P$ is angularly symmetric, if it is angularly symmetric around some $x_P \in \R^d$.
\end{definition}

Angular symmetry can be shown to be an intermediate between the rather strong concept of central symmetry, and the halfspace symmetry, considered in Section~\ref{section:symmetry}. Any $P$ that is centrally symmetric around $x_P$ is angularly symmetric around $x_P$ \citep[Lemma~2.2]{Zuo_Serfling2000B}, and any $P$ angularly symmetric around $x_P$ is also halfspace symmetric around $x_P$ \citep[Lemma~2.4]{Zuo_Serfling2000B}. None of these implications can be reversed. Though, a partial reverse to the second one was asserted in the statistical literature. For $d=2$, \citet[Theorem~2.6]{Zuo_Serfling2000B} in 2000 and \citet[Theorem~2]{Dutta_etal2011} in 2011 independently proved that if $P$ is absolutely continuous and halfspace symmetric around $x_P \in \R^d$, then $P$ must be also angularly symmetric around $x_P$. \citet[Theorems~1 and~2]{Rousseeuw_Struyf2004} in 2004 gave a complete proof for general $d\in\N$ in the following form.

\begin{theorem} \label{theorem:Funk for measures}
The distribution $P\in\Prob$ is angularly symmetric around $x_P \in \R^d$ if and only if
	\[	\HD(x_P;P) = \left( 1 + P(\{x_P\}) \right)/2.	\]
In particular, 
	\begin{enumerate}[label=(\roman*)]
	\item any $P$ halfspace symmetric around $x_P$ with $P(\{x_P\}) = 0$ is angularly symmetric around $x_P$, and
	\item for any $P$ such that $\sup_{x\in\R^d} P(\{x\}) = 0$, halfspace symmetry and angular symmetry are equivalent notions.
	\end{enumerate}
\end{theorem} 


When $P$ is the uniform distribution on a (centered) convex body $K \in \CB$, Theorem~\ref{theorem:Funk for measures} stands as a generalization of Funk's theorem to probability measures. Indeed, assume that $P$ is halfspace symmetric around the origin $x_P = 0 \in \R^d$. Since $P$ is absolutely continuous, by Theorem~\ref{theorem:Funk for measures} it is also angularly symmetric around $x_P$. Because $P$ is uniform, angular symmetry of $P$ implies that the support function $h_K$ from \eqref{support function} must be an even function on $\Sph$, which in turn gives that $K$ must be centrally symmetric around $x_P$. 


Remarkably, \citet[Theorems~1 and~2]{Rousseeuw_Struyf2004} were discovered independently of the results in geometry. The proof of \citet{Rousseeuw_Struyf2004} makes use of the classical theorem of \citet{Cramer_Wold1936} from 1936, closely related to the Fourier transforms of measures. The known proofs of Theorem~\ref{theorem:Funk} employ techniques from spherical harmonics, or integral equations. Thus, all known proofs of Theorems~\ref{theorem:Funk} and~\ref{theorem:Funk for measures} are non-trivial, but have in common the use of harmonic analysis.

\subsection{Measures of symmetry}

Characterization results like Theorem~\ref{theorem:Funk} for convex bodies stimulated much research in convex geometry. Eventually, these efforts led to \emph{measures of symmetry for convex sets}, comprehensively covered by \citet{Grunbaum1963}. A measure of symmetry is a mapping $\MS \colon \CB \to [0,1]$ such that 
	\begin{enumerate}[label=(\roman*)]
	\item $\MS(K) = 1$ if and only if $K$ is (centrally) symmetric, 
	\item $\MS(K) = \MS(T(K))$ for any non-singular affine transformation $T \colon \R^d \to \R^d$, and
	\item $\MS$ is continuous on $\CB$ (equipped with a suitable topology\footnote{For details on possible choices of topology see \citet{Grunbaum1963}.}).
	\end{enumerate}

A variant of part~\ref{maximum:2} in Theorem~\ref{theorem:maximum}, that states that for any $X \sim P\in\Prob$ uniformly distributed on a convex body 
	\[	\HD\left(\E X; P \right) \geq \left(\frac{d}{d+1}\right)^d,	\]
is known since the 1910s as the Winternitz theorem (due to Artur Winternitz, according to \citep{Blaschke1923}). This result gave rise to the following measure of symmetry, which is remarkably close to the halfspace depth.

\begin{definition}
Let $P\in\Prob$ be the uniform distribution on $K\in\CB$. For $x\in K$ and $H \in \mathcal H$ with $x\in H$, let
	\[	w_K(x;H^-) = \frac{P(H^- \cap K)}{1-P(H^- \cap K)},	\]
and consider $w_K(x) = \min\left\{ w_K(x;H^-) \colon H \in \mathcal H, x \in H \right\}$. The \emph{Winternitz measure of symmetry} of $K$ is then defined as
	\[	W(K) = \max\left\{ w_K(x) \colon x \in K \right\}.	\]
\end{definition}

The measure of symmetry $W(K)$ was considered by many authors. For a historical account and the theoretical background on measures of symmetry see the seminal paper of \citet[Section~6.2]{Grunbaum1963}. For a modern treatment of the topic see \citet{Toth2015}. 

Obviously, for $K \in \CB$, the Winternitz measure of symmetry is equivalent with the maximal depth \eqref{maximal depth} attained w.r.t. the uniform measure $P$ on $K$
	\[	\MD(P) = \frac{W(K)}{1+W(K)}.	\]
The function $w_K \colon K \to [0,\infty]$ used in the definition of $W(K)$ links directly to $\HD$ via
	\[	\HD(x;P) = \frac{w_K(x)}{1+w_K(x)}.	\]
For $w_K$, it was noted already by \citet{Grunbaum1963} in 1963 that its upper level sets are convex, and that its maximal value is always attained in $K$ (cf. Sections~\ref{section:quasi-concavity} and \ref{section:symmetry} above).


Connections of the depth $\HD$ with results on partitions of convex bodies (Theorem~\ref{theorem:maximum} above) have already been noted by \citet{Rousseeuw_Ruts1999}. Though, as far as we know, no links between the measures of symmetry for convex bodies and the halfspace depth have yet been established in the statistical literature. 

In the other direction, some notions of depth can be found in the literature on the geometry of convex bodies. For instance, in \citet{Bose_etal2011} the ``depth'' for a convex body $K$ is defined as the halfspace depth \eqref{halfspace depth} of the associated uniform distribution, in connection with a generalized version of the Winternitz theorem. Nonetheless, precise links between the respective fields of mathematics appear to be still lacking.

\subsubsection{The ray basis theorem} 

For $P\in\Prob$ and $x\in\R^d$ we say that a halfspace $H^- \in \mathcal H^-$ is minimal at $x$ if $x \in H$ and $P(H^-) = \HD(x;P)$. $H$ is then called a minimal hyperplane of $x$. From the definition of the minimal halfspace it is easy to see that the following holds. 

\begin{proposition}	\label{minimal halfspace supports}
Let $P\in\Prob$ have contiguous support and let $H^- \in \mathcal H^-$ be minimal at $x \in \R^d$ with $\HD(x;P) = \delta$. Then the halfspace $H^+$ supports $\CR$.
\end{proposition}

An interesting characterization of the halfspace median of a measure $P\in\Prob$ in terms of minimal halfspaces was observed by \citet[pp.~1818--1819]{Donoho_Gasko1992} in 1992 and \citet[Propositions~8 and~12]{Rousseeuw_Ruts1999} in 1999. For $P$ absolutely continuous, $x$ is a halfspace median of $P$ if and only if the union of the collection of minimal halfspaces at $x$ is $\R^d$. In \citet{Rousseeuw_Ruts1999}, this result is dubbed the ray basis theorem.

\begin{theorem} \label{theorem:ray basis}
Let $P\in\Prob$, and $x \in \R^d$ be such that the union of the collection of minimal halfspaces at $x$ is $\R^d$. Then $x$ is a halfspace median of $P$.

Assume that $P$ satisfies \eqref{Delta}, and let $x\in\R^d$ be a halfspace median of $P$. Then there exists a collection of minimal halfspaces at $x$ of cardinality at most $d+1$ whose union is $\R^d$.
\end{theorem}

The smoothness condition \eqref{Delta} is important in Theorem~\ref{theorem:ray basis}. As noted by \citet[Example~4.3]{Masse2004} it is possible to construct distributions $P\in\Prob$, that violate \eqref{Delta}, with a unique minimal halfspace at their halfspace median.

For $P\in\Prob$ uniformly distributed on a convex body $K \in \CB$, a result similar to Theorem~\ref{theorem:ray basis} was stated in \citet[p.~251]{Grunbaum1963} in 1963 for the Winternitz measure of symmetry. There, it was asserted that it follows from a version of Helly's theorem that there must exist at least $d+1$ different minimal halfspaces at the halfspace median $x_P$ of $P$. The assumptions of that result appear, however, to be incomplete, as pointed out to us by M.~Tancer~\citep{Tancer2}.

Another interesting problem closely connected with the halfspace median and Theorem~\ref{theorem:ray basis}, is a conjecture of \citet[p.~41]{Grunbaum1961} from 1961 that asks if for any convex body $K \in \CB$ with $d \geq 2$ there exists a point $x \in K$ that is a centroid of at least $d+1$ sections of $K$ by different hyperplanes passing through $x$. For $d = 2$, the solution to this problem is straightforward, as noted already in \citep{Grunbaum1961}. For $d>2$, this problem appears to be still open (see \citep{Steinhaus1955}, \citep[p.~251]{Grunbaum1963}, and \citep[Problem~A8]{Croft_etal1994}). It is natural to conjecture that the halfspace median is such a point. Indeed, combine Theorem~\ref{theorem:ray basis} with a theorem of \citet{Dupin1822} (stated in part~\ref{Dupin} of Proposition~\ref{PropFloatBod1} below) that says that for any $K\in\CB$ the point $x \in K$ is the centroid of all minimal hyperplanes at $x$ (w.r.t. the uniform distribution $P$ on $K$) to obtain that if the minimal hyperplanes at $x$ are in general position, then the halfspace median is a point as postulated in the conjecture. Here, a set of hyperplanes is said to be in general position if for all choices of at most $d$ such distinct hyperplanes their normals are linearly independent. A further open question is if the conjecture holds true with $x$ being the centroid of $K$.

Theorem~\ref{theorem:ray basis} provides a useful characterization criterion for the depth-based extension of the median. Apart from its theoretical appeal, it promises applications in the computation of the depth, and the depth median.

\subsubsection{Minimality and stability}

An important question regarding the measures of symmetry concerns their minimality, i.e. characterization of sets $K \in \CB$ such that $\MS(K) = \inf\left\{ \MS(K^\prime) \colon K^\prime \in \CB \right\}$. As remarked by \citet[Section~4]{Grunbaum1960} in 1960, for the Winternitz measure of symmetry
	\[	\inf \left\{ W(K^\prime) \colon K^\prime \in \CB \right\} = \frac{d^d}{(d+1)^d - d^d},	\]
and this value is attained if and only if $K$ is a bounded cone in $\R^d$. This value corresponds to
	\[	\inf_{K^\prime\in\CB} \left\{ \MD(P) \colon P \mbox{ is distributed uniformly on }K^\prime \right\} = \left(d/(d+1)\right)^d,	\]
see also Theorem~\ref{theorem:maximum}. In a related question, \citet{Grunbaum1960} also determined the collection of measures $P\in\Prob$ such that 
	\[	\MD(P) = 1/(d+1) = \inf\left\{ \MD(P^\prime) \colon P^\prime \in \Prob \right\},	\]
by showing that this can happen if and only if $P$ is a uniform distribution on the vertices of a non-degenerate simplex in $\R^d$. In statistics, this result was observed independently by \citet[Lemma~6.3]{Donoho_Gasko1992} in 1992 for $\HD$.

In convex analysis, another desirable property of measures of symmetry is their stability. A measure of symmetry $\MS$ is said to have the stability property if for any $\varepsilon > 0$ and $K \in \CB$ with $\MS(K) < \inf\left\{ \MS(K^\prime) \colon K^\prime \in \CB \right\} + \varepsilon$ there exists a constant $c>0$ and $L \in \CB$ such that $\MS(L) = \inf\left\{ \MS(K^\prime) \colon K^\prime \in \CB \right\}$, and $\delta(K,L) \leq c \, \varepsilon$. Here, $\delta$ stands for some metric on $\CB$, and $c$ may depend on $d$, as well as on some characteristic of $K$ such as its volume, or diameter. An important stability theorem for the Winternitz measure of symmetry was derived by \citet[Theorem~2]{Groemer2000}. 

\begin{theorem}	\label{theorem:stability}
Let $K \in \CB$ and let $P\in\Prob$ be uniformly distributed on $K$. Let $\varepsilon \geq 0$. There exists a constant $\lambda > 0$ depending only on $d$ such that $\MD(P) \leq \left(d/(d+1)\right)^d + \varepsilon$ implies that $K$ contains a bounded cone $C\in\CB$ with
$$	
\dS(K,C) \leq \lambda \vol{K} \varepsilon^{1/2d^2},
$$
where, \begin{equation}\label{SymDiff}\dS(K,C) = \vol{K \cup C} - \vol{K \cap C}
\end{equation} 
is the symmetric difference metric on $\CB$.
\end{theorem}

As far as we are aware, no results corresponding to stability theorems can be found for probability measures and the halfspace depth. 

\begin{problem}
Does a variant of a stability result such as Theorem~\ref{theorem:stability} hold for probability measures and depth medians?
\end{problem}

\subsection{Affine invariant points}

Symmetry is a key structural property of convex bodies relevant in many problems. A systematic study of symmetry was initiated by Gr\"unbaum in his seminal paper \citep{Grunbaum1963} from 1963. A crucial notion in his work is that of affine invariant point. It allows to analyze the symmetry situation. In a nutshell: the more affine invariant points, the fewer symmetries.

Recall that the set $\CB$ is equipped with the Hausdorff distance \eqref{Hausdorff distance}.
 
\begin{definition}
A map $p \colon \CB \to \R^{d}$ is called an affine invariant point, if $p$ is continuous and if for every non-singular affine map $T:\mathbb R^{d}\rightarrow \mathbb R^{d}$ one has, 
	\begin{equation*}
	p(T(K))=T(p(K)).
	\end{equation*}		
	We denote by $\AIP$ the set of all affine invariant points on $\R^{d}$.
\end{definition}
$\AIP$ is an affine subspace of $C(\CB,\R^{d})$, the space of continuous mappings from $\CB$ to $\R^d$.

Examples of affine invariant points, already known to \citet{Grunbaum1963} are, e.g., the centroid of a convex body $K$ (i.e. the expectation of the uniform distribution on $K$), the Santal\'o point (the unique point $s(K)$ in the interior of $K\in\CB$ for which the minimum of the functional $\vol{K^{s(K)}}$ is attained, see also the important Blaschke-Santal\'o inequality in \eqref{Blaschke-Santalo} below), and the center of the ellipsoid of maximal volume inside a convex body. 
\citet{Grunbaum1963} asked a number of questions about affine invariant points: 
	\begin{enumerate}[label=(\roman*)]
	\item \label{Grunbaum1} Is there a convex body $K$ such that $\AIP(K)=\{p(K) \colon p\in\AIP\} = \R^{d}$?
	\item \label{Grunbaum2} Is the space $\AIP$ infinite-dimensional?
	\item \label{Grunbaum3} Let $K$ be a convex body and let $T \colon \R^{d}\to\R^{d}$ be an affine map with $T(K)=K$. We denote
		\[	\FD(K) = \left\{ x\in \R^{d} \colon \mbox{ for all $T$ such that $T(K)=K$ we have $Tx=x$} \right\}.	\]
Do we have $\FD(K)=\AIP(K)$?
	\end{enumerate}
One can argue that those convex bodies that have only one affine invariant point are the most symmetric convex bodies. This would include the simplex in $\R^d$ which is from another point of view the most non-symmetric convex body (see Theorem~\ref{theorem:maximum}). 

A convex body has only one affine invariant point, if it has {\em enough symmetries}.
We say that an affine map $T\colon \R^{d}\to\R^{d}$ is a symmetry of a convex body $K$ if $T(K)=K$. We say that a convex body has enough symmetries if the only affine maps commuting with all symmetries of $K$ are multiples of the identity. 

For a convex body $K$ with enough symmetries the halfspace median coincides with the centroid of $K$.

The following theorems answer Gr\"unbaum's questions \ref{Grunbaum1} and \ref{Grunbaum2}. They can be found in \citet{Meyer_etal2015}.

\begin{theorem}
For every $d \geq 2$ there is a body $K \in \CB$ with $\AIP(K) = \left\{p(K)\colon p\in\AIP\right\} = \R^d$. Such convex bodies are actually dense in $\CB$ with respect to the Hausdorff metric.
\end{theorem}

\begin{theorem}
The space $\AIP$ is infinite-dimensional.
\end{theorem}

In the proofs of these theorems, new classes of affine invariant points were introduced using convex floating bodies (see Section~\ref{FloatBodConv} below). We define $p_{\delta} \colon \CB \to \R^d$ to be the mapping that sends $K$ to the centroid of $\CR$ from \eqref{central region} for $P$ uniform on $K$.

Moreover, in \citet[Theorem~2]{Meyer_etal2015} it was shown that for convex bodies $K$ with $\operatorname{dim}(\AIP(K))=d-1$ a positive answer to Gr\"unbaum's question \ref{Grunbaum3} above holds, i.e. $\FD(K)=\AIP(K)$. It was settled in all dimensions by \citet{Mordhorst2017}, based on work by \citet{Kucment1972} (see also \citep{Kucment2016}) where question \ref{Grunbaum3} of Gr\"unbaum was almost proved already in 1972, with only a compactness argument missing.

\begin{theorem}
For any $K \in \CB$ we have that $\FD(K)=\AIP(K)$.
\end{theorem}

\section{Description at the boundary: Convex floating bodies}	\label{section:convex floating bodies}

Data depth is intimately related to the concept of \emph{floating body} which we now introduce. We start with a brief discussion of differentiability properties of the boundary of convex bodies, since this will  be essential in what follows.

\subsection{Curvature of convex bodies}

We take as a measure on the boundary $\partial K$ of a convex body $K\in \CB$  the restriction of the $(d-1)$-dimensional Hausdorff measure to $\partial K$. We call this measure the boundary measure, or the Lebesgue measure on $\partial K$, and denote it by $\mu_{\partial K}$. Let $\mathcal U$ be an open subset of $\R^{d}$ and $f \colon \mathcal U \to \R$ be a twice continuously differentiable function. Then the classical Gau{\ss}-Kronecker curvature at $x_{0}\in\mathcal U$ is
	\[	\kappa(x_{0})=\frac{ \det \left(\nabla^{2} f(x_{0})\right)}{\left(1 + \|\nabla f(x_{0}) \|^2 \right)^{\frac{d+2}{2}}},	\]
where $\nabla f$ is the gradient of $f$ and $\nabla^{2}f$ the Hessian of $f$. The Gau{\ss}-Kronecker curvature of the boundary of a convex body is
the curvature of a function parametrizing the boundary.

By a theorem of Rademacher (see, e.g., \citep[Theorem~2.5.1]{Borwein_Vanderwerff2010}), a convex function on $\mathbb R^{d}$, and in particular the boundary of a convex body, is almost everywhere differentiable. There are, however, examples of convex functions and convex bodies that are not differentiable on a dense set of $\R^{d}$ and of the boundary of the convex body, respectively. Those examples do not have a second derivative at any point and thus the classical Gau{\ss}-Kronecker curvature $\kappa$ does not exist at any point.

Therefore we use the generalized Gau{\ss}-Kronecker curvature as introduced by \citet{Busemann_Feller1936} in dimension $d=3$ and \citet{Aleksandrov1939} in general. We present here only a short explanation of  the generalized Gau{\ss}-Kronecker curvature and  we refer to e.g., 
\citep[Section~1.6]{Schutt_Werner2003} and \citep{Schutt_Werner1990} for a detailed account.  

A cap of $K\in\CB$ at $x \in \partial K$ is the intersection of a halfspace $H^-$ with $K$ such that there is a supporting hyperplane to $K$ at $x$ that is parallel to $H$. There may, of course, be points on the boundary of $K$ having more than one supporting hyperplane. But, those points are of measure $0$ and shall be of less importance in our discussion.

If $K$ has a unique supporting hyperplane at $x\in \partial K$, we denote by $\Delta(x,\delta)$ the height of a cap with volume $\delta$. The height of a cap is the distance of the supporting hyperplane at $x$ to the parallel hyperplane cutting off a set of volume $\delta$.

\begin{definition}
Let $K\in\CB$ and $x\in\partial K$. Let $c_{d}=2^{d+1} \left(\vol[d-1]{\B[d-1]}/(d+1)\right)^{2}$. Assume that $K$ has at $x$ a unique supporting hyperplane. We say that $K$ has a generalized Gau{\ss}-Kronecker curvature if the limit
	\[	\lim_{\delta \to 0}c_{d} \  \frac{\Delta (x,\delta)^{d+1}}{\delta^{2}}	\]
exists.
In this case we define 
\begin{equation}\label{DefGenCurv2}
\kappa(x)
=\lim_{\delta \to 0}c_{d} \  \frac{\Delta (x,\delta)^{d+1}}{\delta^{2}}
\end{equation}
to be the generalized Gau{\ss}-Kronecker curvature at $x$.
\end{definition}

If the Gau{\ss}-Kronecker curvature exists, then it is equal to the generalized Gau{\ss}-Kronecker curvature. By a theorem of Busemann, Feller and Aleksandrov \citep{Busemann_Feller1936, Aleksandrov1939} the generalized Gau{\ss}-Kronecker curvature of a convex body exists almost everywhere. Geometrically, the existence of the generalized Gau{\ss}-Kronecker curvature at $x$ means that $\partial K$ can be ``well" approximated by an ellipsoid, or ellipsoidal cylinder at $x$ (see, e.g., \citep[Section~1.6]{Schutt_Werner2003}).

The following example clarifies the difference between Gau{\ss}-Kronecker curvature and generalized Gau{\ss}-Kronecker curvature.

\begin{example}	\label{example:function}
Let $f:[-1,1]\to\mathbb R$ be defined by
	\[	f(x)=	\begin{cases}
						x^{2} & \text{if } \left\vert x \right\vert = 1/n \mbox{ and }n\in\N, \\
						\frac{2n+1}{n(n+1)}\left\vert x \right\vert-\frac{1}{n(n+1)} & \text{if }\frac{1}{n+1} < \left\vert x \right\vert < 1/n \mbox{ and }n\in\N. 
						\end{cases}	\]
The function $f$ is not differentiable at the points $x=\pm 1/n$ and therefore $f$ is not twice differentiable at $0$. Thus, the Gau{\ss}-Kronecker curvature of $f$ does not exist at $0$. On the other hand, it is not difficult to compute that $f$ has a generalized Gau{\ss}-Kronecker curvature at $0$ and this curvature is $2$, see Figure~\ref{figure:function}.
\end{example}

\begin{figure}[htpb]
\includegraphics[width=.65\textwidth]{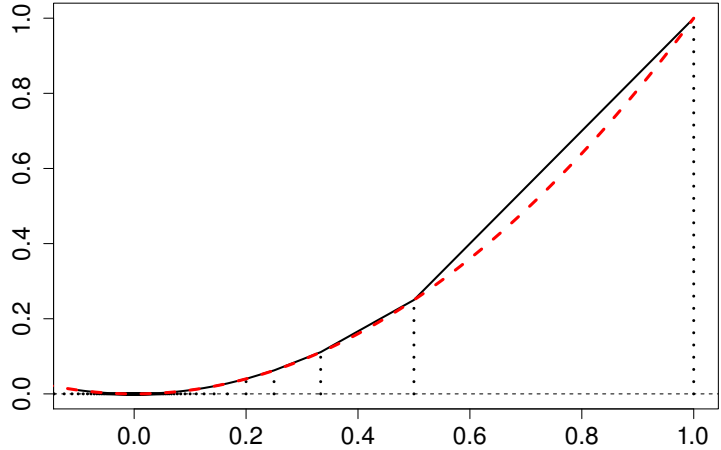} 
\caption{The function $f$ from Example~\ref{example:function} (black solid line) and the function $x \mapsto x^2$ (red dashed line) that approximates $f$ around $x=0$. Since $f$ is not twice differentiable at $x = 0$, its Gau{\ss}-Kronecker curvature does not exist at $0$. Its generalized Gau{\ss}-Kronecker curvature at $0$ exists and is equal to $2$, the Gau{\ss}-Kronecker curvature of $x \mapsto x^2$.}
\label{figure:function}
\end{figure}

\subsection{Floating body and convex floating body}\label{FloatBodConv}

Earliest records on floating bodies can be traced back to the early 19th century work of \citet{Dupin1822} and are motivated by mechanics. By the Archimedean principle, a solid convex body $K\in\CB[3]$ of constant (volumetric mass) density that floats in water has always a set of the same volume above the water surface, regardless of its position.

This leads to the definition of floating bodies for convex bodies in $K\in\CB$ according to Dupin: A nonempty convex subset $K_{[\delta]}$ of $K$ is a floating body of $K$ if each supporting hyperplane to $K_{[\delta]}$ cuts off a set of volume $\delta>0$ of $K$. Dupin observed that a support hyperplane $H$ to $K_{[\delta]}$ touches the boundary of $K_{[\delta]}$ in exactly one point, the barycenter of $K \cap H$. It implies that if $K_{[\delta]}$ exists, its boundary is given by the surface of all barycenters of $H \cap K$ for hyperplanes $H$ that cut off volume $\delta$ from $K$.

\begin{figure}[htpb]
\includegraphics[width=.45\textwidth]{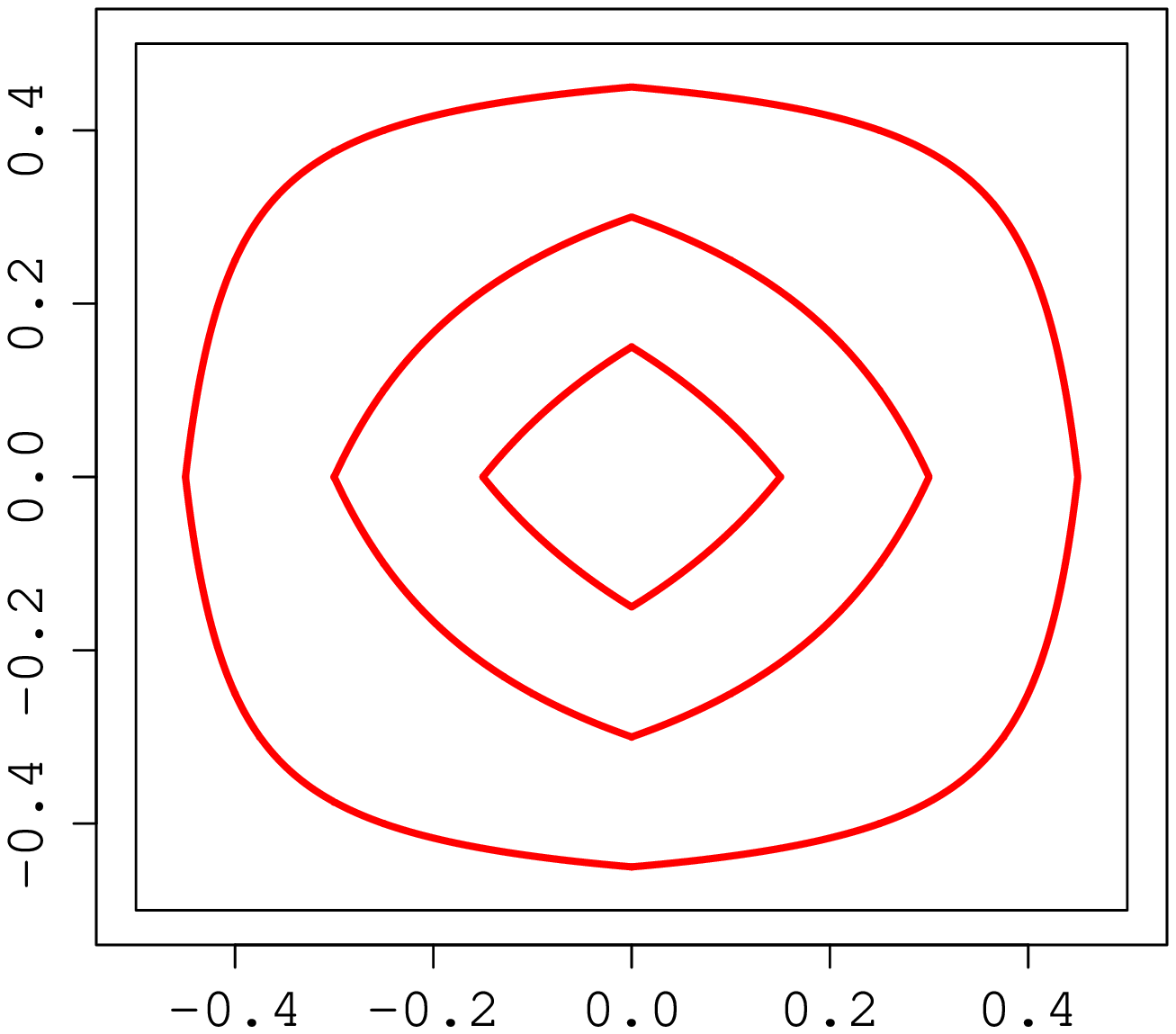} \quad \includegraphics[width=.45\textwidth]{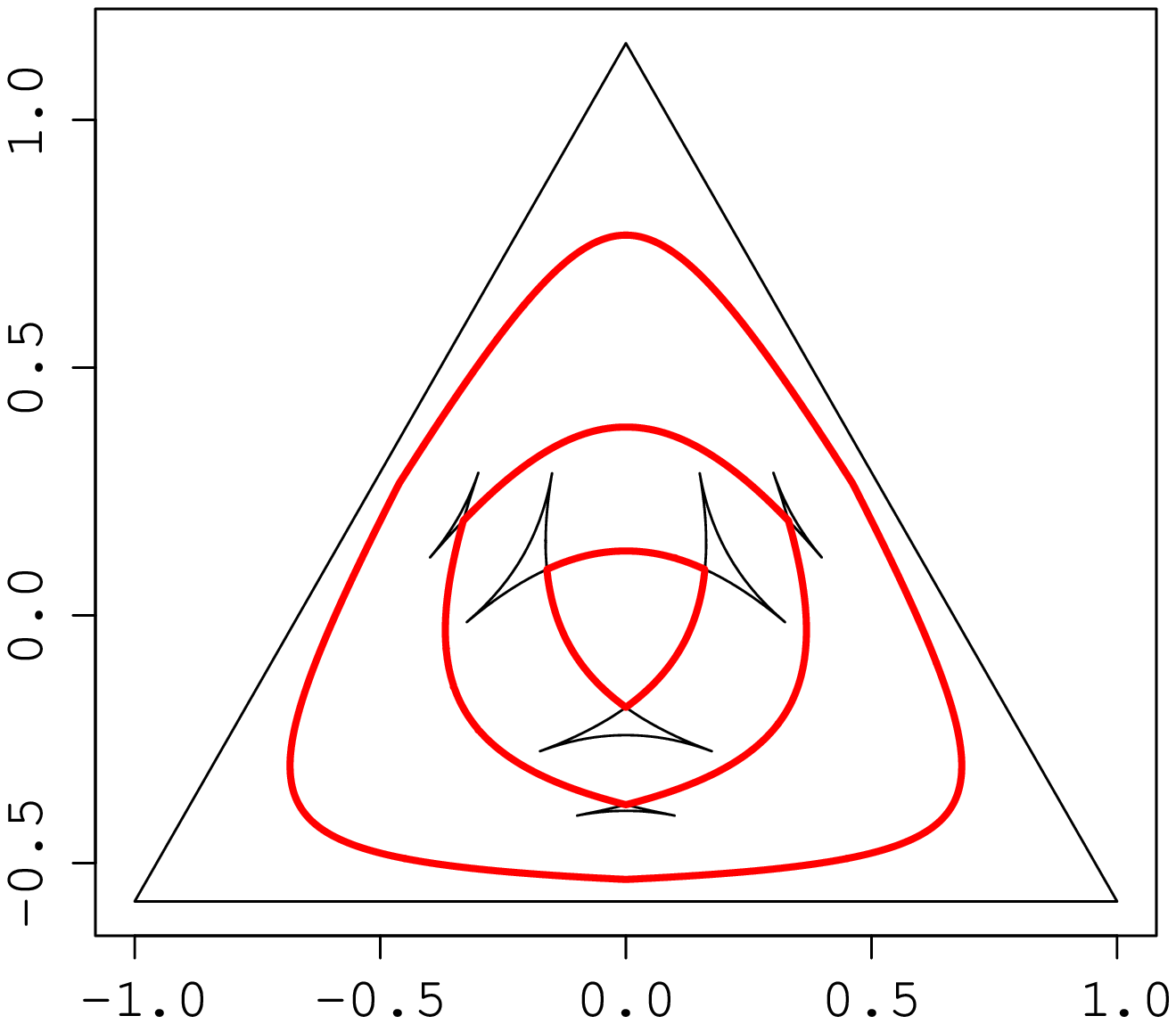}
\caption{Curves of barycenters of hyperplanes that cut off volume $\delta \in \{0.05,0.20,0.35\}$ from convex bodies (thin black lines), and the boundaries of convex floating bodies for the same values of $\delta$ (thick red lines), for the uniform distribution on a square (left panel) and a triangle (right panel). For the square, all (Dupin's) floating bodies exist, and coincide with the convex floating bodies. For the triangle, the boundaries of all convex floating bodies are proper sub-curves of the corresponding black curves (the difference is not visible in the plot for $\delta = 0.05$). The (Dupin's) floating bodies do not exist.}
\label{figure:Dupin floating bodies1}
\end{figure}


The floating body cannot exist for $\delta > \vol{K}/2$. Suppose it does exist. Then any two different parallel supporting hyperplanes of $K_{[\delta]}$ cut off disjoint sets of volume $\delta$ from $K$, and therefore $K_{[\delta]}$ is the empty set. As shown in the next example, the floating body $K_{[\delta]}$ may not exist even for small $\delta > 0$. 

\begin{example}	\label{example:floating bodies}
Let $K \in \CB[2]$ be the equilateral triangle from Example~\ref{example:depth level sets of convex bodies}. For all $\delta>0$, the curve of barycenters of lines that cut off volume $\delta$ from $K$ is not the boundary of a convex set. Some of these curves for various values of $\delta$ are displayed on the left panel of Figure~\ref{figure:Dupin floating bodies1}. Therefore, in agreement with the observation of \citet[pp. 433--434]{Leichtweiss1986}, no floating body of a triangle exists. Compare this also to Example~\ref{example:depth level sets of convex bodies}.

If  $K \in \CB[2]$  is the unit square of Example~\ref{example:depth level sets of convex bodies}, all floating bodies $K_{[\delta]}$ exist for $\delta \in (0,\vol{K}/2]$, and they coincide with the halfspace depth central regions \eqref{central region}.
\end{example}

If $K\in\CB$ has a sufficiently smooth boundary, then $K_{[\delta]}$ exists by \citet[Satz~2]{Leichtweiss1986}, at least for small $\delta>0$.  
However, in many applications (e.g., in Section~\ref{asa} below), existence of  floating bodies  
for all convex bodies is needed. Therefore a modified definition has been proposed, independently by \citet{Barany_Larman1988} and \citet{Schutt_Werner1990}, called the convex floating body.

\begin{definition}
Let $K$ be a convex body in $\mathbb R^d$ and $\delta\geq0$. The convex floating body is the intersection of all halfspaces whose defining hyperplanes cut off a set of volume
$\delta$ of $K$, 
	\begin{equation*}	
	K_{\delta} = \bigcap_{\vol{K\cap H^{-}}=\delta} H^{+},	
	\end{equation*}
where $H\in\mathcal H$ and $H^{+}$ and $H^{-}$ are its associated halfspaces.
\end{definition} 
The convex floating body exists for all convex bodies since it is an intersection of halfspaces. For instance, the convex floating body of the triangle has a boundary described by the red curve in Figure~\ref{figure:Dupin floating bodies1}. Note also that $K_0=K$. It is easy to see that whenever $K_{[\delta]}$ exists, then $K_{[\delta]} = K_\delta$ \citep{Schutt_Werner1990}. Unlike the floating body, the convex floating body is allowed to be an empty set. This way, all convex floating bodies $K_\delta$ of $K$ are well defined convex sets, but certainly $K_\delta = \emptyset$ if $\delta > \vol{K}/2$.

Properties of the convex floating body are stated in the next proposition.

\begin{proposition}\label{PropFloatBod1}
Let $K\in\CB$ and $\delta>0$. 
	\begin{enumerate}[label=(\roman*)]
	\item \label{existence of support} Through every point of $\partial K_{\delta}$ there is at least one supporting hyperplane of $K_{\delta}$ that cuts off a set of volume $\delta$ from $K$.
	\item \label{Dupin}	A supporting hyperplane $H$ of $K_{\delta}$ that cuts off a set of volume $\delta$ touches $K_{\delta}$ in exactly one point, the barycenter of $K\cap H$.
	\item \label{strict convexity} $K_{\delta}$ is strictly convex.
	\item \label{median} Let 
		\begin{equation}\label{delta0}	
		\delta_{0}=\sup\{\delta \colon \vol{K_{\delta}}>0	\}. 
		\end{equation}	
	Then $K_{\delta_{0}}$ consists of one point only and  for $\delta<\delta_{0}$ we have that $K_{\delta}$ is a convex body.
	\end{enumerate}
\end{proposition}

Most of Proposition~\ref{PropFloatBod1} was proved in \citep[Lemma~2]{Schutt_Werner1994}. Part~\ref{Dupin}, in dimension $d=3$, is due to \citet{Dupin1822}, see also \citep[p.~435]{Leichtweiss1986}. In general, it is not true that all supporting hyperplanes to the convex floating body $K_{\delta}$ cut off a set of exactly volume $\delta$ from $K$. An example is the simplex, as can be seen also from Example~\ref{example:floating bodies}. Not every point on the boundary of $K_\delta$ has a unique supporting hyperplane. An example is the cube, see Example~\ref{example:floating bodies}.

\citet{Meyer_Reisner1991} show that for centrally symmetric convex bodies $K_{[\delta]}$ exists for any $\delta\in(0,\vol{K}/2]$. Moreover, in that case each $K_{[\delta]}$ is also (centrally) symmetric around the same center of symmetry as $K$. In an unpublished work, K.~Ball gave a different proof of the existence result, see \citep[Section~4]{Meyer_Reisner1991B}. 

\begin{proposition}\label{FloatSymm}
Let $K\in\CB$ be a convex body that is (centrally) symmetric with respect to the origin $0$, i.e. $x\in K$ implies $-x\in K$. Then we have for all $\delta \in (0,\vol{K}/2)$
	\begin{enumerate}[label=(\roman*)]
	\item The floating body of $K$ exists.
	\item \label{C2 differentiability} For all convex bodies $K$ with $C^1$ boundary and all $\delta$ the floating body $K_{\delta}$ has a $C^{2}$ boundary.
	\end{enumerate}
\end{proposition}

The next two results can be found in \citet[Theorem~5.3 and Proposition~5.1]{Schutt_Werner1992}, and describe the behavior of the volume of $K\setminus K_{\delta}$.

\begin{proposition}\label{DiffFloat1}
Let $K\in\CB$, and let $\delta_{0}$ be as in \eqref{delta0}. Then $\vol{K\setminus K_{\delta}}$ is a differentiable function of $\delta$ on $(0,\delta_{0})$ and 
	\[	\frac{\dd}{\dd \delta}\vol{K\setminus K_{\delta}}=\int_{\partial K_{\delta}} \frac{1}{\vol[d-1]{K\cap H(x,N_{\partial K_{\delta}}(x))}}	\dd \mu_{\partial K_{\delta}}(x),	\]
where $H(x,N_{\partial K_{\delta}}(x))$ is the hyperplane passing through $x$ orthogonal to the normal of $K_\delta$ at $x$. 
\end{proposition} 

\begin{proposition}\label{DiffFloat3}
Let $K\in\CB$ be a (centrally) symmetric convex body in $\R^{d}$. Then we have for all $\delta\in(0,\vol{K}/2)$
	\[	\frac{\dd}{\dd \delta}\vol{K\setminus K_{\delta}} \leq \frac{d}{\delta}\vol{K\setminus K_{\delta}}.	\]
\end{proposition}

\subsection{Affine surface area} \label{asa}


An important affine invariant from affine convex  geometry is the affine surface area.  Applications of the affine surface area are numerous. We only name some in convex geometry \citep{LutwakZhang1997, GardZhang1998, Boe2010, Ludwig2010, Ludwig_Reitzner2010, Haberl_Parapatits2014}, in 
differential geometry \citep{Andrews1996, Andrews1999, Stancu2003, Mohammad_Stancu2013}, approximation of convex bodies by polytopes (see Section \ref{ApproxCoBodPol}), information theory \citep{LYZ2004, Werner2012, AKSW2012, Werner2013}, and partial differential equations \citep{Lutwak_Oliker1995, Trudinger_Wang2005}.

Let $K$ be a convex body in $\R^d$ with a $C^2$ boundary. Then  for all $x \in \partial K$, the Gau{\ss}-Kronecker curvature $\kappa(x)$ exists and  the 
(classical) affine surface area, introduced by \citet{Blaschke1923} in 1923 in dimensions two and three, is defined as
	\begin{equation*}	
	\asa{K} = \int_{\partial K} \kappa(x)^{\frac{1}{d+1}} \dd \mu_{\partial K}(x).	
	\end{equation*}
For a Euclidean ball with radius $1$, the affine surface area equals its surface area. It is $0$ for all polytopes. 
 \citet{Blaschke1923} observed that for convex bodies in $\R^3$ with analytic boundary the following identity holds
\begin{equation}\label{asa-Blaschke}
	\lim_{\delta \to 0} \frac{\text{vol}_3(K)-\text{vol}_3(K_{[\delta]})}{\delta^{\frac{1}{2}}}=\frac{1}{\sqrt{\pi}}
\int_{\partial K}\kappa(x)^{\frac{1}{4}} \dd \mu_{\partial K}(x).
	\end{equation}
An important tool in the proof of this identity is the rolling theorem of \citet{Blaschke1923}: The floating body exists if a sufficiently small Euclidean ball rolls freely inside $K$, i.e., there is $r>0$
such that for all $x \in \partial K$ there is $y \in K$ such that $\|x-y\|=r$ and $\B(y, r) \subset K$.

It is natural to ask if formula \eqref{asa-Blaschke} can be extended to all dimensions and all convex bodies using the convex floating body instead of the floating body. This is indeed the case and was achieved in \citet{Schutt_Werner1990}, where now the function $\kappa$ under the integral is the generalized Gau{\ss}-Kronecker curvature \eqref{DefGenCurv2}.

\begin{theorem} \label{FloatASA1}
Let $K\in\CB$. Then
	\begin{equation}\label{FloatASA1-1}
	\lim_{\delta \to 0} \frac{\vol{K}-\vol{K_{\delta}}}{\delta^{\frac{2}{d+1}}}=\frac{1}{2}\left(\frac{d+1}{\vol[d-1]{\B[d-1]}}\right)^{\frac{2}{d+1}}
\int_{\partial K}\kappa(x)^{\frac{1}{d+1}} \dd \mu_{\partial K}(x).
	\end{equation}
\end{theorem}

The expressions in the above theorem can thus be used to define the affine surface area for all convex bodies. Around the same time, different  extensions of the affine surface area to arbitrary convex bodies were given by \citet{Leichtweiss1986} and \citet{Lutwak1991} and afterwards several more have been found, e.g., \citep{Huang_etal2018, Meyer_Werner2000, Werner1996}. It has been shown that all those extensions coincide.

Expression \eqref{FloatASA1-1} is called the affine surface area because of its similarity to  Min\-kow\-ski's definition of  surface area 
	\[	\vol[d-1]{\partial K} = \lim_{\delta\to 0}\frac{1}{\delta}\left(\vol{K+\delta \B}-\vol{K}\right),	\]
and because for all affine maps $T \colon \R^{d}\to\R^{d}$, 
	\(	\asa{T(K)} = \left\vert \det(T) \right\vert^{\frac{d-1}{d+1}}\asa{K}.	\)
The latter equation follows easily from \eqref{FloatASA1-1}. Indeed,
	\begin{equation*}
	(T(K))_{\delta}=T\left(K_{\frac{\delta}{\left\vert \det(T) \right\vert}}\right).
	\end{equation*}


An important tool in the proof of  Theorem \ref{FloatASA1} is a strengthening of Blaschke's rolling theorem.  
To achieve this, \citet{Schutt_Werner1990} introduce the rolling function. For $x \in \partial K$, the rolling function $r(x)$ is the supremum of all radii of Euclidean balls that contain $x$ and that are contained in $K$, i.e. $r \colon \partial K\to\R$ is defined by
	\begin{equation*}
	r(x)=\sup\left\{\left\Vert x-z \right\Vert \colon z\in K, \B(z,\|x-z\|)\subseteq K\right\}.
	\end{equation*}
If $K$ does not have a unique normal at $x$ then $r(x)=0$. The following was shown by \citet[Lemmas~4 and~5]{Schutt_Werner1990}.

\begin{proposition}\label{SW1}
Let $K\in\CB$ be such that $\B \subset K$. Then we have for all $t$ with $0 \leq t \leq 1$ that $\{x \in \partial K \colon r(x) \geq t\}$ is a closed set and
	\[	(1-t)^{d-1}\vol[d-1]{\partial K} \leq \vol[d-1]{\{x \in \partial K \colon r(x) \geq t\}}.	\]
The inequality is optimal.
In particular, the function $r^{-\alpha}\colon \partial K\to\R$ is Lebesgue integrable for all $\alpha$ with $0\leq\alpha<1$.
\end{proposition}

Note that by taking $t=0$ in Proposition~\ref{SW1} it follows that the boundary of a  convex body is almost everywhere differentiable.


Affine invariance is a useful property as it lets us  consider convex  bodies independent of their position in space. Another extremely important property of the affine surface area is the affine isoperimetric inequality which says that for all convex bodies $K\in\CB$,
	\begin{equation}	\label{AffIsoIneq1}
	\frac{\asa{K}}{\asa{\B}} \leq \left(\frac{\vol{K}}{\vol{\B}}\right)^{\frac{d-1}{d+1}},
	\end{equation}
with equality if and only if $K$ is an ellipsoid (see, e.g., \citep[Section~10.5]{Schneider2014}). The affine isoperimetric inequality is stronger than the classical isoperimetric inequality and provides solutions to many problems where ellipsoids are extrema \citep{Lutwak1993, Schutt_Werner1994, Stancu2006, Werner_Ye2011}.

The affine isoperimetric inequality \eqref{AffIsoIneq1} is equivalent to another classical inequality from convex geometry, the Blaschke-Santal\'o inequality \citep{Blaschke1923, Santalo1949}. For an interior point $x_{0}$ of a convex body $K$ recall the definition of the polar body $K^{x_{0}}$ of $K$ w.r.t. $x_{0}$ from \eqref{polar}. The Blaschke-Santal\'o inequality states that for all convex bodies $K$ in $\R^{d}$, 
	\begin{equation}	\label{Blaschke-Santalo}
	\vol{K}\vol{K^{s(K)}} \leq \vol{\B}^{2},	
	\end{equation}
where $s(K)$ is the Santal\'o point of $K$, i.e. the unique point for which the minimum is attained on the left hand side. This inequality and its counterpart, the reverse Blaschke-Santal\'o inequality (proved by \citet{Bourgain_Milman1987} and closely connected to the still-unsolved Mahler's conjecture, see e.g. \citet{Giannopoulos_etal2014}), are helpful to estimate  the volume of convex bodies in situations, when it is easier to compute the volume of the polar $K^{x_{0}}$ of a convex body. These inequalities have important applications in convex geometry, functional analysis, Banach space theory, quantum information theory, operator theory and geometric number theory. For background including references,
see e.g.,  the books \citep{AGM2015, Gardner2006, Giannopoulos_etal2014, Koldobsky2005, Schneider2014}.

To conclude this section, note that for a polytope $S \in \CB$ we have a different behavior of the volume difference $\vol{S \setminus S_{\delta}}$ than that from Theorem~\ref{FloatASA1}. To describe it, we need the notion of flag. A flag of a polytope $S$ is a $d$-tuple $(f_{0},\dots,f_{d-1})$ where $f_i$ is an $i$-dimensional face of $S$ with $f_{i}\subset f_{i+1}$. $\operatorname{fl}_{d}(S)$ denotes the number of flags of the polytope $S$.

\begin{theorem}\label{PolyFloatBod1}
Let $S$ be a convex polytope with nonempty interior in $\R^d$. Then
	\[	\lim_{\delta\to0}\frac{\vol{S}-\vol{S_{\delta}}}{\delta\left(\log\frac{1}{\delta}\right)^{d-1}}=\frac{\operatorname{fl}_{d}(S)}{d!\,d^{d-1}}.	\]
\end{theorem}
Theorem~\ref{PolyFloatBod1} was proved by \citet[Theorem~1.2]{Schutt1991}. Recent extensions of this theorem can be found in~\citet{Besau_etal2018}.

\subsection{\texorpdfstring{$L_p$}{Lp}-affine surface area} 

The concept of affine surface area for convex bodies has been generalized to $L_p$-affine surface areas. Those are by now the cornerstones of the 
rapidly developing {\em $L_p$-Brunn-Minkowski theory}, initiated in the groundbreaking paper of \citet{Lutwak1996}. See also \citep[Section~9.1]{Schneider2014} and, e.g., \citep{Paouris_Werner2012, HaberlSchuster2009, LuYaZh2000,Meyer_Werner2000}. The next definition was given by \citet{Lutwak1996} for $p>1$, and \citet{Schutt_Werner2004} for all other $p$. See also \citet{Hug1996}.

\begin{definition}
Let $K$ be a convex body in $\R^d$ such that $0$ is in the interior of $K$. Let  $- \infty \leq p \leq \infty$, $p \neq -d$. The $L_p$-affine surface area of $K$ is
	\begin{equation}\label{asp-K}
	\asa[p]{K} = \int_{\partial K} \frac{ \kappa(x)^\frac{p}{d+p}}{\left\langle x, N_K(x) \right\rangle^\frac{d(p-1)}{d+p}} \dd\mu_{\partial K}(x).
	\end{equation}
Here, $N_K(x)$ is the outer unit normal at $x \in \partial K$, $\mu_{\partial K}$ is the usual surface area measure on $\partial K$ and $\kappa$ is the generalized Gau{\ss}-Kronecker curvature at $x$. 
\end{definition}

For $p=0$, $\asa[0]{K} = d \vol{K}$. For $p = \pm \infty$, the $L_p$-affine surface area is defined by the corresponding limit in \eqref{asp-K} 
\begin{equation*}
	\asa[\pm \infty]{K} = \int_{\partial K} \frac{ \kappa(x) }{\left\langle x, N_K(x) \right\rangle^d} \dd\mu_{\partial K}(x), 
\end{equation*}
which, for $K$ sufficiently smooth, gives $\asa[\pm \infty]{K} = d \vol{K^\circ}$, where $K^\circ$ is the polar body \eqref{polar} of $K$ w.r.t. $0$. For $p=1$ we get the above mentioned affine surface area of $K$,
\begin{equation*}
\asa[1]{K} = \asa{K} = \int _{\partial K } \kappa(x) ^\frac{1}{d+1}  \dd \mu_{\partial K}(x).
\end{equation*}
Note that in general the $L_p$-affine surface area  is not an affine invariant anymore, only a linear invariant. There exist geometric identities, analogous to \eqref{FloatASA1-1}, also for $L_p$-affine surface area. These use 
{\em weighted floating bodies} \citep{Werner2002}, {\em Santal\'o bodies} \citep{Meyer_Werner2000} and 
{\em surface bodies} \citep{Schutt_Werner2004}. We refer to those references for the details. Moreover, the corresponding $L_p$-affine isoperimetric inequalities hold true as well. 

\begin{theorem} \label{pasainequalities}
Let $K\in\CB$ with the origin in its interior.
	\begin{enumerate}[label=(\roman*)]
	\item If $p\geq 0$, then
		\begin{eqnarray*}
		\frac{\asa[p]{K}}{\asa[p]{\B}} \leq \left(\frac{\vol{K}}{\vol{\B}}\right)^{\frac{d-p}{d+p}}.
		\end{eqnarray*}
	\item If $-d<p<0$, then
		\begin{eqnarray*}
		\frac{\asa[p]{K}}{\asa[p]{\B}} \geq \left(\frac{\vol{K}}{\vol{\B}}\right)^{\frac{d-p}{d+p}}.
		\end{eqnarray*}		
	Equality holds in (i) and (ii) if and only if $K$ is an ellipsoid.
	\item If  $K$  in addition has  $C^2$ boundary with strictly positive Gau{\ss}-Kronecker curvature everywhere and if $p < -d$, then
		\begin{equation*}
		c^{\frac{d p}{d+p}} \left(\frac{\vol{K}}{\vol{\B}}\right)^{\frac{d-p}{d+p}} \leq \frac{\asa[p]{K}}{\asa[p]{\B}}.
		\end{equation*}
	The constant $c$ in (iii) is the constant from the reverse Blaschke-Santal\'o inequality due to \citet[Theorem~1]{Bourgain_Milman1987}.
	\end{enumerate}
\end{theorem}

Theorem~\ref{pasainequalities}
was proved by \citet{Lutwak1996} for $p >1$ and by \citet[Theorem~4.2]{Werner_Ye2008} for all other $p$.

\subsection {Floating measures} \label{FLME}
Much effort has been devoted to extend the theory of convex bodies  to a functional setting (e.g., \citep{Ball1986, Artstein_etal2004, Fradelizi_Meyer2007}).  Natural analogs of convex bodies in the realm of functions are log-concave functions, i.e. densities of log-concave measures. For such measures we present a notion of {\em floating measure}. Another approach will be shown in Section~\ref{section:floating bodies for measures}. 


Let $\psi\colon \R^{d} \rightarrow \R$ be a convex function such that 
	\begin{equation}	\label{psi integrability condition}
	0 < \int _{\R^{d}} \Euler^{-\psi(x)} \dd x < \infty.	
	\end{equation}
In the general case, when $\psi$ is neither smooth nor strictly convex, the gradient of $\psi$, denoted  by $\nabla \psi$, exists almost everywhere by Rademacher's theorem \citep[Theorem~2.5.1]{Borwein_Vanderwerff2010}. A theorem of \citet{Busemann_Feller1936} and \citet{Aleksandrov1939} guarantees the existence of the (generalized) Hessian, denoted by $\nabla^2 \psi$, almost everywhere in $\R^{d}$ (for details see, e.g., \citep[Section~1.6]{Schutt_Werner2003}). 
The Hessian is a quadratic form on $\R^{d}$, and if $\psi$ is a convex function, for almost every $x \in \R^{d}$ one has, when $y \rightarrow 0$, that 
	\[	\psi( x + y) = \psi (x) + \langle \nabla \psi(x),  y  \rangle + \frac{1}{2}  \langle \nabla^2 \psi(x) (y), y \rangle + o( \|y\|^2).	\]
Let $\mu$ be a log-concave measure on $\R^{d}$, i.e. a measure with density $\Euler^{-\psi}$, where $\psi \colon \R^{d} \to \R$ is a convex function. 
Note that we do not necessarily require that $\mu$ is a probability measure. Let
	\[	\operatorname{epi}(\psi) = \{ (x,y) \in \R^{d} \times \R \colon y \geq \psi(x)\}	\]
be the  epigraph of $\psi$. Then $\operatorname{epi}(\psi)$  is a  closed convex set in $\R^{d+1}$ and for sufficiently small $\delta$ we can define its floating set $\operatorname{epi}(\psi)_\delta$ as
	\begin{equation*} 
	\operatorname{epi}(\psi)_\delta=\bigcap_{\{H \in \mathcal H \colon \vol{H^-\cap \, \operatorname{epi}(\psi)} \leq \delta \}} {H^+}. 
	\end{equation*}
This was done in \citep{Li_etal2018}, where also the definition of a floating set was introduced for convex, not necessarily bounded subsets of $\R^d$.

It is easy to see that there exists a unique convex function $\psi_\delta \colon \R^d \rightarrow \R$ such that $(\operatorname{epi}(\psi))_\delta =\operatorname{epi}(\psi_\delta)$. Consequently, \citet{Li_etal2018} define the floating function of a convex function $\psi$ and the floating measure of the (not necessarily probability) measure $\mu$ as follows.

\begin{definition} 
Let $\psi \colon \R^d \rightarrow \R$ be a  convex function. Let $\delta >0$.
	\begin{enumerate}[label=(\roman*)]
	\item The floating function of $\psi$ is defined to be the function $\psi_\delta$ such that
		\begin{equation*}
		(\operatorname{epi}(\psi))_\delta = \operatorname{epi}\left(\psi_\delta\right).
		\end{equation*} 
	\item Let $\mu$ be a measure with density $f(x)= \Euler^{-\psi (x)}$. The floating measure of $\mu$ is the measure with density $f_\delta$ where
		\begin{equation*}
		f_\delta (x) = \Euler^{-\psi_\delta (x)} \quad\mbox{ for }x\in\R^d.
		\end{equation*}
	\end{enumerate}
\end{definition}

Note that when $\psi$ is affine, $\psi_\delta=\psi$ and, for $f=\Euler^{-\psi}$, $f_\delta=f$.

\subsection{Affine surface areas for log-concave measures}	\label{section:Asa for measures}

As far as we know, at present there are two approaches for a definition of affine surface area for log-concave measures. The first one is similar to the one discussed in Section~\ref{asa} and uses the floating measure of Section~\ref{FLME} instead of the floating bodies $K_\delta$. It was proposed in \citep{Li_etal2018} and is inspired by the formula of Theorem~\ref{FloatASA1}. As in Section~\ref{FLME}, we do not require that the log-concave measure $\mu$ with density $\Euler^{-\psi}$ is a probability measure.

\begin{theorem}\label{theo:f-deltafloat}
Let  $\psi: \R^{d} \rightarrow \R $ be a convex function such that \eqref{psi integrability condition} holds true. Then 
	\begin{equation*}
	\lim _{\delta \rightarrow 0} \frac{  \int_{\R^d} \Euler^{-\psi (x)}  \dd x - \int_{\R^d} \Euler^{-\psi_{\delta}(x)}  \dd x} {\delta^{2/(d+2)}} =   \frac{1}{2} \left(\frac{d+2}{\vol{\B}}\right)^\frac{2}{d+2} \int_{\R^{d}} \left(\det\left( \nabla^2 \psi (x) \right)\right)^\frac{1}{d+2} \  \Euler^{-\psi(x)} \dd x.
	\end{equation*}
	\end{theorem}
	
This theorem was proved in \citep[Theorem~1]{Li_etal2018}. Its comparison with convex bodies (see Theorem~\ref{FloatASA1}) led \citet{Li_etal2018} to call the right hand side integral of Theorem~\ref{theo:f-deltafloat} the affine surface area of the measure $\mu$.

	\begin{definition}
	For a log-concave measure $\mu$ on $\R^d$ with density $\Euler^{-\psi}$ such that \eqref{psi integrability condition} holds true, the affine surface area of the measure $\mu$ is given by
	\begin{equation}\label{Pasa}
	\asa{\mu} = \int_{\R^{d}} \left(\det\left( \nabla^2 \psi (x)\right)\right)^\frac{1}{d+2} \  \Euler^{-\psi(x)} \dd x.
	\end{equation}
	\end{definition} 

This definition is further justified as the expression shares many properties of the affine surface area for convex bodies. For instance, it is invariant under affine transformations with determinant $1$. For the standard Gau{\ss}ian measure $P$ we have that $\asa{P} = 1$.

Another definition of affine surface area for log-concave measures was put forward in \citet{Calgar_etal2016}. Actually, an even more general approach was proposed, again for convex functions $\psi$ such that \eqref{psi integrability condition} holds true. We put $\Omega_{\psi}$ to be the set of vectors in $\R^d$ at which $\nabla^2\psi$ exists and is invertible. 

	\begin{definition}
	For a log-concave measure $\mu$ on $\R^d$ with density $\Euler^{-\psi}$ such that \eqref{psi integrability condition} holds true and $\lambda \in \R$, the $\lambda$-affine surface areas are 
		\begin{equation}\label{lPasa}
		\Asa[\lambda]{\mu} =  \int_{\Omega_\psi} \Euler^{\lambda \left( 2 \psi(x)- \langle x, \nabla\psi(x)\rangle\right)}\left(\det \, \nabla^2 \psi (x)\right)^\lambda \Euler^{- \psi(x) } \dd x.
		\end{equation}
	We can replace $\Omega_\psi$ by $\R^d$ for $\lambda >0$.
	\end{definition}
	

Differentiating with respect to $\alpha$ at $\alpha=1$, we get in  the case of $2$-homogeneous convex functions $\psi$, that is $\psi(\alpha x) = \alpha^2 \psi(x)$,  for any $\alpha>0$ and $x \in \R^d$, 
that
	\[	\langle x, \nabla \psi(x) \rangle = 2 \psi(x).	\] 
Thus, for $2$-homogeneous functions $\psi$, formula \eqref{lPasa} simplifies to 
	\begin{equation*}
	\Asa[\lambda]{\psi}=\int_{\Omega_\psi}\left(\det \, \nabla^2  \psi (x)\right)^\lambda \Euler^{-\psi(x)} \dd x, 
	\end{equation*}
and definitions \eqref{Pasa} and \eqref{lPasa} agree for $\lambda=\frac{1}{d+2}$.

To understand why it is justified to name the quantities \eqref{lPasa} affine surface areas, we recall the definition of the $L_p$-affine surface areas 
\eqref{asp-K} for convex bodies $K$.
It was noted in \citet{Calgar_etal2016} that the definition of $\lambda$-affine surface area for a log-concave density agrees with the definition of $L_p$-affine surface area for convex bodies if the function is the gauge function $\| \cdot \|_K$ of a convex body $K$ with $0$ in its interior, 
	\[	\|x\|_K = \min\{ \alpha\geq 0 \colon \ x \in \alpha K\}  .	\]
	
The next theorem is from \citet[Theorem~3]{Calgar_etal2016}.
	
\begin{theorem}
Let $K$ be a convex body in $\R^d$ that contains the origin in its interior. For any $p \geq 0$, let $\lambda=\frac{p}{d+p}$. Then 
	\[	\Asa[\lambda]{\frac{\|\cdot\|_K^2}{2}} = \frac{(2\pi)^\frac{d}{2}}{d \vol{\B}} \asa[p]{K}.	\]
Moreover, if the set of points of $\partial K$ where the generalized Gau{\ss}-Kronecker curvature is strictly positive has full measure in $\partial K$, then the same relation holds true for every $p \ne -d$.
\end{theorem}

The $L_p$-affine isoperimetric inequalities for convex bodies of Theorem \ref{pasainequalities} have analogs for the $\lambda$-affine surface areas for log-concave measures. 
We only mention the case $\lambda \in [0,1]$ and refer to \citet{Calgar_etal2016} for the other cases.

	\begin{proposition}
	Let $\psi \colon \R^d\to \R\cup\{+\infty\}$ be a convex function such that \eqref{psi integrability condition} holds true and such that $\int _{\R^d} x \Euler^{-\psi(x)} \dd x = 0$. Then we have for all $\lambda \in [0,1]$, 
		\[
		\frac{\Asa[\lambda]{\mu}}{\Asa[\lambda]{\frac{\|\cdot\|^2}{2}}} \leq \left( \frac{\int_{\Omega_\psi} \Euler^{-\psi(x)} \dd x}{\int_{\R^d} \Euler^{-\frac{\|x\|^2}{2}} \dd x}\right)^{1-2\lambda}.
		\]
	In particular, if $\psi$ is in addition $2$-homogeneous, then
		\[	
		\frac{\asa{\mu}}{\asa{\frac{\|\cdot\|^2}{2}}} \leq \left( \frac{\int_{\Omega_\psi} \Euler^{-\psi(x)} \dd x}{\int_{\R^d} \Euler^{-\frac{\|x\|^2}{2}} \dd x}\right)^{\frac{d}{d+2}}.	
		\]
	Equality holds in the inequalities if and only if there are $a\in\R$ and a positive definite matrix $A$ such that for all $x\in\R^{d}$
		\[	\psi(x)=\langle Ax,x\rangle + a.	\]
	\end{proposition}
A main ingredient in the proof of this proposition is a functional version of the Blaschke-Santal\'o inequality. We refer to \citep{Ball1986, Artstein_etal2004, Fradelizi_Meyer2007} for the details.

\subsection{Applications of affine surface area: Approximation of convex bodies by polytopes}	\label{ApproxCoBodPol}

Approximation by polytopes is a central topic in convex geometry with numerous applications. There is a huge amount of literature on the subject. A (very incomplete) list is \citep{Boe2000, Gruber1993B,Reitzner2005, Schue1994, GroWer2018,HoehnerSchuettWerner2018,LuSchW2006}. We present only one aspect of the subject, approximation by polytopes with a fixed number of vertices and refer to the literature for others.

\subsubsection{Best and random approximation}

Ideally, in approximation problems, one seeks a best approximating polytope in a given metric. One such result is given in the next theorem, where we consider all polytopes with at most $N$ vertices that are contained in a convex body $K$. By compactness, there is a polytope $P_{N}$ in this class with maximal volume. This means that the symmetric difference metric $\dS(K,P_{N})$ from \eqref{SymDiff} is minimal. Such a polytope is called best approximating with respect to the symmetric difference metric.

	\begin{theorem}\label{Best}
  Let $K$ be a convex body in $\R^d$ with $C^2$-boundary $\partial K$ and everywhere strictly positive Gau{\ss}-Kronecker curvature $\kappa$. For every $N\in\N$ let $P_{N}$ be a best approximating polytope of $K$ with at most $N$ vertices. Then 
	\begin{equation}\label{BestApprox-1}
	\lim_{N \to \infty}\frac{\dS(K,P_N)}{\left(\frac{1}{N}\right)^\frac{2}{d-1}} =\frac{1}{2} \operatorname{del}_{d-1} \left(\int_{\partial K}
\kappa(x)^{\frac{1}{d+1}} \dd \mu_{\partial K}(x)\right)^\frac{d+1}{d-1},
	\end{equation}
where $\operatorname{del}_{d-1}$ is a constant depending only on the dimension $d$.
	\end{theorem}

This theorem was proved by \citet{McClure_Vitale1975} in dimension $2$ and by \citet{Gruber1993B} for general dimension. It was shown by \citet{Man_Schue2001} that $\operatorname{del}_{d-1}$ is of the order of dimension, or more precisely, 
	\begin{equation}\label{ManSch}
	\frac{d-1}{d+1}\vol[d-1]{\B[d-1]}^{-\frac{2}{d-1}} \leq \operatorname{del}_{d-1} \leq \frac{d-1}{d+1}\vol[d-1]{\B[d-1]}^{-\frac{2}{d-1}} \frac{\Gamma\left(d+1+\frac{2}{d-1}\right)}{d!}.
	\end{equation}
Note that $\frac{\Gamma\left(d+1+\frac{2}{d-1}\right)}{d!} \leq 1+ c \frac{\log d}{d}$, where $c$ is an absolute constant.	
	
On the right hand side of equation \eqref{BestApprox-1} we find the affine surface area of $K$ from Section~\ref{asa}. It is natural that such a term should appear in approximation questions: Intuitively, we expect that more vertices of the approximating polytope should be put where the boundary of $K$ is very curved, and fewer points where the boundary is flat, to get a good approximation in the $\dS$-metric. 

However it is only in rare cases that a best approximating polytope can be singled out. Consequently, a common practice is to randomize: 
Choose $N$ points at random in $K$ with respect to a probability measure $P$ on $K$. The convex hull of these randomly chosen points is a random
polytope. The expected volume of a random polytope of $N$ points is
	\[	E(K,N) = \int_{K} \cdots \int_{K} \vol{[x_{1},\dots,x_{N}]} \dd P(x_{1})\dots \dd P(x_{N}),	\]
where $[x_{1},\dots,x_{N}]$ is the convex hull of the points $x_{1},\dots,x_{N}$. Thus the expression $\vol{K} -  E(K,N)$ measures how close a random polytope and the convex body are in the symmetric difference metric.

We now compare best approximation with random approximation. The analog to Theorem~\ref{Best} in the random case is the following theorem. There, the probability measure is the normalized Lebesgue measure on $K$.

	\begin{theorem} \label{inK}
	Let $K$ be a convex body in $\R^d$. Then  
		\[	\lim_{N \to \infty}\frac{\vol{K}- E(K,N)}{\left(\frac{\vol{K}}{N}\right)^\frac{2}{d+1}} = c(d) \int_{\partial K} \kappa(x)^{\frac{1}{d+1}}
	\dd \mu_{\partial K}(x),	\]
	where $c(d)$ is a constant that depends only on $d$.
	\end{theorem}

This theorem was proved by \citet{Ren_Sulanke1, Ren_Sulanke2} in  dimension $2$. \citet{Wie1} settled the case of the Euclidean ball in dimension $d$. \citet{Barany1992} proved the result for convex bodies with $C^{3}$-boundary and everywhere positive Gau{\ss}-Kronecker curvature. Finally, the general result for arbitrary convex bodies was proved by \citet{Schue1994} and \citet{BoerFeHu2010}.

Notice that Theorem~\ref{inK} does not give the optimal dependence on $N$ for best approximation. One reason is that not all the points chosen at random from $K$ appear as vertices of the approximating random polytope. Thus we now choose the points randomly from the boundary of $K$ according to a measure with a density with respect to $\mu_{\partial K}$. We denote by ${E}(K,f,N)$ the expected volume of the corresponding random polytope. Which density is optimal? It turns out that it is, up to normalization, the $(d+1)$-root of the generalized Gau{\ss}-Kronecker curvature. The integral of this function is the affine surface area. The next theorem was shown by \citet[Theorem~1.1]{Schutt_Werner2003}, see also \citet{Reitzner2002}.

	\begin{theorem} \label{Random}
	Let $K$ be a convex body in $\R^d$ such that there are $0<r\leq R < \infty$ so that we have for all $x\in\partial K$
		\[	\B(x-r N_{K}(x),r)\subseteq K \subseteq \B(x-R N_{K}(x),R),	\]
	for $N_K(x)$ an outer unit normal of $K$ at $x$, and let $f \colon \partial K\rightarrow (0,\infty)$ be continuous with $\int_{\partial K} f(x) \dd \mu_{\partial K}(x)=1$. 
	Then
		\begin{equation}	\label{equation random}
		\lim_{N \to \infty} \frac{\vol{K} - {E}(K,f,N)}{\left(\frac{1}{N}\right)^\frac{2}{d-1}} = c(d)\int_{\partial K} \frac{\kappa(x)^{\frac{1}{d-1}}}{f(x)^{\frac{2}{d-1}}} \dd \mu_{\partial K}(x)	
		\end{equation}
	where
		\[	c(d) = \frac{(d-1)^{\frac{d+1}{d-1}} \Gamma\left(d+1+\tfrac{2}{d-1}\right)}{2(d+1)!(\vol[d-2]{\partial \B[d-1]})^{\frac{2}{d-1}}}.	\]
	The minimum at the right-hand side of \eqref{equation random} is attained for the normalized affine surface area measure with density
		\[	f_{as}(x) = \frac{\kappa(x)^{\frac{1}{d+1}}}{\int_{\partial K}\kappa(x)^{\frac{1}{d+1}} \dd \mu_{\partial K}(x)} \quad \mbox{for }x\in\partial K.	\]
	\end{theorem}

Best approximation of Theorem~\ref{Best} differs from random approximation of Theorem~\ref{Random} only in the dimensional constants $\operatorname{del}_{d-1}$ and $c(d)$. Comparing those, using also \eqref{ManSch}, an amazing fact follows: with the density $f_{as}$ random approximation is almost as good as best approximation, 
	\[	\lim_{N \to \infty}\frac{\dS(K,P_N)}{\left(\frac{1}{N}\right)^\frac{2}{d-1}} \leq \lim_{N \to \infty} \frac{\vol{K}- {E}(K,f_{as},N)}{\left(\frac{1}{N}\right)^\frac{2}{d-1}} \leq \left(1+c\frac{\log d}{d}\right)\lim_{N \to \infty}\frac{\dS(K,P_N)} 
{\left(\frac{1}{N}\right)^\frac{2}{d-1}},	\]
where $c$ is an absolute constant. 

\subsubsection{The floating body algorithm}

\citet[Theorem~1]{Barany_Larman1988} established a relation between floating bodies and random polytopes for the uniform measure on the convex body: With high probability the volume of a random polytope is close to the volume of an appropriate floating body.

	\begin{theorem}
	Let $K$ be a convex body in $\R^d$. Then there is $N_{0}\in\N$ such that for all $N\geq N_{0}$
		\[	c_1 \left(\vol{K} - \vol{K_{\frac{1}{N}\vol{K}}} \right)	\leq\vol{K} - E(K,N) \leq c_{2}\left(\vol{K} - \vol{K_{\frac{1}{N}\vol{K}}}\right),	\]	 	 
	where $c_1$ and $c_2$ are constants that depend on $d$ only.
	\end{theorem}

Even more can be said about the connection between floating bodies and random polytopes. There is an algorithm, the floating body algorithm, where, for a given convex body $K$ in $\R^{d}$, one uses floating bodies to construct a polytope $P_N$ with as few vertices $N$ as possible such that for a suitable $\delta$, $K_{\delta}\subseteq P_{N}\subseteq K$ and such that $P_N$ approximates the convex body $K$ very well in the symmetric difference metric. It should be noted that we make no assumption on $K$.

We describe this algorithm: We are choosing the vertices $x_{1},\dots,x_{N}\in\partial K$ of the polytope $P_{N}$. $x_{1}$ is chosen arbitrarily.
Having chosen $x_{1},\dots,x_{k-1}$ we choose $x_{k}$ such that
	\[	\{x_{1},\dots,x_{k-1}\}\cap\operatorname{Int}\left(K\cap H^{-}(x_{k}-\Delta_{k}N_K(x_{k}),N_K(x_{k}))\right) = \emptyset \]
where $N_K(x_{k})$ denotes a (not necessarily unique) outer normal to $\partial K$ at $x_{0}$, $\operatorname{Int}(C)$ is the interior of a set $C \subset \R^d$, and $\Delta_{k}$ is determined by 
	\[	\vol{K\cap H^{-}\left( x_{k}-\Delta_{k} N_K(x_{k}), N_K(x_{k})\right)}=\delta.	\]
The next theorem can be found in \citet{Schue1999}. 

	\begin{theorem}\label{PolyApproxFloat1}
	Let $K$ be a convex body in $\R^{d}$. Then, for all $\delta$ with $0< \delta\leq\frac{1}{4 \Euler^{4}}\vol{K}$ there exists $N\in\N$ with 
		\begin{equation*}
		\vol{K\setminus K_{\delta}} \leq N\left(\frac{c}{4 \Euler^{4}}\right)^{d} \vol{\B} 4 \Euler^{3} \delta
		\end{equation*}
	where $c$ is a universal constant, and there exists a polytope $P_{N}$ that has at most $N$ vertices and such that
		\[	K_{\delta}\subseteq P_{N}\subseteq K.	\]
	\end{theorem}
	
How well does this polytope approximate $K$? It follows from Theorem~\ref{PolyApproxFloat1} that
	\[	\limsup_{N\to\infty}\frac{\dS(K,P_{N})}{\left(\frac{1}{N}\right)^{\frac{2}{d-1}}} \leq c \  d^{2}\left(\int_{\partial K}\kappa(x)^{\frac{1}{d+1}} \dd \mu_{\partial K} (x) \right)^{\frac{d+1}{d-1}}.	\]
This should be compared to~\eqref{BestApprox-1}. Since $\operatorname{del}_{d-1}$ is of the order of $d$, both expressions differ only by a factor of the order of dimension $d$.

\section{Floating bodies of measures}	\label{section:floating bodies for measures}


The definition of the (convex) floating body of a convex body $K \in \CB$ discussed in Section~\ref{section:convex floating bodies} extends naturally also to general probability measures, in a manner different than that from Section~\ref{FLME}. It is closely related to the halfspace depth. Analogously to the approach of \citet{Dupin1822}, let $P\in\Prob$ and $\delta > 0$. We say that the nonempty convex set $P_{[\delta]} \in \CB$ is the floating body of $P$ if for each supporting halfspace $H^+$ of $P_{[\delta]}$ we have $P(H^-) = \delta$. 

For $P$ distributed uniformly on a convex body $K \in \CB$ of unit volume, $P_{[\delta]} = K_{[\delta]}$. Therefore, the floating body $P_{[\delta]}$ does not exist for $\delta > 1/2$, and it may happen that it does not exist for any $\delta>0$, see the example of (the uniform distribution on) a triangle from Example~\ref{example:floating bodies}. Unlike in the situation with the floating body of $K \in \CB$, even if the floating body $P_{[\delta]}$ of $P \in \Prob$ exists, it may not be uniquely defined. Take, for instance, a distribution $P$ on $\R$ whose support is not contiguous, such as that displayed in Figure~\ref{figure:difference}. For $\delta = 1/4$, and $q_2$ the $\left(1-\delta\right)$-quantile of $P$, each interval $[q_1, q_2]$ for $q_1 \in [0,1]$ is a floating body of $P$. Note that if $P$ has contiguous support and $P_{[\delta]}$ exists, then it is unique.

To avoid these problems, let us consider, as in the case of convex bodies, the convex floating body of $P$, given by an intersection of halfspaces.

\begin{definition}
Let $P\in\Prob$. For $\delta \geq 0$, the convex floating body of $P$ with index $\delta$ is defined as the intersection of all closed halfspaces whose defining hyperplanes cut off a set of probability content at most $\delta$ from $P$, i.e.
	\begin{equation}	\label{weighted floating body}
	\CRFB = \bigcap_{P(H^-)\leq \delta} H^+, 
	\end{equation}
where $H\in\mathcal H$ and $H^{+}$ and $H^{-}$ are its associated closed halfspaces.
\end{definition} 

Note that with the convention that the intersection of an empty collection of subsets of $\R^d$ is $\R^d$, convex floating bodies of a measure are always well defined, unique, convex subsets of $\R^d$. It can happen that $\CRFB = \emptyset$, especially for larger values of $\delta$. It is easy to see that for $P \in \Prob$ distributed uniformly on $K \in \CB$ with $\vol{K} = 1$, $\CRFB = K_{\delta}$ for any $\delta \geq 0$, and the convex floating bodies of measures generalize the convex floating bodies discussed throughout Section~\ref{section:convex floating bodies}.

(Convex) floating bodies for general measures have already been considered in the literature, mainly due to the association of convex bodies and log-concave measures established by \citet{Ball1988}. The previous definitions were considered by \citet{Werner2002}, \citet{Bobkov2010}, \citet{Fresen2012, Fresen2013}, and \citet{Brunel2018}, among others. In connection with the halfspace depth, the floating bodies \eqref{weighted floating body} were considered in \citet{Nolan1992}, and \citet{Masse_Theodorescu1994}. In the latter paper, those regions are called the $\delta$-trimmed regions of $P$.

The convex floating body of a measure $P$ is very closely related to the depth central region $\CR$, defined in \eqref{central region} as the upper level set of the depth $\HD\left(\cdot;P\right)$. Indeed, recall the characterization of \citet[Proposition~6]{Rousseeuw_Ruts1999}, who showed that for any $P\in\Prob$ and $\delta>0$
	\begin{equation*}
	\CR = \bigcap_{P(H^+) > 1 - \delta} H^+.	
	\end{equation*}
On the other hand, it is not difficult to see that the convex floating body \eqref{weighted floating body} can be written also in the form
	\begin{equation}	\label{floating body reverse}
	\CRFB = \bigcap_{P(H^{+}) \geq 1 - \delta} H^{+}.	
	\end{equation}
Now it is obvious that for all $\delta \geq 0$ we have that
	\[	\CRFB \subseteq \CR,	\]
and under the assumption of the contiguity of the support of $P$,
	\[	\CR = \CRFB.	\]
These results were noted by \citet[Theorem~2]{Kong_Mizera2012} and \citet[Lemma~1]{Brunel2018}. For general measures $P$ it may happen that the convex floating body is a proper subset of the depth central region, see Figure~\ref{figure:difference}. 

It is interesting to investigate which results for convex bodies described in Section~\ref{section:convex floating bodies} carry over to measures. 

\begin{figure}[htpb]
\includegraphics[width=.75\textwidth]{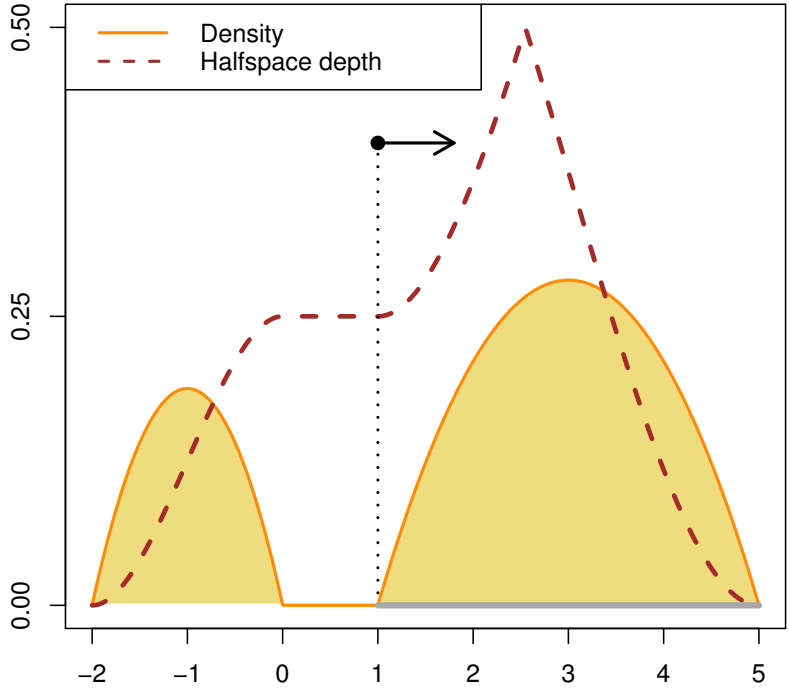} 
\caption{For distributions whose support is not contiguous, the convex floating body \eqref{weighted floating body} and the halfspace depth central region \eqref{central region} may differ. In this example, the density of $P \in \Prob[\R^1]$, supported on disjoint intervals $[-2,0]$ and $[1,5]$, is displayed (orange line), along with its halfspace depth function (dashed brown line). For $\delta = 1/4$, the left endpoint of the interval $\CR$ is $0$. But the complement of the halfline $[1,\infty)$ (black arrow) has probability $1/4$, and the left endpoint of the convex $1/4$-floating body of $P$ is $1$. Points in the interval $(0,1)$ are not boundaries of any convex floating body of $P$.}
\label{figure:difference}
\end{figure}

Let us first relate the floating body $P_{[\delta]}$ with the the convex floating body $\CRFB$ and the central region $\CR$. If a unique floating body $P_{[\delta]}$ of a measure $P\in\Prob$ exists, then the corresponding convex floating body $\CRFB$ must be equal to $P_{[\delta]}$. For the sake of completeness, let us provide an elementary proof of this result.

%
%
%

\begin{proposition}	
Let  $P\in\Prob$ have contiguous support. Let  $\delta > 0$ and  assume that $P_{[\delta]}$ exists. Then $\CRFB = \CR = P_{[\delta]}$.
\end{proposition}

\begin{proof}
Recall that if $P_{[\delta]}$ exists, then it is unique as $P$ has contiguous support. For $P$ with contiguous support, the proof of $\CRFB = \CR$ can be found in \citep[Lemma~1]{Brunel2018}. 

We show now that $P_{[\delta]}= P_{\delta}$. We show first that $P_{\delta} \subseteq P_{[\delta]}$. Let $x \notin P_{[\delta]}$.  By the Hahn-Banach separation theorem there is a support hyperplane $H_0$ to $P_{[\delta]}$ that strictly separates $x$ and $P_{[\delta]}$, i.e., $x \in \operatorname {Int}(H_0^-)$, $P_{[\delta]} \subset H_0^+$ and $P_{[\delta]} \cap H_0 \neq \emptyset$. Since $H_0$ is a supporting hyperplane to $P_{[\delta]}$, $P(H_0^-) = \delta$. Then, as $ P_{\delta}= \bigcap_{P(H^-)\leq \delta} H^+$, $x \notin P_\delta$.

Now we show that $P_{[\delta]} \subseteq P_{\delta}$. Suppose not. Then there exists $x \in P_{[\delta]}$ such that, by \eqref{floating body reverse}, $x \notin H_1^+$ for some $H_1$ with $P(H_1^+) \geq 1 - \delta$. In that case there must exist a hyperplane $H_2$ with $x \in H_2^- \subsetneq H_1^-$. Because $H_2^-$ lies completely in the open halfspace complementary to $H_1^+$, by the contiguity of $P$ we know that $P(H_2^-) < \delta$. This contradicts $x \in P_{[\delta]}$, as for $P$ contiguous the boundary hyperplane of any closed halfspace with probability $\delta$ must support $P_{[\delta]}$.
\end{proof}

Now we explore whether analogues of Propositions~\ref{PropFloatBod1}--\ref{DiffFloat3} stated for convex bodies in Section~\ref{section:convex floating bodies} hold true also for measures.

\subsubsection{Proposition~\ref{PropFloatBod1} for measures}

A result analogous to part~\ref{existence of support} of Proposition~\ref{PropFloatBod1} would require that the infimum in the definition of the halfspace depth \eqref{halfspace depth} can be replaced by a minimum, i.e. that a minimal halfspace of $\HD$ exists at each $x\in\R^d$ for any $P\in\Prob$. For measures that that do not satisfy \eqref{Delta} this is not true, as noted already by \citet[Remark~1]{Rousseeuw_Ruts1999}. There, the following example is given. 

\begin{example}
Let $P\in\Prob[\R^2]$ be a mixture of the standard bivariate Gau{\ss}ian distribution and the Dirac measure at the point $(1,1)$, with equal mixing proportions. Then, at $x = (0,1)$, we have $\HD\left(x;P\right) = \Phi(-1)/2$, where $\Phi$ is the distribution function of the standard univariate Gau{\ss}ian distribution, see also Example~\ref{example:alpha symmetric distributions}. Yet, no minimal halfspace at $x$ exists.
\end{example}

For a different example of the same phenomenon, see \citet[Section~2]{Masse2004}. For distributions that satisfy \eqref{Delta}, a minimal halfspace always exists for all $x\in\R^d$. That was shown, e.g., by \citet[Proposition~4.5 (i)]{Masse2004}.

An extension of Dupin's theorem (part~\ref{Dupin} of Proposition~\ref{PropFloatBod1}) to probability distributions was stated in \citet[Theorem~3.1]{Hassairi_Regaieg2008}. Here we provide a version of that result with a slightly modified set of assumptions. The proof of the proposition follows very closely the original proof of \citep[Theorem~3.1]{Hassairi_Regaieg2008}, and is omitted.

\begin{proposition}	\label{proposition:Hassairi}
Let $X \sim P\in\Prob$ be absolutely continuous with contiguous support $\Supp(P)$ and let $x\in\R^d$ be such that $\HD(x;P) > 0$. Denote by $f_u$ the density of the random variable given by $\left\langle X - x, u \right\rangle$ with $u \in \Sph$. Suppose that $f_u(y)$ is continuous as a function of $u \in \Sph$ and $y$ in a neighborhood of $0 \in \R$. Let $H^-\in \mathcal H^-$ be a minimal halfspace at $x$, i.e. $x\in H$ and $P(H^-) = \HD(x;P)$. Then 
	\[	x = \frac{\int_H y f(y) \dd y}{\int_H f(y) \dd y},	\]
i.e. $x$ is the conditional expectation of $P$ given $H$. The integrals in the formula above are taken with respect to the $(d-1)$-dimensional Lebesgue measure on $H$.
\end{proposition}

One has to be careful with the statement of Proposition~\ref{proposition:Hassairi}. Without the required continuity properties of the marginal densities $f_u$, the conditional expectation of $P$ given a hyperplane $H$, may not even be well defined. To illustrate our point, we give an example that was brought to our attention by M.~Tancer~\citep{Tancer}.

\begin{example}	\label{example:Tancer}
Let $P\in\Prob[\R^2]$ be distributed uniformly on the union of two squares with vertices $(1,0)$, $(1,2)$, $(-1,2)$, $(-1,0)$, and $(2,0)$, $(2,-4)$, $(-2,-4)$, $(-2,0)$, respectively, see the left panel of Figure~\ref{figure:Tancer}. Consider $x = (\varepsilon,0)$ for $-1/2 \leq \varepsilon \leq 1/2$. A simple computation shows that $\HD(x;P) = 1/5$, and the unique minimal halfspace at all such points $x$ is the halfspace $H^+$ that cuts off the smaller square from $P$. A direct analogue of Dupin's theorem would now assert that the conditional expectation of $H = \partial H^+$ is not unique --- any $x$ on the line segment $L$ that joins $(-1/2,0)$ and $(1/2,0)$ would be a candidate for the barycenter of $P$ given $H$. The problem here, of course, is due to the discontinuity of the marginal density of $P$ at $H$. For this particular $H$, the conditional expectation of $P$ given $H$ is not properly defined.
\end{example}

\begin{problem}
Does a version of Dupin's theorem (i.e. a variant of Proposition~\ref{proposition:Hassairi}) hold true also under weaker conditions on measures $P\in\Prob$? 
\end{problem}

\begin{figure}[htpb]
\includegraphics[width=.45\textwidth]{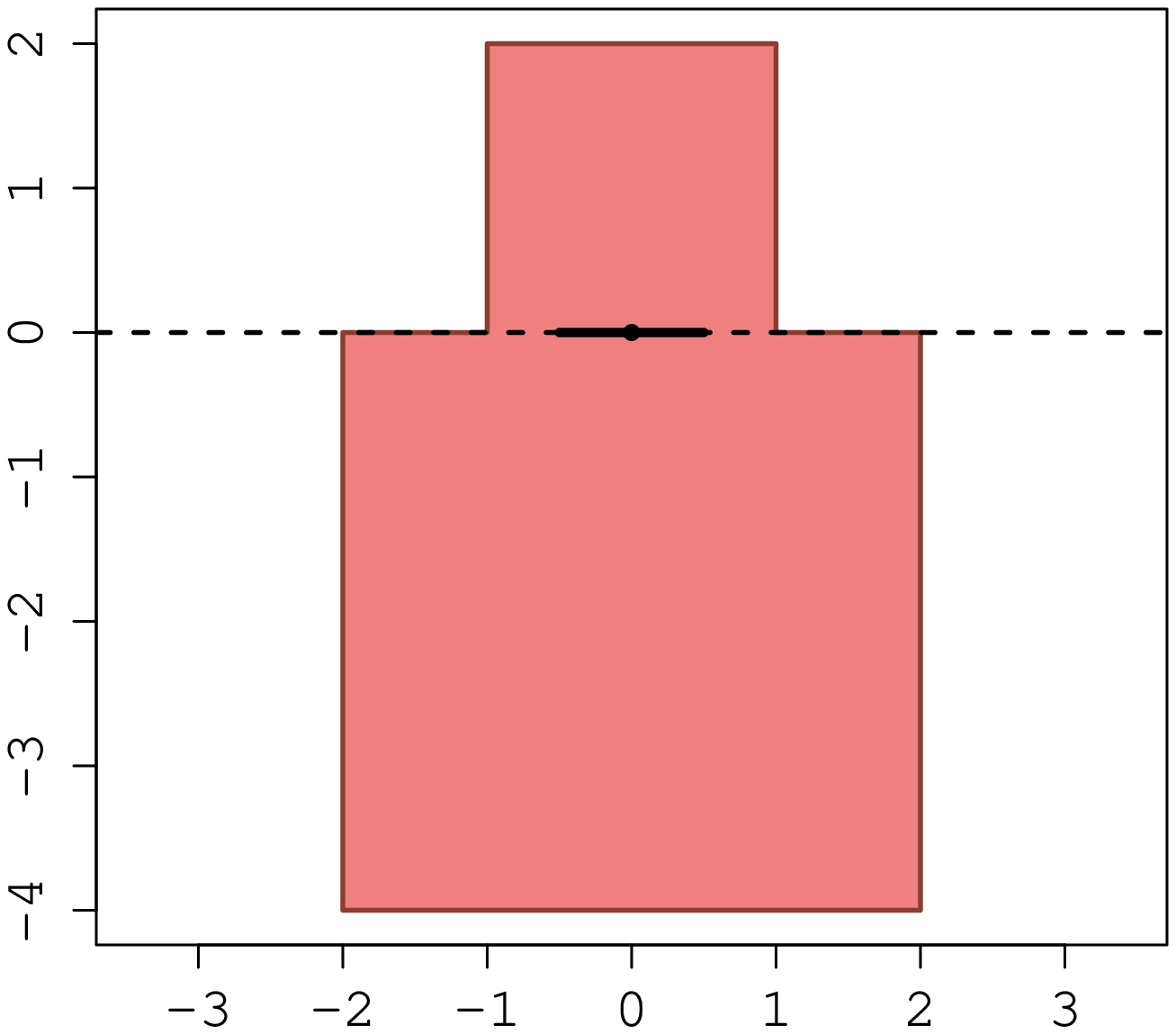} \quad \includegraphics[width=.45\textwidth]{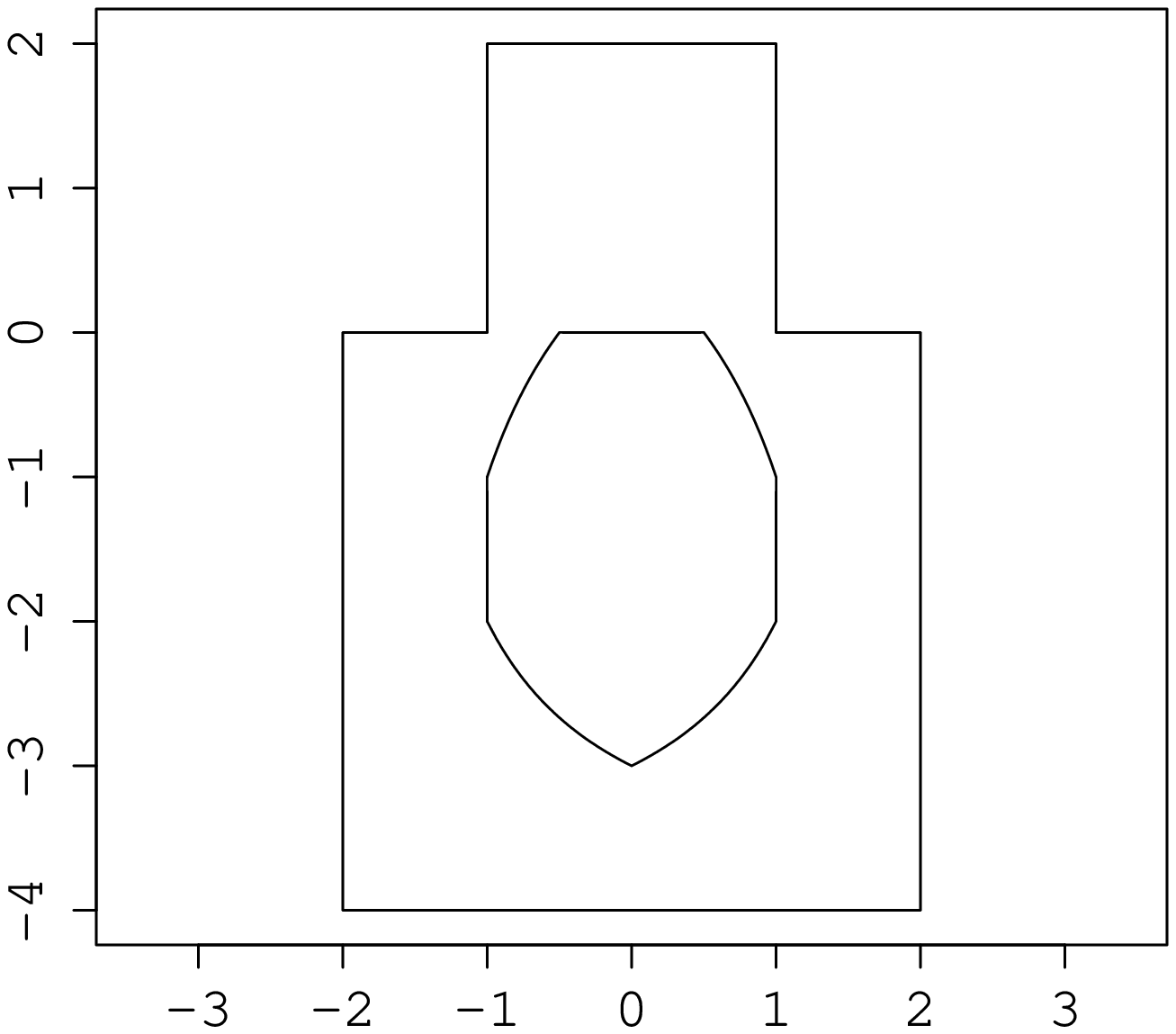}
\caption{Left panel: polygon where the measure $P$ is supported in Example~\ref{example:Tancer}, line segment $L$ (thick solid line), and the unique minimal hyperplane at $x = (\varepsilon,0)$ for $\left\vert \varepsilon \right\vert \leq 1/2$ (dashed line). Right panel: central region $\CR$ of $P$ for $\delta = 0.2$. This region is not strictly convex, as it contains the line segment $L$ between the points $(-1/2,0)$ and $(1/2,0)$.}
\label{figure:Tancer}
\end{figure}

To see that the strict convexity of the central regions (part~\ref{strict convexity} of Proposition~\ref{PropFloatBod1}) does not hold true for all measures, it is enough to return to Example~\ref{example:Tancer}. Indeed, due to the considerations made there, the line segment $L$ lies on the boundary of the central region $\CR = \CRFB$ for $\delta = 1/5$, and $P_{1/5}$ is not strictly convex, see also the right panel of Figure~\ref{figure:Tancer}. For another example where the strict convexity of $\CR$ is violated, recall the collection of $\alpha$-symmetric distributions from Example~\ref{example:alpha symmetric distributions} for $\alpha \leq 1$, and the right panel of Figure~\ref{figure:depth for distributions with density}. In Example~\ref{example:Tancer}, the problem appears to stem from discontinuity of the density of $P$ at the boundary of a minimal halfspace. For $\alpha$-symmetric distributions the problem is that the expectation of $P$ is not defined.

\begin{problem}
Under which conditions are the central regions $\CR$ and the convex floating bodies $\CRFB$ strictly convex? 
\end{problem}

An extension of part~\ref{median} of Proposition~\ref{PropFloatBod1} was given by~\citet[Proposition~7]{Mizera_Volauf2002}, who stated that if \eqref{Delta} is true for $P$ with contiguous support, then the halfspace median of $P$ is a unique point. 

\subsubsection{Proposition~\ref{FloatSymm} for measures}

The (Dupin's) floating body $P_{[\delta]}$ of a general measure $P$ may not exist. Sufficient conditions for the existence of floating bodies of probability measures appear to be a challenging problem of great importance in mathematical statistics, and the theory of data depth (see, e.g., \citep[Open question~1]{Brunel2018}, or \citep{Masse2004, Masse2009}). Many theoretical results on the behavior of the depth and its central regions hold true only under the assumption of existence of floating bodies of $P$, see also the discussion in Section~\ref{section:characterization} below. \citet[Open question~2]{Brunel2018} asks a question that can be rephrased as follows:

	\begin{quotation}
	\emph{Is it true that for any log-concave measure $P\in\Prob$ all floating bodies $P_{[\delta]}$ for $\delta > 0$ small enough exist?}
	\end{quotation}

From the example of the uniform distribution on a triangle (Example~\ref{example:floating bodies}), we see that the answer to the above question is negative. Though, under the additional assumption of central symmetry of $P$, similar properties have been investigated by \citet{Meyer_Reisner1991} for convex bodies (see Proposition~\ref{FloatSymm} above), and extended to certain probability measures by \citet[Section~6]{Bobkov2010}. In the latter paper, it is shown that all $P_{[\delta]}$ exist for centrally symmetric $s$-concave measures with $s\geq -1$. As far as we are aware, the following theorem from \citep[Theorem~6.1]{Bobkov2010} is, up to date, the most general result on the existence of floating bodies of measures $P\in\Prob$.

\begin{theorem}	\label{theorem:existence}
Let $P\in\Prob$ be a centrally symmetric $s$-concave measure with $s \geq -1$ such that $\Supp(P)$ is a $d$-dimensional subset of $\R^d$. Then $P_{[\delta]}$ exists for all $\delta \in (0,1/2]$.
\end{theorem}

As remarked by \citet{Bobkov2010}, it is not known whether the restriction $s \geq -1$ can be dropped. 

\begin{problem}
Do the floating bodies of all centrally symmetric $s$-concave measures with full-dimensional support exist?
\end{problem}

\begin{problem}
Let $P\in\Prob$ be centrally symmetric with a sufficiently smooth density $f$ that is positive on $\R^d$. Suppose that all the upper level sets of the density $\{x\in\R^d \colon f(x)\geq t\}$ are (strictly) convex. Does this imply that the floating bodies of $P$ exist for $\delta \in (0,1/2]$?
\end{problem}

As shown in the following theorem, there exists a close connection between the question of existence of floating bodies, and the problem of smoothness of the boundaries of the central regions $\CR$.

\begin{proposition}	\label{theorem:smoothness and floating bodies}
Let $P\in\Prob$ satisfy \eqref{Delta}, $\delta \in (0,1/2)$, and let $\CR$ be a convex body whose boundary is $C^1$. Then $P_\delta$ is a floating body.
\end{proposition}

\begin{proof}
Let $x \in \partial \CR$. Since, under \eqref{Delta}, the depth $\HD\left(\cdot;P\right)$ is continuous on $\R^d$ (\citep[Proposition~1]{Mizera_Volauf2002}, or Section~\ref{section:continuity} above), $\HD(x;P) = \delta$. Using \citep[Proposition~4.5 (i)]{Masse2004} there exists a minimal halfspace $H^-\in\mathcal H^-$ at $x$, and $H^+$ then must support $\CR$ at $x$ by Proposition~\ref{minimal halfspace supports}. Because of the smoothness of the boundary of $\CR$, there is only a single supporting halfspace of $\CR$ at each $x \in \partial \CR$. Thus, we have shown that for any supporting halfspace $H^+$ of $\CR$, $P(H^-) = \delta$, and $\CR$ is a floating body of $P$.
\end{proof}

Smoothness of boundaries of $\CR$ was recognized to be crucial in establishing theoretical properties of $\HD$ already by \citet{Nolan1992}, and \citet{Masse_Theodorescu1994}. 
Many theoretical results stated for the halfspace depth in statistics rely on that condition. For instance, as shown by \citet[Theorem~2.1]{Masse2004}, the asymptotic distribution of the sample halfspace depth at $x$ is Gau{\ss}ian if the boundary of $\CR$ passing through $x$ has a unique minimal halfspace. For another application of the smoothness of boundaries of floating bodies see Section~\ref{section:characterization} below.

Despite being of critical importance, so far the only examples of distributions with smooth contours of $\HD$ are the (full-dimensional affine images of) $\alpha$-symmetric distributions with $\alpha>1$, see Example~\ref{example:alpha symmetric distributions}. As discussed in \citet{Gijbels_Nagy2016}, apart from those distributions, no other multivariate measure with smooth depth contours is known in statistics. In that paper, it is also shown that simple distributions such as mixtures of multivariate Gau{\ss}ian distributions, and distributions with smooth centrally symmetric, or smooth strictly quasi-concave densities, may have points at which the boundary of $\CR$ is not smooth. It is therefore remarkable that \citet[Theorem~3]{Meyer_Reisner1991} (part~\ref{C2 differentiability} of Proposition~\ref{FloatSymm} above) showed that for certain (centrally) symmetric convex bodies, the boundaries of $K_{\delta}$ exhibit a high degree of smoothness. We are not aware of any result giving sufficient conditions for higher order differentiability of the boundary of the depth central regions, or convex floating bodies of measures, in statistics. 

\begin{problem}
Under which conditions have the central regions $\CR$ and the convex floating bodies $\CRFB$ boundaries of type $C^1$ or $C^2$?
\end{problem}

\subsubsection{Propositions~\ref{DiffFloat1} and~\ref{DiffFloat3} for measures}

\begin{problem}
Are there analogues of Proposition~\ref{DiffFloat1} and Proposition~\ref{DiffFloat3} for measures? 
\end{problem}

\subsection{Application: Multivariate extremes and depth}

The intimate connections of floating bodies with the approximation problems described in Section~\ref{ApproxCoBodPol} have analogues for probability measures. If $X_1, \dots, X_n$ is a random sample from distribution $P \in \Prob$, one can ask how fast does the random polytope given by the convex hull of these random points grow to the convex hull of the support of $P$. In conjunction with the advances for uniform measures on convex bodies outlined in Section~\ref{ApproxCoBodPol}, it is not surprising that the halfspace depth and floating bodies of measures play a prominent role in these problems.

The following theorem, called the multivariate Gnedenko law of large numbers, can be found in \citet[Theorem~2]{Fresen2013}. 

\begin{theorem}	\label{theorem:Fresen}
Let $q > 0$ and $p > 1$, and let $P\in\Prob$ be a probability measure with a density of the form $f(x) = c \Euler^{-g(x)^p}$ where $g \colon \R^d \to [0,\infty)$ is a convex function and $c>0$. Then there exist constants $c_1, c_2 > 0$ such that for a random sample $X_1, \dots, X_n$ of any size $n \in \N$ with $n \geq d+2$ from $P$, it holds true that
	\begin{equation}	\label{Fresen inequality}
	\PP\left( \dH\left( \left[ X_1, \dots, X_n \right], P_{1/n} \right) \leq c_1 \frac{\log \log n}{\left(\log n\right)^{1-1/p}} \right) \geq 1 - c_2 \left( \log n \right)^{-q},
	\end{equation}
where $\left[ X_1, \dots, X_n \right]$ is the closed convex hull of the points $X_1, \dots, X_n$, and $P_{1/n}$ is the depth central region $\CR$ with $\delta = 1/n$. 
\end{theorem}

Distributions $P$ from Theorem~\ref{theorem:Fresen} are sometimes called $p$-log-concave measures. For usual log-concave measures, an inequality only slightly weaker than \eqref{Fresen inequality} is given in \citet[Theorem~1]{Fresen2013}. 

Theorem~\ref{theorem:Fresen} asserts that, with large probability, convex hulls of large random samples from $P$ behave as the halfspace depth central regions $\CR$ for very small values of $\delta$. This observation opens a whole new field of applications of the depth in multivariate extreme value theory. Indeed, by now, data depth has been used in statistics predominantly as a robust tool that identifies the central parts of the probability mass of distributions, and little attention was paid to its behavior near the tails. Theorem~\ref{theorem:Fresen} gives a probabilistic interpretation also to the boundaries of those depth regions that correspond to the extreme depth-quantiles. It is also interesting to compare Theorem~\ref{theorem:Fresen} with the recent advances of \citet{Einmahl_etal2015} and \citet{He_etal2017}. There, the authors employ extreme value theory in order to estimate $\CR$ for low values of $\delta$ reliably from the data. It will be interesting to see what can be obtained by a proper combination of the estimation techniques from the latter papers, and the asymptotic representations of \citet{Fresen2013}.

In a further analogue with the exposition from Section~\ref{section:convex floating bodies}, one may study the limit behavior of the quantity
	\begin{equation}	\label{affine surface area for measures} 
	P\left(P_{0}\right)-P\left(\CR\right) = 1 - P\left(\CR\right)	
	\end{equation}
as $\delta \to 0$ from the right. More specifically, assume that the difference \eqref{affine surface area for measures} is scaled properly, so that the resulting limit is a finite, non-negative number $\Omega(P)$. The characteristic $\Omega(P)$, together with the sequence of its scaling constants, is then an affine invariant on $\Prob$. $\Omega(P)$ is not a generalized notion of the affine surface area such as the functional from Section~\ref{section:Asa for measures}, but it is interesting in its own right. From the viewpoint of statistics, $\Omega(P)$ may serve as an index of heavy-tailedness of the distribution $P$, where not only the size of the tails is evaluated, but also ``the complexity of the boundary" of $\Supp(P)$ is taken into account.

\section{Mahalanobis ellipsoids and the halfspace depth}	\label{section:ellipsoids}

Let $X \sim P\in\Prob$ be distributed uniformly on $K\in\CB$. The body $K$ is said to be isotropic, or in the isotropic position, if $\vol{K} = 1$, $\E X = 0$, and $\Var X = L_K^2 I_d$ where $L_K > 0$ a constant and $I_d$ the $d\times d$ identity matrix. Geometrically, this means that the barycenter of $K$ is at the origin and that the ellipsoid of inertia of $K$, or equivalently, all Mahalanobis ellipsoids of $P$ from \eqref{Mahalanobis level set}, are Euclidean balls. The constant
	\begin{equation*}
	L_K= \frac{1}{d} \  \left( \int_K \| x\|^2 \dd x \right)^\frac{1}{2}
	\end{equation*}
is called the {\em isotropic constant} of $K$.

The isotropic constant plays an important role in the analysis of convex bodies. We refer to e.g., \citet{Milman_Pajor1989} and the book of \citet[Chapter~3]{Brazitikos_etal2014}. The conjecture that for all $K \in \CB$ the constant $L_K$ is bounded from above by an absolute constant independent of the dimension $d$ is one of the major open problems in geometric analysis. The best known upper estimate so far, due to \citet{Klartag2006}, is that $L_K \leq c \ d^\frac{1}{4}$ for an absolute constant $c$, improving an earlier estimate by \citet{Bourgain1991} by a logarithmic factor. The conjecture is equivalent to the {\em hyperplane conjecture}, first formulated by Bourgain, which asks if every centered convex body of volume $1$ has a hyperplane section through the origin whose $(d-1)$-dimensional volume is greater than an absolute positive constant, independent of dimension $d$. We refer to e.g., \citep[Section~3.1]{Brazitikos_etal2014} for the details.

For any $K$ there exists an affine transformation $T$ such that $T(K)$ is isotropic, and the isotropic position is uniquely determined up to orthogonal transformations. The isotropic constant of a general body $K \in \CB$ is then defined as the isotropic constant of the corresponding isotropic body $T(K)$. 

Similarly, we can define the isotropic constant for probability measures $P$ with log-concave density $f$. A measure that corresponds to $X \sim P$ is isotropic if it is centered, i.e. $\E X=0$, and if for all $u \in \Sph$, 
	\begin{equation*}
  \int _{\R^d}  \langle x, u \rangle^2  f(x) \dd x  = 1,
	\end{equation*}
or, equivalently, $\Var X = I_d$. Then
	\[	L_P = \left(\sup_{x \in \R^d} f(x)\right)^\frac{1}{d}	\]
is the isotropic constant of $P$. The isotropic constant of a general probability measure $P$ with log-concave density is, again, given as the isotropic constant of an affine image of $P$ that is isotropic \citep[Chapter~2]{Brazitikos_etal2014}. Note that a convex body $K \in \CB$ of volume $1$ is isotropic, if and only if the density of the uniform distribution on the convex body $K/L_K$ is an isotropic log-concave density.

For bodies and log-concave measures in isotropic position, many important geometrical results are known. In this section we state one that relates to the subjects of data depth and floating bodies.

\begin{proposition}	\label{theorem:Legendre}
The following holds true:
	\begin{enumerate}[label=(\roman*)]
	\item \label{Milman Legendre} For any isotropic convex body $K\in\CB$ and any $\delta\in\left(0, \frac{1}{\Euler}\right)$
		\begin{equation*}	
		\left( \frac{1}{\Euler} -\delta \right) L_K \, \B  \subseteq K_\delta \subseteq 17 \log \left(\frac{1}{\delta}\right) L_K \, \B.	
		\end{equation*}
	\item \label{Fresen Legendre} For any isotropic measure $P\in\Prob$ with a log-concave density 
		\begin{equation*}	
		\left( \frac{1}{\Euler} -\delta \right) L_P \, \B  \subseteq P_\delta  \subseteq 17 \log \left(\frac{1}{\delta}\right) L_P \, \B.	
		\end{equation*}
	\end{enumerate}
\end{proposition}

This proposition was proved by \citet[Proposition in the Appendix]{Milman_Pajor1989}, and re-stated by \citet{Fresen2012} who also gave the formulation in part \ref{Fresen Legendre} for isotropic log-concave measures. Its further extension to centrally symmetric $s$-concave measures with $s > -\infty$ can be found in \citet[Theorem~5.1]{Bobkov2010}. Part~\ref{Milman Legendre} of this proposition is a special case of more general relations between floating bodies and {\em $p$-centroid bodies} which can be found in \citep[Theorem~2.2]{Paouris_Werner2012}.

Proposition~\ref{theorem:Legendre} has important implications for the theory of halfspace depth. By affine equivariance of the halfspace depth central regions $\CR$, for any log-concave measure $P\in\Prob$ with expectation $\mu\in\R^d$ and a positive definite variance matrix $\Sigma \in \R^{d \times d}$,
	\[	\left\{ x \in \R^d \colon \dist(x,\mu) \leq \left( \frac{1}{\Euler} -\delta \right) L_P \right\} \subseteq P_\delta \subseteq  \left\{ x \in \R^d \colon \dist(x,\mu) \leq 17 \log \left(\frac{1}{\delta}\right) L_P \right\}	\]
where $\dist$ is the Mahalanobis distance from~\eqref{Mahalanobis distance}. 
%
%
Therefore, all central regions of the halfspace depth for $\delta < 1/\Euler$ of log-concave measures are, up to a constant that depends only on $\delta$ and $L_P$, isomorphic to the Mahalanobis ellipsoids given by the covariance structure of $P$. This corroborates the findings from statistics, where it has been long observed that the depth central regions $\CR$ tend to take more ``ellipsoidal" shapes than the level sets of the densities, see also Figures~\ref{figure:floating bodies} and~\ref{figure:depth for distributions with density} above. Results in this section provide quantitative statements that support those claims. 

\begin{problem}
Is it possible to state an analogue of Proposition~\ref{theorem:Legendre} also for more general probability measures?
\end{problem}

\section{Characterization of distributions}	\label{section:characterization}

One of the most important open questions connected with the halfspace depth is the halfspace depth characterization problem. It has been conjectured (e.g. \citep[p.~2306]{Cuesta_Nieto2008c} and \citep[p.~1598]{Kong_Mizera2012}) that for each distribution $P\in\Prob$ there exists a unique depth surface $\left\{ \HD(x;P) \colon x \in \R^d \right\}$, i.e., that all probability distributions are determined by their halfspace depth. Such a result would be invaluable in statistics, as it would assert that just as the distribution function or the characteristic function of a random vector, also the halfspace depth could be used as a complete representative of any probability distribution.

Recently, the depth characterization conjecture was disproved in \citep{Nagy2018s}, where an example of two different probability distributions with the same depth at all points in $\R^d$, $d \geq 2$, was given. The example employs collections of different $\alpha$-symmetric distributions with $\alpha\leq 1$ whose projections coincide in some directions.

Even though the general characterization conjecture turned out to be false, important partial positive results to the characterization problem can be found in the literature. Thanks to the results of \citet{Struyf_Rousseeuw1999}, \citet{Koshevoy2002}, and \citet{Hassairi_Regaieg2007} we know that if $P, Q\in\Prob$ are distributions whose supports are finite subsets of $\R^d$, then $\HD(x;P) = \HD(x;Q)$ for all $x \in\R^d$ implies $P=Q$. For non-atomic distributions, two results can be found in the literature in the papers of \citet{Hassairi_Regaieg2008}, and \citet{Kong_Zuo2010}. In this section we show that the last two theorems are special cases of the following theorem.

\begin{theorem}	\label{theorem:retrieval}
Let $P\in\Prob$ have contiguous support, and let $x_P \in \R^d$ be the halfspace median of $P$. Then the following are equivalent:
	\begin{enumerate}[start=1, label=$\mathbf{(FB_{\arabic*})}$,ref=$\mathbf{(FB_{\arabic*})}$]
		\item	\label{FB1} For each $\delta \in (0,1/2)$ the floating body $P_{[\delta]}$ of $P$ exists. 
		\item \label{FB2} \eqref{Delta} holds true, and
			\begin{equation}	\label{retrieval}
					P(H^-) = 	\begin{cases}
									\sup_{x\in H} \HD(x;P)	& \mbox{for any $H\in\mathcal H$ with $x_P \notin H^-$}, \\
									1 - \sup_{x\in H} \HD(x;P)	& \mbox{for any $H\in\mathcal H$ with $x_P \in H^-$}.
									\end{cases}	
			\end{equation}
	\end{enumerate}
Consequently, if \ref{FB1} is true, then $P$ is characterized by its halfspace depth, i.e. there is no other probability distribution with the same depth at all points in $\R^d$. 
\end{theorem}

\begin{proof}
Assume first that \ref{FB1} is true. We show first that \eqref{Delta} holds. Suppose it does not hold. Then there exists a hyperplane $H$ such that $P(H)>0$. Without loss of generality we can assume that $P(H^-) \leq P(H^+)$.  We put $\delta = P(H^-) -  \frac{3}{4} P(H)$. Then $0 < \delta < 1/2$. We claim that the floating body $P_{[\delta]}$ does not exist, for if it does exist, then there is a supporting hyperplane $H_1$ to $P_{[\delta]}$ parallel to $H$ such that $P(H_1^-) = \delta$. Note that $P(H^-) = \delta + \frac{3}{4} P(H) > \delta$, and it must be that $H_1^- \subsetneq H^-$. But, in that case, $P(H_1^-) \leq P(H^-) - P(H) < \delta$, a contradiction.


Take now an arbitrary hyperplane $H\in\mathcal H$, and define $\psi(H^-) = \sup_{x\in H} \HD(x;P)$. For any $x \in H$ we have 
	\begin{equation}	\label{one side}
	\HD(x;P) = \inf \left\{ P(G^-) \colon G^- \in \mathcal H^-, x \in G \right\} \leq P(H^-)	
	\end{equation}
since the halfspace $H^- \in \mathcal H^-$ belongs to the collection over which the infimum is taken. Because \eqref{one side} is valid for any $x\in H$, $\psi(H^-) \leq P(H^-)$. 


To prove the other inequality, assume that $\delta = P(H^-) > 0$. Otherwise, trivially $\psi(H^-) \geq P(H^-) = 0$. 
Further, it is possible to assume that $\delta\leq 1/2$. 
If this is not the case, take $H^+ \in \mathcal H^-$, the closed halfspace complementary to $H^-$, and proceed with $H^+$ (note that in the latter case, we know by \eqref{Delta} that $P(H^+) \leq 1/2$ and $P(H^+) + P(H^-) = 1$). 
We first treat the case $\delta < 1/2$. Because all floating bodies of $P$ are assumed to exist and because $P(H^-) = \delta$, the hyperplane $H$ supports the floating body $P_{[\delta]}$ of $P$. 
That is, there must exist a point $x_H \in H \cap P_{[\delta]}$. As $x_H \in P_{[\delta]} = \CR = \{ y \in \R^d: hD(y;P) \geq \delta\}$, 
$$P(H^-) = \delta \leq \HD(x_H;P) \leq \sup_{x \in H} \HD(x;P) = \psi(H^-).$$
 Thus (\ref{retrieval})  holds for $\delta < 1/2$. By continuity, it also holds for $\delta = 1/2$. 
 Hence \ref{FB1} implies that the probability of halfspaces is characterized by their depth as in \eqref{retrieval}.

For the opposite implication, assume that \ref{FB2} is true and let $\delta \in (0,1/2)$. Consider the depth level set $\CR$. This is a convex compact set. From \eqref{retrieval} with $H^-$ such that $P(H^-) = 1/2$ and the continuity of the depth $\HD(\cdot;P)$ guaranteed by \eqref{Delta}, we see that $\CR$ must be non-empty for all $\delta \in (0,1/2)$. Take any $H \in \mathcal H$ such that $P(H^-) = \delta$, and consider the family $\mathcal G \subset \mathcal H$ of all hyperplanes parallel to $H$. Then $\CR$ must be supported by some $G \in \mathcal G$ with $G^- \subseteq H^-$ or $G^- \supseteq H^-$. If $P(G^-) = \delta^\prime > \delta$, \ref{FB2} cannot be true as $\CR[\delta^\prime] \subset \CR[\delta]$ by the nestedness and convexity of the central regions, and the continuity of $\HD$. Indeed, because $G$ supports $\CR$, for all $x \in G$ either $x \in \partial \CR$ or $x \notin \CR$. In both cases $\HD(x;P) \leq \delta$, since, using the continuity of $\HD$ again, $\HD(x;P) = \delta$ for any $x \in \partial \CR$. By \eqref{retrieval} this means that we have $\delta^\prime = P(G^-) = \sup_{x \in G} \HD(x;P) \leq \delta$, a contradiction. If $\delta^\prime \leq \delta$, then there must exist $x_0 \in G \cap \CR[\delta]$. But then $\delta \leq \HD(x_0;P) \leq P(G^-) = \delta^\prime \leq \delta$, and necessarily $P(G^-) = \delta^\prime = \delta$. Because $P$ has contiguous support, this means that $G = H$, and $\CR[\delta]$ is supported by $H$. As this is true for any $H\in\mathcal H$ such that $P(H^-) = \delta$, $\CR[\delta] = P_{[\delta]}$, and \ref{FB2}$\implies$\ref{FB1}.

The characterization of $P$ follows from \ref{FB2} by a theorem of \citet{Cramer_Wold1936}, see also \citep[p.~383]{Billingsley1995}. 
\end{proof}

Note that a further minor extension of Theorem~\ref{theorem:retrieval} can be obtained if $P$ is allowed to have a single atom at its halfspace median $x_P$, with obvious modifications to the statement and the proof of this theorem.

By Theorem~\ref{theorem:retrieval} and Proposition~\ref{theorem:smoothness and floating bodies} we obtain that all $\alpha$-symmetric distributions with $\alpha > 1$, and their full-dimensional affine images, satisfy \ref{FB1}. This array of examples complements the known examples of $s$-concave centrally symmetric measures with $s \geq -1$ from Theorem~\ref{theorem:existence}, for which \ref{FB1} is true. As far as we know, there are no further examples of measures satisfying \ref{FB1} known at this time. 

To see that there exist distributions $P\in\Prob$ that satisfy \ref{FB1}, but not the assumptions of Proposition~\ref{theorem:smoothness and floating bodies}, take $P\in\Prob[\R^2]$ to be the uniform distribution on a square in $\R^2$ from Example~\ref{example:floating bodies}. For $P$ it is known \citep[pp.~433--434]{Leichtweiss1986} that \ref{FB1} is true, yet each floating body $P_{[\delta]}$ for $\delta\in(0,1/2)$ contains four non-smooth points at its boundary, see also the left panel of Figure~\ref{figure:floating bodies}. 

Condition \ref{FB1} is, however, still rather strict. Not only does it impose \eqref{Delta} on $P$, but also it means that $P$ must be halfspace symmetric. For (uniform measures on) convex bodies, this was noted by \citet[Lemma~4]{Meyer_Reisner1991}. The next proposition extends that result to probability measures. Its proof follows closely the arguments of \citet[Lemma~4]{Meyer_Reisner1991}, and is omitted.

\begin{proposition}
Let $P\in\Prob$ have contiguous support. If \ref{FB1} is true for $P$, then $P$ must be halfspace symmetric.
\end{proposition} 

\begin{problem}
Describe the collection of all probability measures $P\in\Prob$ whose halfspace depth is unique, i.e. there is no $Q \ne P$ with $\HD(x;P) = \HD(x;Q)$ for all $x\in\R^d$. Is the existence of the expectation $\E X$ for $X \sim P$ sufficient for the halfspace depth of $P$ to be unique? Is the uniform distribution on a simplex in $\R^d$ characterized by its halfspace depth?
\end{problem}

\begin{problem}
If condition \ref{FB1} is not satisfied, how can one reconstruct the probability content of all halfspaces $P(H^-)$ from the depth $\HD(x;P)$ for all $x \in \R^d$ only?
\end{problem}


%
%

\subsection{Characterization theorem of Kong and Zuo (2010)}

In \citep[Theorem~3.2]{Kong_Zuo2010} it is shown that if, for $P\in\Prob$ with contiguous support,
	\begin{equation}	\label{smoothness} 
	\text{for all $\delta \in (0,1/2)$ the boundary of the central region $\CR$ is $C^1$,}	
	\end{equation}
and \eqref{Delta} holds, then \eqref{retrieval} is true, and $P$ is characterized by its halfspace depth. In Theorem~\ref{theorem:retrieval} we provide a generalization of this result. Indeed, by Proposition~\ref{theorem:smoothness and floating bodies} above, if \eqref{smoothness} and \eqref{Delta} are true, then the floating body $P_{[\delta]}$ of $P$ exists for all $\delta \in (0,1/2)$, and Theorem~\ref{theorem:retrieval} can be used.

\subsection{Characterization theorem of Hassairi and Regaieg (2008)}	

Let us state a characterization result for the halfspace depth for absolutely continuous distributions that can be found in \citep[Theorem~3.2]{Hassairi_Regaieg2008}. For this, we define for any $x \in \R^d$ the halfspace function
	\[	\phi_x \colon \Sph \to [0,1] \colon u \mapsto P\left(H_{u,\langle x, u \rangle}^-\right),	\]
where $H_{u,\langle x, u \rangle}^- \in \mathcal H^-$ is the closed halfspace in $\R^d$ whose outer normal is parallel to $u$, and $x \in H_{u,\langle x, u \rangle}$. 

\begin{theorem}	\label{result:Hassairi2}
Let $P\in\Prob$ be as in Proposition~\ref{proposition:Hassairi}, and suppose that
	\begin{equation}	\label{H condition} 
	\text{for all $x\in\R^d$, if $\phi_x$ has a local minimum at $u = u(x) \in\Sph$, then $\phi_x(u) = \HD(x;P)$.} 
	\end{equation}
Then \eqref{retrieval} holds true, and $P$ is characterized by its halfspace depth.
\end{theorem}

In \citep[Theorem~3.2]{Hassairi_Regaieg2008}, condition \eqref{H condition} is formulated in a slightly different manner in terms of derivatives of functions related to $\phi_x$. It is easy to see that for $P$ that satisfies the conditions from Proposition~\ref{proposition:Hassairi}, \eqref{H condition} and the corresponding condition from \citep{Hassairi_Regaieg2008} are equivalent. 

If $P$ satisfies \eqref{Delta}, then for any $x\in\R^d$ the function $\phi_x$ is continuous on $\Sph$ \citep[Proposition~4.5]{Masse2004}. Thus, it must attain a global minimum over its domain. Condition~\eqref{H condition} therefore means that there cannot exist any local minimum of $\phi_x$ that is not global. 

Suppose for a moment that \eqref{Delta} is valid for $P$. By Theorem~\ref{result:Hassairi2}, \eqref{H condition} implies the characterization result \eqref{retrieval} which is, by Theorem~\ref{theorem:retrieval}, equivalent with \ref{FB1}. Therefore, given that \eqref{Delta} is true, Condition \eqref{H condition} implies \ref{FB1}, and the characterization of \citet{Hassairi_Regaieg2008} is a special case of Theorem~\ref{theorem:retrieval} above\footnote{Nonetheless, it must be noted that smoothness of the density of $P$ is not sufficient for \eqref{H condition} to hold true (cf. \citep[p.~2312]{Hassairi_Regaieg2008}). 
To see this, consider the uniform distribution $P$ on a triangle from Example~\ref{example:depth level sets of convex bodies}. This distribution has a smooth density in the interior of $\Supp(P)$, yet $\MD(P) = 4/9$, and $\sup_{x \in H} \HD(x;P) \leq 4/9  < 1/2$ for any $H\in\mathcal H$. Thus, the probability of halfspaces cannot be recovered as in \eqref{retrieval}, at least not for $H \in \mathcal H$ with $1/2 \geq P(H^-) > 4/9$. It is easy to see that \eqref{H condition} is violated for $P$, too.}.
	
\subsection{Homothety conjecture}

In convex geometry, the following open question, similar in nature to the depth characterization conjecture, was posed by \citep{Schutt_Werner1994}:
	
	\begin{quotation}
	Let a convex body $K \in \CB$ and one of its convex floating bodies $K_{\delta}$ be homothetic, i.e. $K_{\delta} = \lambda K + x$ for some $\lambda > 0$ and $x \in \R^d$. Is then $K$ necessarily an ellipsoid?
	\end{quotation}

\citet{Schutt_Werner1994} showed that if $K$ is homothetic to a sequence of its floating bodies $K_{\delta_n}$ with $\delta_n \to 0$, then $K$ must be an ellipsoid. \citet{Stancu2006} demonstrated that for $K$ with a sufficiently smooth boundary, $K$ is homothetic to $K_{\delta}$ for a single small $\delta$ also implies that $K$ is an ellipsoid. The latter result was later refined in \citep{Werner_Ye2011}. 

\begin{problem}
Does the homothety conjecture hold true? More generally, which convex bodies are characterized by any of their convex floating bodies?
\end{problem}

\section{Conclusions and further perspectives}	\label{section:extensions}

In this survey, we discussed little known relations of the concept of halfspace depth, studied extensively in statistics, and paradigms well known in functional analysis and geometry. In Section~\ref{section:Winternitz} we saw that the depth of the halfspace median is a particular example of a more general concept of measures of symmetry. In Sections~\ref{section:convex floating bodies} and~\ref{section:floating bodies for measures} we focused on the floating body and its possible generalizations towards (probability) measures. These little explored junctions of mathematical statistics and geometry are, however, hardly limited only to the halfspace depth $\HD$ defined in finite-dimensional linear spaces $\R^d$. In this concluding section of our paper our intention is to outline, and properly refer to, a few further links between the statistics of depth functions, and current research in pure mathematics.

\subsection{Depth in non-linear spaces}

By directional data one understands data that live on the unit sphere $\Sph$ of $\R^d$ \citep{Mardia1972}. Each observation can be interpreted as a direction of a non-zero vector in $\R^d$. Such data appear quite naturally, and it is of great interest to find depth functions suitable also for this kind of observations. Several definitions of depth have been proposed for directional data \citep{Small1987, Liu_Singh1992, Agostinelli_Romanazzi2013B, Ley_etal2014, Pandolfo_etal2017}. The following depth, proposed by \citet{Small1987}, is an analogue of the halfspace depth for directional data.

\begin{definition}
Let $P\in\Prob[\Sph]$ and $x\in\Sph$. The \emph{angular halfspace depth} (or \emph{angular Tukey depth}) of $x$ w.r.t. $P$ is defined as
	\begin{equation*}
	\AHD\left( x; P \right) = \inf \left\{ P(H^-) \colon H \in \mathcal H_0, x \in H \right\} ,	
	\end{equation*}
where $\mathcal H_0$ denotes the set of hyperplanes $H \in \mathcal H$ in $\R^d$ such that $0 \in H$.
\end{definition}

It is natural to consider the collection $\mathcal H_0$ in the definition of $\AHD$, as $\mathcal H_0 \cap \Sph$ is the collection of all closed hemispheres of $\Sph$. Therefore, it is not surprising that also for spherical convex bodies, concepts similar to floating bodies have been investigated. Recall that for $K \subset \Sph$, $K$ is said to be spherically convex if the radial extension of $K$, given by
	\[	\rad K = \left\{ \lambda x \colon x \in K, \lambda \geq 0 \right\},	\]
is a convex set in $\R^d$. A closed spherically convex subset of $\Sph$ such that the interior of $\rad K$ is nonempty is called a spherical convex body. Analogues of floating bodies and convex floating bodies for spherical convex bodies were studied by \citet{Besau_Werner2016}.

\begin{definition}
For a spherical convex body $K \subset \Sph$ and $P\in\Prob[\Sph]$ uniformly distributed on $K$ take $\delta\geq 0$. The spherical convex floating body of $K$ is defined as
	\begin{equation*}	
	\bigcap_{H \in \mathcal H_0 \colon P(H^-) \leq \delta} H^+.	
	\end{equation*}
\end{definition}

Just as in Section~\ref{section:floating bodies for measures} it is possible to define floating bodies, and convex floating bodies also for general probability measures on $\Sph$, and it is easy to see that the spherical convex floating body coincides with the central regions of the angular halfspace depth for uniform distributions on spherical convex bodies. Some results in the spirit of those discussed in Section~\ref{section:convex floating bodies} can be obtained also for spherical convex floating bodies \citep{Besau_Werner2016}. In another paper, \citet{Besau_Werner2017} provide extensions of those results also to certain Riemannian manifolds. Research in this direction in the statistics of data depth is still only in its beginnings \citep{Fraiman_etal2018}.

\subsection{Depth for infinite-dimensional data}

In statistics, since the work of \citet{Liu_Singh1997} and \citet{Fraiman_Muniz2001}, considerable attention has focused also on devising depth functions applicable to data from high-dimensional, and infinite-dimensional (functional) spaces. Direct applications of the halfspace depth are known to be inadequate \citep{Dutta_etal2011}, but many other depth functions that are suited for functional data can be found in the literature \citep{Cuevas_Fraiman2009, Lopez_Romo2009, Mosler2013, Claeskens_etal2014, Chakraborty_Chaudhuri2014B, Narisetty_Nair2016, Nieto_Battey2016, Gijbels_Nagy2017}. In geometry, some advances that appear to be related are the floating functions \citep{Li_etal2018} considered in Section~\ref{FLME} above. Solid connections between these two areas of research appear to be uncharted. 

\subsection{Centroid body and simplicial volume depth}

Apart from the halfspace depth, the simplical depth, and the Mahalanobis depth mentioned above, there exists an abundance of other depth functions defined in $\R^d$ in statistics. A comprehensive survey on some of those is \citep{Zuo_Serfling2000}, where, based on the ideas of \citet{Oja1983}, also the following depth function can be found.

\begin{definition}
Let $X \sim P \in \Prob$ be such that $\Var X = \Sigma$ is a positive definite matrix and $x\in\R^d$. The \emph{simplicial volume depth} (or \emph{Oja depth}) of $x$ w.r.t. $P$ is defined as
	\begin{equation}	\label{SVD definition}
	\SVD(x;P) = \left( 1 + \E \frac{\vol{\left[x,X_1, \dots, X_d\right]}}{\sqrt{\det \Sigma}} \right)^{-1},	
	\end{equation}
where $X_1, \dots, X_d \sim P$ are independent.
\end{definition} 

The factor $\sqrt{\det \Sigma}$ ensures the affine invariance of $\SVD$. Similarly as the Mahalanobis depth $\MahD$, also $\SVD$ is not defined for all $P\in\Prob$, but only for distributions with finite second moments, and positive definite variance matrices.

For (a uniform distribution on) a compact (possibly non-convex) set $K \subset \R^d$ with $\vol{K}>0$, a concept closely related to $x \mapsto \vol{\left[x,X_1, \dots, X_d\right]}$, that is central in \eqref{SVD definition}, is that of the centroid body of $K$. The centroid body of $K$ is a convex body $Z \in \CB$ defined via its support function \eqref{support function}
	\[	h_{Z}(u) = \frac{1}{\vol{K}}\int_K \left\vert \langle x, u \rangle \right\vert \dd x.	\]
If $K$ is (centrally) symmetric around around the origin, $\partial Z$ is the locus of centroids of all intersections of halfspaces $H^- \in \mathcal H^-$ such that $0 \in H$ with $K$. As discussed in \citep[Section~9.1]{Gardner2006}, this body was defined by \citet{Petty1961}, but its earlier predecessors can be traced back to the work of \citet{Dupin1822}. The volume of the centroid body $Z$ of $K$ determines the simplicial volume depth $\SVD$ of $0\in\R^d$ with respect to the the uniform distribution on $K$. The next theorem can be found in \citet[Theorem~9.1.5]{Gardner2006}. Extensions not listed here can be found in \citep{Petty1961, Schneider_Weil1983}. For star bodies $K \subset \R^d$ a version of this theorem is given in \citep[Section~10.8]{Schneider2014}.

\begin{theorem}	\label{theorem:simplicial volume relations}
Let $X \sim P \in \Prob$ be uniformly distributed on a compact set $K \subset \R^d$ with $\vol{K}>0$. Denote $\Var X = \Sigma$. Let $Z_x$ be the centroid body of $K - x$. Then
	\[	\SVD(x;P) = \left( 1 + \frac{2^d}{\vol{K}^d} \frac{\vol{Z_x}}{\sqrt{\det \Sigma}} \right)^{-1}.	\]
\end{theorem}

Centroid bodies have been the subject of numerous studies in geometry and functional analysis. We only refer here to \citep[Section~10.8]{Schneider2014} and \citep[Section~5.1]{Brazitikos_etal2014} and the references therein for a comprehensive account of results that can be found in the literature on centroid bodies and their extensions.

\begin{problem}
Is it possible to extend Theorem~\ref{theorem:simplicial volume relations} also to more general probability measures?
\end{problem}

\subsection*{Acknowledgment}
Stanislav Nagy is supported by the grant 18-00522Y of the Czech Science Foundation, and by the PRIMUS/17/SCI/3 project of Charles University.
Elisabeth Werner is partially supported by NSF grant DMS-1504701.

\bibliographystyle{apalike}
\def\cprime{$'$} \def\polhk#1{\setbox0=\hbox{#1}{\ooalign{\hidewidth
  \lower1.5ex\hbox{`}\hidewidth\crcr\unhbox0}}}

\end{document}